\documentclass[reqno,oneside]{amsart}
\usepackage[margin=3cm]{geometry}
\usepackage{amssymb,amsmath}
\usepackage{amsfonts}
\usepackage{amsbsy}
\usepackage{rotating}
\usepackage[T1]{fontenc}
\usepackage{lmodern}
\usepackage{bbm}
\usepackage{mathtools}
\mathtoolsset{showonlyrefs=true}
\usepackage{graphicx}
\usepackage[boxed,oldcommands,norelsize]{algorithm2e}
\usepackage{hyperref}
\usepackage{subcaption}
\usepackage[numbers,sort&compress]{natbib}
\usepackage{multirow}
\usepackage{hhline}
\usepackage{colortbl}
\newcolumntype{L}{>{\centering\arraybackslash}m{0.46\textwidth}}
\usepackage{arydshln}
\usepackage[nolist]{acronym}

\usepackage{comment}
\usepackage{stmaryrd,xparse}
\usepackage{nicefrac}

\newtheorem{lemma}{Lemma}[section]

\newtheorem{definition}{Definition}[section]
\newtheorem{corollary}{Corollary}[section]

\newcommand{\M}{\boldsymbol{\mathsf{M}}}
\newcommand{\sM}{\tilde{\boldsymbol{\mathsf{M}}}}
\newcommand{\oM}{\overline{\boldsymbol{\mathsf{M}}}}

\newcommand{\B}{\boldsymbol{\mathsf{B}}}
\newcommand{\K}{\boldsymbol{\mathsf{K}}}
\newcommand{\sK}{\tilde{\boldsymbol{\mathsf{K}}}}
\newcommand{\oK}{\overline{\boldsymbol{\mathsf{K}}}}
\newcommand{\R}{\boldsymbol{\mathsf{R}}}
\newcommand{\J}{\boldsymbol{\mathsf{J}}}
\newcommand{\0}{\boldsymbol{\mathsf{0}}}
\newcommand{\G}{\boldsymbol{\mathsf{G}}} 
\newcommand{\A}{\mathsf{A}}

\newcommand{\U}{\boldsymbol{\mathsf{U}}}

\newcommand{\rr}{\boldsymbol{r}}

\newcommand{\normal}{\boldsymbol{n}}
\newcommand{\x}{\boldsymbol{x}}

\newcommand{\gradient}{\boldsymbol\nabla}
\newcommand{\cc}{\boldsymbol{c}}

\newcommand{\nn}{\boldsymbol{n}}

\newcommand{\half}{\frac{1}{2}}
\def\ij{{ij}} 
 
\newcommand{\ji}{{ji}}

\newcommand{\der}{{\rm d}}
\newcommand{\derd}{\delta}
\newcommand{\dtnp}{\Delta t_{n+1}}
\newcommand{\RR}{\mathbb{R}}

\newcommand{\id}{\mathbb{I}}
\newcommand{\np}{{n+1}}
\newcommand{\nk}{{k,n+1}}

\newcommand{\vel}{\mathbf{v}}

\newcommand{\conv}{\boldsymbol{f}} 

\newcommand{\force}{g}
\newcommand{\bforce}{\boldsymbol{\force}}

\newcommand{\intenergy}{\imath}
\newcommand{\roesound}{\rev{a_{ij}}}
\newcommand{\detectorComps}{\rev{\mathcal{J}}}

\newcommand{\lone}{L^1}
\newcommand{\ltwo}{L^2}

\newcommand{\inflowboundary}{\Gamma_{\rm in}}
\newcommand{\outflowboundary}{\Gamma_{\rm out}}

\newcommand{\domain}{\Omega}

\newcommand{\mesh}{\mathcal{T}_h}

\newcommand{\element}[1][]{K_{#1}}

\newcommand{\nnodes}{N}
\newcommand{\fespace}{V_h}
\newcommand{\bfespace}{\boldsymbol{V}_h}

\newcommand{\btestspace}{\boldsymbol{V}_{h0}}

\newcommand{\nodes}{\mathcal{N}_h}
\newcommand{\allnodes}{\tilde{\mathcal{N}}_h}

\newcommand{\neighborhood}[1][]{\mathcal{N}_h(\domain_{#1})} 

\newcommand{\symneigh}[1][]{\mathcal{N}^{\sym}_h(\domain_{#1})}

\newcommand{\hangingnodes}{\mathcal{N}^{hg}_h}
\newcommand{\constraint}{\mathcal{M}}
\newcommand{\constraining}{\overline{\mathcal{M}}}
\newcommand{\constmass}{\boldsymbol{\mathcal{M}}}
\newcommand{\constflux}{\boldsymbol{\mathcal{K}}}
\newcommand{\constforce}{\boldsymbol{\mathcal{G}}}

\newcommand{\contunk}{\boldsymbol{u}}
\newcommand{\unk}{u_h}

\newcommand{\bunk}{\contunk_h}

\newcommand{\berror}{\boldsymbol{e}_h}
\newcommand{\estimator}{\eta}
\newcommand{\lestimator}{\tilde{\eta}}

\newcommand{\uboundary}{\overline{u}}
\newcommand{\test}{v_h}
\newcommand{\btest}{\boldsymbol{v}_h}
\newcommand{\shapef}[1][]{\varphi_{#1}}

\newcommand{\smax}{\max{}_{\sigma_h}}

\newcommand{\detector}[1][]{\alpha_{#1}}
\newcommand{\vdetector}[1][]{\boldsymbol{\alpha}_{#1}}

\newcommand{\graphl}{\ell}

\newcommand{\jump}[1]{\left\llbracket #1 \right\rrbracket}
\newcommand{\mean}[1]{%
  \sbox0{%
    \mathsurround=0pt 
    $\left\{\vphantom{#1}\right.\kern-\nulldelimiterspace$%
  }%
  \sbox2{\{}%
  \ifdim\ht0=\ht2
    \{\kern-.625\wd2 \{#1\}\kern-.625\wd2 \}%
  \else
    \left\{\kern-.7\wd0\left\{#1\right\}\kern-.7\wd0\right\}%
  \fi
}

\newcommand{\smthlimit}[1]{Z\left(#1\right)}
\newcommand{\absn}[2][\varepsilon_h]{\left\vert #2 \right\vert_{1,#1}}
\newcommand{\absd}[2][\varepsilon_h]{\left\vert #2 \right\vert_{2,#1}}
\newcommand{\smthdetector}[1][]{ %
	\ifx\empty#1\empty%
	\alpha_{\varepsilon_h} %
	\else %
	\alpha_{\varepsilon_h,#1} %
	\fi}
\newcommand{\vsmthdetector}[1][]{ %
	\ifx\empty#1\empty%
	\boldsymbol{\alpha}_{\varepsilon_h} %
	\else %
	\boldsymbol{\alpha}_{\varepsilon_h,#1} %
	\fi}
\newcommand{\smthspacedetector}[1][]{ %
	\ifx\empty#1\empty%
	\alpha_{\varepsilon_h}^s %
	\else %
	\alpha_{\varepsilon_h,#1}^s %
	\fi}
\newcommand{\smthtimedetector}[1][]{ %
	\ifx\empty#1\empty%
	\alpha_{\varepsilon_h}^t %
	\else %
	\alpha_{\varepsilon_h,#1}^t %
	\fi}
\newcommand{\smthstdetector}[1][]{ %
	\ifx\empty#1\empty%
	\alpha_{\varepsilon_h}^{st} %
	\else %
	\alpha_{\varepsilon_h,#1}^{st} %
	\fi}

\newcommand{\sdetectorAprox}[1][]{ %
 \ifx\empty#1\empty%
    \tilde{\alpha}_{\varepsilon_h} %
 \else %
    \tilde{\alpha}_{\varepsilon_h,#1} %
 \fi}
\newcommand{\artdif}{\nu}
\newcommand{\smthartdif}{\tilde{\nu}}
\newcommand{\sym}{{\rm sym}}
\newcommand{\lambdamax}[1][]{\lambda_{#1}^{\max}}

\newcommand{\rev}[1]{#1}
\newcommand{\revv}[1]{#1}

\graphicspath{{figures/}}
\begin{document}

\begin{acronym}
	\acro{fe}[FE]{finite element}
	\acro{fvm}[FVM]{finite volume methods}
	\acro{dg}[dG]{discontinuous Galerkin}
	\acro{cg}[cG]{continuous Galerkin}
	\acro{dof}[DOF]{degree of freedom}
	\acro{ssp}[SSP]{strong stability preserving}
	\acro{rk}[RK]{Runge Kutta}
	\acro{be}[BE]{Backward Euler}
	\acro{dmp}[DMP]{discrete maximum principle}
	\acro{mp}[MP]{maximum principle}
	\acro{afc}[AFC]{algebraic flux correction}
	\acro{fct}[FCT]{flux corrected transport}
	\acro{led}[LED]{local extremum diminishing}
	\acro{dled}[DLED]{discrete local extremum diminishing}
	\acro{amr}[AMR]{adaptive mesh refinement}
	\acro{pde}[PDE]{partial differential equation}
\end{acronym}

\title[Monotonicity-preserving FE schemes with AMR for hyperbolic problems]{Monotonicity-preserving finite element schemes with adaptive mesh refinement for hyperbolic problems}

\author{ Jes\'us Bonilla$^{1,2}$ \and Santiago Badia$^{2,3}$ }
\thanks{
    $^1$ Universitat Polit\`ecnica de Catalunya, Jordi Girona 1-3, Edifici C1, 08034
	Barcelona, Spain. \\
	\indent$^2$ Centre Internacional de M\`etodes Num\`erics en Enginyeria (CIMNE), Esteve Terradas 5, 08860 Castelldefels, Spain. \\ 
	\indent$^3$ School of Mathematics, Monash University, Clayton, Victoria, 3800, Australia. \\ 
}

\renewcommand{\thefootnote}{\arabic{footnote}}

\maketitle

\begin{abstract} 
This work is focused on the extension and assessment of the monotonicity-preserving  scheme in \cite{badia_monotonicity-preserving_2017} and the local bounds preserving  scheme in \cite{Badia2019c} to hierarchical octree \ac{amr}. Whereas the former can readily be used on this kind of meshes, the latter requires some modifications.  A key question that we want to answer in this work is whether to move from a linear to a nonlinear stabilization mechanism pays the price when combined with shock-adapted meshes. Whereas nonlinear (or shock-capturing) stabilization leads to improved accuracy compared to linear schemes, it also negatively hinders nonlinear convergence, increasing computational cost. We compare linear and nonlinear schemes in terms of the required computational time versus accuracy for several steady benchmark problems. Numerical results indicate that, in general, nonlinear schemes can be cost-effective for sufficiently refined meshes. Besides, it is also observed that it is better to refine further around shocks rather than use sharper shock capturing terms, which usually yield stiffer nonlinear problems.  In addition, a new refinement criterion has been proposed. The proposed criterion is based on the graph Laplacian used in the definition of the stabilization method. Numerical results show that this shock detector performs better than the well-known Kelly estimator for problems with shocks or discontinuities.
\end{abstract}


\tableofcontents

\pagestyle{myheadings}
\thispagestyle{plain}

\section{Introduction}

Natural phenomena can develop shock waves in different scenarios. A classical example is the shock wave generated by an object traveling faster than sound. The numerical modeling of problems with shocks is still a challenge, especially when the admissible physical solution has some physical constraints, e.g., positivity or non-negativity, that must be preserved at the discrete level to have well-posedness; E.g., the fluid density and temperature are positive quantities in a compressible flow.

Several numerical schemes have been proposed so far to approximate this kind of problems by combining \ac{fvm} or \ac{dg} \acp{fe} for space discretization with explicit time integrators (see \cite{Leveque2002,cockburn_rungekutta_2001,Toro2009,Feistauer2003}). Explicit time integrators are only stable under a Courant-Friedrichs-Levy (CFL) restriction over the time step size, which implies to capture all time scales. Thus, explicit methods are not suitable for problems in which the smallest time scales are not of interest. For instance, the fastest time scales at a confined plasma in a nuclear fusion reactor are not of engineering interest whereas explicit time integration is unaffordable in practical simulations \cite{kritz_fusion_2009}.

Implicit monotonicity-preserving (or at least positivity-preserving) methods are still scarce. As proved by Godunov \cite{Godunov1959}, linear monotonicity-preserving schemes can be at most first-order accurate. For scalar problems (and under some mesh restrictions), Burman and Ern \cite{burman_nonlinear_2002}, Barrenechea and co-workers \cite{Barrenechea2016,barrenechea_edge-based_2016-1}, Kuzmin and co-workers \cite{Kuzmin2017,Kuzmin2012,Lohmann2017a}, and Badia and Hierro \cite{badia_monotonicity-preserving_2014,badia_discrete_2015} have proposed nonlinear schemes that preserve monotonicity and can presumably attain higher order accuracy.\footnote{In this work, schemes with nonlinear stabilization are also referred to as high-order and linear stabilization schemes as low or first-order.} However, these properties \revv{come} at the cost of solving a very stiff nonlinear problem \cite{kuzmin_linearity-preserving_2012}. The authors \cite{badia_differentiable_2017,badia_monotonicity-preserving_2017} have proposed differentiable schemes that improve the nonlinear convergence behavior of previous methods.

For hyperbolic systems of equations, numerical methods are less well developed. For explicit time integration, Guermond and Popov \cite{Guermond2015} have recently proposed a \ac{cg} \ac{fe} scheme that preserves positivity of density and energy under certain CFL-like condition. \rev{More recently, Kuzmin \cite{Kuzmin2020} has extended the previous work to monolithic convex limiting. This allows one to use implicit time integration while preserving positivity. Another approach is \ac{fct} \cite{Lohner1987,Kuzmin2012}. The schemes in \cite{Mabuza2018,Mabuza2019} combine the diffusion operators in \ac{fct} with novel shock-detection techniques to obtain a nonlinear monolithic scheme. Those methods have been shown experimentally to be robust, but lack of a theoretical analysis.} Besides, this strategy also yields very stiff nonlinear problems. Differentiable schemes for compressible flows have been proposed in \cite{Badia2019c} to alleviate (but not eliminate) this problem.

Shocks are non-smooth and localized and thus suitable for \ac{amr} \cite{demkowicz_computing_2006,Verfurth2013}. \ac{amr} allows one to increase the mesh resolution only in the vicinity of shocks or discontinuities. In brief, the \ac{amr} process can be divided into two main ingredients. On the one hand, to estimate the error at each element. On the other hand, to decide which elements need to be refined or coarsened. This iterative process  provides a mesh locally adapted to the features of the problem at hand. As a result, it is a nonlinear approximation scheme which tries to minimize the error for a target computational cost. \rev{In some situations, the optimal order of convergence can be achieved even for solutions with limited regularity using \ac{amr} \cite{demkowicz_computing_2006}, whereas convergence is limited by the regularity of the solution for uniform mesh refinements.} 

In this context, a key question is whether it is computationally more effective to consider a nonlinear high-order scheme (with the nonlinear convergence issues) or a cheaper linear (first-order) scheme in a more refined mesh. The motivation of this work is to shed light on this issue. First, we adapt the schemes developed in \cite{Badia2019c,badia_monotonicity-preserving_2017} to hierarchical octree \ac{amr} \cite{TiankaiTu2005,Badia2019d}. Next, we propose a refinement criterium that relies on information already present in the stabilization technique; nonlinear stabilization methods include a shock detector to activate the artificial diffusion only close to discontinuities. We propose to use a modification of the shock detector in \cite{badia_monotonicity-preserving_2017} to drive the \ac{amr} process.

This paper is structured as follows. First, we introduce the problem, its discretization, and monotonicity properties for scalar problems and hyperbolic systems in Sect.\ \ref{sec.preliminaries}. Then, the stabilization techniques are introduced in Sect.\ \ref{sec.stabilization}. Sect.\ \ref{sec.adaptivity} is devoted to the \ac{amr} strategy. We introduce the nonlinear solvers in Sect.\ \ref{sec.solvers}. Finally, we show numerical experiments in Sect.\ \ref{sec.numerical-exp} and draw some conclusions in Sect.\ \ref{sec.conclusions}.

\section{Preliminaries}\label{sec.preliminaries}

\subsection{Continuous problem}

Let us consider an open bounded and connected domain, $\domain\subset\RR^d$, where $d$ is the number of spatial dimensions. Let $\partial\domain$ be the Lipschitz continuous boundary of $\domain$. The conservative form of a first order hyperbolic problem reads
\begin{equation}\label{eq.continuous-problem}
\arraycolsep=1.4pt\def\arraystretch{1.1}
\left\{\begin{array}{rcll}
\partial_t \contunk + \gradient\cdot\conv(\contunk) & = & 
\bforce, & \text{in } \domain\times(0,T], \\
u^\beta(x,t) & = & \bar{u}^\beta(x,t), & \text{on } 
\inflowboundary^\beta\times(0,T],\; \beta = 1,...,m,\\
\contunk(x,0) & = & \contunk_0(x), &  x\in\domain,
\end{array}
\right.
\end{equation}
where $\contunk = \{ u^\beta \}_{\beta=1}^m$ are $m\geq 1$ conserved variables, $\conv$ is the physical flux, $\bar{u}^\beta(x,t)$ are the boundary values for the $\beta$th-component of $\contunk$, $\contunk_0(x)$ are the initial conditions, and $\bforce(x,t)$ is a function defining the body forces. Note that the flux, $\conv : \RR^m \to \RR^{m\times d}$, is composed of $\conv = \{\conv_i\}_{i=1}^d$, where $\conv_i : \RR^m \to \RR^m$ is the flux in the $i$th spatial direction. We denote by $\conv^\prime:\RR^m \to \RR^{m\times m\times d}$ the flux Jacobian. Let $\normal\in\RR^d$ be any direction vector. Since the system is hyperbolic, the flux Jacobian in any direction is diagonalizable and has only real eigenvalues, i.e., $\conv^\prime(\contunk)\cdot\normal=\sum_{i=1}^{d} \conv_i^\prime(\contunk)n_i $ is diagonalizable with real eigenvalues $\{\lambda_\beta\}_{\beta=1}^m$. These eigenvalues might have different multiplicities and different signs. Hence, for a given direction $\normal$, each characteristic variable might be convected forward (along $\normal$) or backwards (along $-\normal$). Therefore, it is convenient to define inflow and outflow boundaries for each component. The inflow boundary for component $\beta$ is defined as $\inflowboundary^\beta\doteq\{\x\in\partial\domain \ :\ \lambda_\beta(\conv^\prime(\contunk)\cdot \normal_{{\partial\domain}}) \leq 0\}$, where $\normal_{\partial\domain}$ is the unit outward normal to the boundary and $\lambda_\beta$ is the $\beta$th-eigenvalue of the flux Jacobian. We define the outflow boundary as $\outflowboundary^\beta\doteq\partial\domain\backslash\inflowboundary^\beta$. We refer the reader to \cite{Feistauer2003,Toro2009,Gurris2009} for a detailed discussion on boundary conditions for hyperbolic problems.
In the present study, we will also consider the steady counterpart of \eqref{eq.continuous-problem}, which is obtained by dropping the time derivative term and the initial conditions.

In this work, we work with both scalar convection equations and Euler equations. Taking $m=1$ and $\conv(u) \doteq \vel u$ with $\vel$ a divergence-free convection field, we recover the well known scalar transport problem. On the other hand, Euler equations for ideal gases are recovered by defining $m=d+2$ and
\begin{equation}
\contunk\doteq \left(\begin{array}{c}
\rho \\ \boldsymbol{m} \\ \rho E
\end{array}\right),\quad \conv\doteq \left(\begin{array}{c}
\boldsymbol{m} \\ \boldsymbol{m}\otimes\vel + p \id \\ \vel(\rho E+p)
\end{array}\right),\, \  \hbox{and} \ \ \bforce\doteq\left(\begin{array}{c}
0 \\ \boldsymbol{b} \\ \boldsymbol{b}\cdot\vel + r
\end{array}\right),
\end{equation}
where $\rho$ is the density, $E$ is the total energy, $p$ is the pressure, $\boldsymbol{m} = \{{m}_1,\ldots,{m}_d\}$, where $m_i = \rho v_i$, is the momentum, $\vel = \{v_1,\ldots,v_d\}$ is the velocity, $\boldsymbol{b} = \{{b}_1,\ldots,{b}_d\}$ are the body forces, $r$ is an energy source term per unit mass, and $\mathbb{I}$ is the identity matrix of dimension $d\times d$. In addition, the system is equipped with the ideal gas equation of state $p=(\gamma-1) \rho \intenergy$, where $\intenergy=E-\half\|\vel\|^2$ is the internal energy and $\gamma$ is the adiabatic index. 

\subsection{Discretization}

The discretization used in this work is able to adapt its local size to the features of the problem at hand. In particular, it is a hierarchically refined octree--based hexahedral mesh \cite{TiankaiTu2005}. This type of discretizations are constructed hierarchically. At every step of the refinement process, marked cells are refined into four (eight) cells in 2D (3D). The adaptation of the mesh to the problem at hand is achieved by only marking for refining a targeted amount of cells. This results in a mesh with different refinement levels at different regions. \emph{Hanging} nodes appear at the interface between cells at different refinement levels. These are nodes that only belong to the cells at a higher refinement level (see Fig.~\ref{fig.hanging-node}). In our case, the meshes used are \emph{2:1 balanced}. This restriction implies that there can only be a difference of one refinement level between neighboring cells. This restriction is a trade-off between implementation complexity and performance gain that has been adopted by many \ac{amr} codes \cite{TiankaiTu2005}.

\begin{figure}[h]
	\centering
	\includegraphics[width=0.3\textwidth]{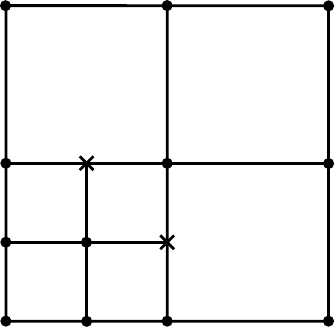}
	\caption{Example of a mesh with \emph{hanging} nodes.}
	\label{fig.hanging-node}
\end{figure}

Hanging nodes need to be treated carefully in the case of working with conforming \ac{fe} discretizations. Otherwise, associating a regular \ac{dof} to a hanging node may lead to discontinuities in the approximated solution. To preserve continuity of the \ac{fe} space, \emph{hanging \acp{dof}} values are not included in the assembled system of equations but obtained by interpolating the values of the neighboring \emph{regular} \acp{dof}. For more details in the definitions of these \revv{conformity} constraints we refer the reader to \cite{badia_fempar:_2017,Badia2019,Badia2019d}.

Let $\mesh$ be a hierarchical octree-based partition of $\domain$. Consider a Lagrangian (nodal) \ac{fe} space on top of this mesh. The set of all nodes in the \ac{fe} space is represented with $\allnodes$. For every node $i \in \allnodes$, $\x_i$ stands for the node coordinates. We can split $\allnodes$ into two subsets, namely the set of hanging nodes $\hangingnodes$ and the set of conforming nodes  $\nodes\doteq\allnodes\backslash\hangingnodes$. We denote by $\nnodes \doteq \mathrm{card}(\nodes)$ the total number of conforming nodes. The set of nodes belonging to a particular element $\element \in \mesh$ is defined as $\nodes(\element)\doteq\{i\in\nodes : \x_i\in \element\}$. Moreover, $\domain_i$ is the macroelement composed by the union of elements that contain node $i$, i.e., $\domain_i\doteq \bigcup_{\element\in\mesh, \; \x_i\in \element}  \element$. To simplify the discussion below, we abuse notation and use $i$ for both the node and its associated index.

We restrict the present work to first order \acp{fe} and define the \ac{fe} space as follows. We define $\bfespace\doteq \big \{\btest\in(\mathcal{C}^0(\domain))^m : \btest|_K\in({Q}_1(\element))^m \forall K\in\mesh\big\}$, where $Q_1(\element)$ is the space of polynomials of partial degree less than or equal to one. Furthermore, we define the space $\btestspace\doteq\{\btest\in\bfespace\ :\ \btest(\x)=0\ \forall\x\in\inflowboundary\}$.
The functions $\btest\in\bfespace$ can be constructed as a linear combination of the basis $\{\shapef[i]\}_{i\in\allnodes}$ and nodal values $\boldsymbol{v}_i$, where $\shapef[i]$ is the shape function associated to the node $i$. Hence, $\btest=\sum_{i\in\allnodes}\shapef[i]\boldsymbol{v}_i$.

We use standard notation for Sobolev spaces. The $\ltwo(\omega)$ scalar product is denoted by $(\cdot,\cdot)_\omega$ for $\omega\subset\domain$. However, we omit the subscript for $\omega\equiv\domain$. The $\ltwo$ norm is denoted by $\Vert\cdot\Vert$. 

The method of lines is applied in combination with the \ac{fe} spaces described above for the spatial discretization. We approximate the solution $\contunk\approx\bunk=\sum_{i\in\allnodes} \shapef[i]\contunk_i$. \revv{In addition, we make use of the group--FEM approximation \cite{Fletcher1983}. Hence, fluxes are discretized in the same \ac{fe} space as the unknown, i.e.} $\conv \approx \conv_h =\sum_{i\in\allnodes} \shapef[i]\conv(\contunk_i)$. For simplicity in the exposition, we use the \ac{be} scheme for the time discretization; higher order time discretizations can be achieved using \ac{ssp}--\ac{rk} methods (see \cite{Gottlieb2001}). In the latter case, a CFL-like condition must be satisfied to enjoy the monotonicity properties in Sect. \ref{sec.stability} (see \cite{kuzmin_algebraic_2005,kuzmin_flux_2002}). 

The semi-discrete Galerkin \ac{fe} approximation of the weak form of \eqref{eq.continuous-problem} reads: 
find $\bunk\in\bfespace$ such that $\unk^\beta=\bar{u}^\beta_h$ on $\inflowboundary^\beta$, $\bunk=\boldsymbol{u}_{0h}$ at $t=0$, and
\begin{equation} \label{eq.discrete-problem}
(\partial_t\bunk,\btest) - (\bunk,\conv_h^{\prime}
(\bunk):\gradient\btest) + (\bunk,
\normal_{\outflowboundary}\cdot\conv_h^\prime(\bunk)\btest)_{\outflowboundary} 
= (\bforce,\btest),\; \mbox{for all } \btest\in\btestspace,
\end{equation}
where $\uboundary^\beta_h$ and $\contunk_{0h}$ are admissible \ac{fe} approximations of the boundary and initial conditions $\uboundary^\beta$ and $\contunk_0$. 
In this context, we consider admissible any approximation that satisfies the maximum principle, i.e., it does not introduce new extrema. \rev{Notice that boundary conditions are strongly imposed. For transonic, or complex problems, this strategy might lead to convergence issues. However, in the present paper we use this strategy for the sake of simplicity. As previously mentioned, we refer the reader to \cite{Feistauer2003,Toro2009,Gurris2009} for a detailed discussion on boundary conditions for hyperbolic problems.} Note that the double contraction is applied as $\conv_h^{\prime}
(\bunk):\gradient\btest = \sum_{k,\gamma} \conv^{\prime}_{h}
(\bunk)_k^{\beta\gamma}\ \boldsymbol{v}_{h\,\gamma,k} $. 

As commented above, we need to apply constraints to all hanging \acp{dof} to keep conformity. The value of the hanging \ac{dof} needs to be equal to the value of the interpolation of the unknown at the neighboring coarser elements. That is, given $i\in\hangingnodes$ and \revv{one of} its neighboring (coarse) \ac{fe} $\element\in\mesh$, $\boldsymbol{v}_i=\sum_{j\in\nodes(\element)} \shapef[j](\x_i)\boldsymbol{v}_j$. In general, we will represent this constraint with $\boldsymbol{v}_i=\sum_{j\in\constraining(i)} C_{ij}\boldsymbol{v}_j$, \rev{where $\constraining(i)$ is the set of \acp{dof} constraining \ac{dof} $i$, and $C_{ij}\doteq \shapef[j](\x_i)$. It is also useful to define $\constraint(i)$, which is the set of \acp{dof} constraint by $i$}. For details of the implementation of this kind of constraints see \cite{badia_fempar:_2017,Badia2019,Badia2019d}. 

Finally, to obtain the fully discrete problem, we consider a partition of the time domain $(0,T]$ into $n^{ts}$ sub-intervals \revv{$(t^n,t^{n+1}]$ of length $\Delta t_{n+1}$}. Then, at every time step $n = 0,\ldots,n^{ts}-1$, the discrete problem consists in solving
\begin{equation}\label{eq.discrete-matrix-problem}
\M \derd_t \U^{n+1} + \K\U^{n+1} = \G,
\end{equation}
where $\U^{n+1}\doteq [\contunk_1^{n+1},...,\contunk_{\nnodes}^{n+1}]^T$ is the vector of nodal values at time $t^{n+1}$, $\derd_t (\U) \doteq \dtnp^{-1} (\U^{n+1} - \U^n)$, and $\dtnp \doteq ( t^{n+1} - t^n)$. The $m\times m$-matrices relating nodes $i,j\in\nodes$ are given by
\begin{align}
\M_{\ij}^{\beta\gamma} &\doteq(\shapef[j],\shapef[i]) \delta_{\beta\gamma} + \constmass_{\ij}^{\beta\gamma}, \\
\K_{\ij}^{\beta\gamma} &\doteq -(\shapef[j]\delta_{\beta\xi},\conv^{\prime}_k (\contunk_j^\np)^{\xi\eta}\cdot\partial_k \shapef[i]\delta_{\eta\gamma}) + (\shapef[j]\delta_{\beta\xi},n_k\cdot\conv^\prime_k(\contunk_j^\np)^{\xi\eta} \shapef[i]\delta_{\eta\gamma})_{\outflowboundary} + \constflux_{\ij}^{\beta\gamma}, \\
\G_i^\beta&\doteq(\bforce^\beta,\shapef[i]) + \constforce_{i}^{\beta},
\end{align}
where Einstein summation applies, $\beta,\gamma,\xi,\eta \in  \{1,\ldots,m\}$ are the component indices, $\delta_{\beta\gamma}$ is the Kronecker delta, and $\constmass$, $\constflux$, and $\constforce$ are the terms arising from applying the \revv{conformity} constraints in the mass, flux and body forces terms.

\subsection{Stability properties}\label{sec.stability}

Finally, let us review some concepts required for discussing the stabilization method used in the subsequent sections. Let us recall some definitions used for scalar problems.
\begin{definition}[Local discrete extremum]\label{def.extremum}
	The function $\test\in\fespace$ has a local discrete minimum (resp. maximum) on $i\in\nodes$ if $u_i\leq u_j$ (resp. $u_i\geq u_j$) ${\forall j\in\neighborhood[i]}$.
\end{definition}
\begin{definition}[Local \ac{dmp}]\label{def.local-dmp}
	A solution $\unk\in\fespace$ satisfies the local discrete maximum principle if for every $i\in\nodes$
	\begin{equation}
	\min_{j\in\neighborhood[i]\backslash\{i\}} u_j \leq u_i \leq \max_{j\in\neighborhood[i]\backslash\{i\}} u_j.
	\end{equation}
\end{definition}
\begin{definition}[LED]\label{def.led}
	A scheme is local extremum diminishing if, for every $u_i$ that is a local discrete maximum (resp. minimum),
	\begin{equation}
          \frac{\der u_i}{\der t} \leq 0, \qquad \qquad \left(\hbox{resp.} \ \frac{\der u_i}{\der t} \geq 0 \right),
	\end{equation}
	is satisfied.
\end{definition}

One possible strategy to satisfy the above properties consists in designing a scheme that yields a positive diagonal mass matrix and a stiffness matrix that satisfies 
\begin{equation}\label{eq.m-matrix}
\sum_j \A_{\ij} = 0, \quad\text{and}\quad \A_{\ij} \leq  0,  \ i\neq j. 
\end{equation}
In this case, it is possible to rewrite the system as
\begin{equation}\label{eq.general-led-problem}
m_i \derd_t u^\np + \sum_{j\in\neighborhood[i]\backslash\{i\}} \A_{\ij} (u_j^\np - u_i^\np) = 0, \quad \forall \, i\in\nodes.
\end{equation}
\revv{As shown in \cite{kuzmin_flux_2002}, such a scheme is \ac{led}. Moreover, for the steady scheme obtained by dropping the transient term, property \eqref{eq.m-matrix} leads to solutions that satisfy the local \ac{dmp} \cite{codina_discontinuity-capturing_1993}.}

Stability properties for hyperbolic systems can be based on the extension of the above to hyperbolic systems in characteristic variables. In this direction, we define local bounds preserving schemes as follows.
\begin{definition}\label{def.system-led}
	The discrete scheme
	\begin{equation}
	\sum_{j}\M_\ij\derd_t \contunk_i^\np + \sum_{j\neq i} \A_{\ij} (\contunk_j^\np - \contunk_i^\np ) = \boldsymbol{0}
	\end{equation}
	is said to be local bounds preserving if $\M$ is diagonal with positive entries (i.e., $\M_{\ij}=m_i\delta_{\ij} I_{m\times m}$), $\A_{\ij}$ has non-positive eigenvalues for every $j\neq i$, and $\sum_j\A_{\ij}=\boldsymbol{0}$.
\end{definition}
Unfortunately, to the best of our knowledge, satisfying this definition does not ensure positivity of density, internal energy, or non-decreasing entropy. In any case, numerical schemes based on this definition have shown good numerical behavior \cite{Kuzmin2003,Kuzmin2005,Lohmann2016,Mabuza2018}. 

Several stabilization strategies have been defined based on the previous ideas. One of the most simple strategies consists in adding a scalar artificial diffusion term proportional to the spectral radius of $\A_{\ij}$ \cite{Lohner2004,Kuzmin2005}. This strategy is usually called Rusanov artificial diffusion, since the scheme results in the Rusanov Riemann solver for linear \acp{fe} in one dimension \cite{Kuzmin2005,Toro2009}. Without any special treatment, the resulting scheme is only first order accurate. The key for recovering high-order convergence is to modulate the action of the artificial diffusion term, and restrict its action to the vicinity of discontinuities. In the present work, our stabilization term for systems of equations is based on Rusanov artificial diffusion and a differentiable shock detector recently developed for scalar problems \cite{badia_monotonicity-preserving_2017,Bonilla2019a}. 

Finally, it is also important to define the concept of linearity preservation.

\begin{definition}\label{def.lp}
Given $\bunk\in\bfespace$ \rev{and $\detectorComps$ the set of conservative variables that are used to detect inadmissible values of $\unk$}, a stabilization scheme is said to be linearity preserving if \revv{the stabilization vanishes} at any region such that $\unk^\beta \in P_1(\Omega)\, \forall\, \beta\in \detectorComps$.
\end{definition}

\section{Nonlinear stabilization}\label{sec.stabilization}

In this section, we describe the additional terms used for the \revv{stabilization of problem \eqref{eq.discrete-matrix-problem}}. In particular, we use the stabilization terms defined in \cite{badia_monotonicity-preserving_2017} for the scalar problem and \cite{Badia2019c} for Euler:
\begin{equation}\label{eq.stab-term}
B_h(\boldsymbol{w}_h;\bunk,\btest)\doteq 
\left\{\begin{array}{ll}
\sum_{i\in\nodes}\sum_{j\in\neighborhood[i]} \artdif_\ij(w_h) v_i u_j \graphl(i,j), & \text{for } m=1, \\
\sum_{\element[e]\in\mesh}\sum_{i,j \in \nodes(\element[e])} \artdif^e_{\ij}(\boldsymbol{w}_h)\graphl(i,j) \boldsymbol{v}_i \cdot  I_{m\times m} \boldsymbol{u}_j, & \text{for } m>1,
\end{array}\right.
\end{equation}
for any $\boldsymbol{w}_h, \ \bunk, \ \btest\in\bfespace$. Here, $\graphl$ is the graph-Laplacian operator defined as $\graphl(i,j)\doteq 2\delta_\ij-1$ (see \cite{guermond_second-order_2014,badia_monotonicity-preserving_2017}). In the case of a scalar problem, $m=1$, the nodal artificial diffusion $\artdif_\ij(w_h)$ is defined as
\begin{align}\label{eq.artificial-diffusion-scalar}
\artdif_\ij(w_h)&\doteq\max\{\detector[i](w_h)\K_\ij,0,\detector[j](w_h)\K_\ji\} \quad \text{for}\quad j\in\neighborhood[i]\backslash\{i\}, \\
\artdif_{ii}(w_h)&\doteq\displaystyle\sum_{j\in\neighborhood[i]\backslash\{i\}} \artdif_\ij(w_h).
\end{align}
We denote by $\detector(w_h)$ the scalar shock detector used for computing the artificial diffusion parameter. In the case of the Euler equations,
the element-wise artificial diffusion $\artdif^e_{\ij}(\boldsymbol{w}_h)$ is defined as
\begin{equation}\label{eq.artificial-diffusion}
\begin{aligned}
\artdif^e_{\ij}(\boldsymbol{w}_h) &  \doteq\max\left(\vdetector[i](\boldsymbol{w}_h)\lambdamax[ij],\vdetector[j](\boldsymbol{w}_h)\lambdamax[ji]\right) \\ 
& + \sum_{k\in\constraint(i)} C_{ki} \max\left(\vdetector[k](\boldsymbol{w}_h)\lambdamax[kj],\vdetector[j](\boldsymbol{w}_h)\lambdamax[jk]\right) \\ 
& + \sum_{k\in\constraint(j)} C_{kj} \max\left(\vdetector[i](\boldsymbol{w}_h)\lambdamax[ik],\vdetector[k](\boldsymbol{w}_h)\lambdamax[ki]\right) \\
& + \sum_{k\in\constraint(i)\cap\constraint(j)} C_{ki} C_{kj} \vdetector[k](\boldsymbol{w}_h)\lambdamax[kk], \quad  
\text{for } j\in\neighborhood[i]\backslash\{i\}, \\
\artdif^e_{ii}(\boldsymbol{w}_h)&\doteq\displaystyle \sum_{j\in\neighborhood[i]\backslash\{i\}} \artdif^e_{\ij}(\boldsymbol{w}_h),\quad
\end{aligned}
\end{equation}
where $\lambdamax[ij]$ is the spectral radius of the elemental convection matrix relating nodes $i,j\in\nodes$, i.e.,  $\rho\left(\conv^\prime(\contunk_{\ij})\cdot(\gradient\shapef[j],\shapef[i])_{\element[e]}\right)$, where $\contunk_\ij$ is the Roe average between $\contunk_i$ and $\contunk_j$ (see \eqref{eq.roe-average}). \revv{Notice that for $\artdif^e_\ij$ and $k\in\hangingnodes$, $\lambdamax[ik]$ is actually the spectral radius of $\rho\left(\conv^\prime(\contunk_{\ij})\cdot(\gradient\shapef[k],\shapef[i])_{\element[e]}\right)$. This is a direct consequence of using a \ac{fe} approximation of the fluxes.}
\rev{As previously discussed, this artificial diffusion term is based on Rusanov scalar diffusion \cite{Kuzmin2012}. It is important to mention that, given any direction vector $\nn$, the eigenvalues $\conv^\prime(\contunk_\ij)\cdot\nn$ of the Roe-averaged flux Jacobian in that direction can be easily computed as
\begin{equation}\label{eq.eigenvalues}
\lambda_{1,..,d}=\vel_{\ij} \cdot \nn,\quad \lambda_{d+1}=\vel_{\ij}\cdot \nn - \roesound\|\nn\|, \quad \lambda_{d+2}=\vel_{\ij}\cdot \nn + \roesound\|\nn\|,
\end{equation}
where the velocity $\vel_\ij$ and sound speed $\roesound$ are computed using the Roe-averaged values
\begin{align}\label{eq.roe-average}
\roesound = \sqrt{(\gamma -1)\left(H_{\ij} - \frac{\|\vel_{\ij}\|^2}{2}\right)}, & 
\quad \vel_{\ij} = \frac{\vel_i\sqrt{\rho_j}+\vel_j\sqrt{\rho_i}}{\sqrt{\rho_i}+\sqrt{\rho_j}}, \\ 
H_{\ij} = \frac{H_i\sqrt{\rho_i}+H_j\sqrt{\rho_j}}{\sqrt{\rho_i}+\sqrt{\rho_j}},& \quad H_i=\rev{E_i +} \frac{p_i}{\rho_i},\quad\text{and}\quad \rho_{\ij} = \sqrt{\rho_i\rho_j}.
\end{align}
This property greatly simplifies the computation of $\rho\left(\conv^\prime(\contunk_{\ij})\cdot(\gradient\shapef[j],\shapef[i])_{\element[e]}\right)$.}

We denote by $\vdetector[i](\boldsymbol{w}_h)$ the shock detector used for modulating the action of the artificial diffusion term. The idea behind the definition of this detector is to minimize the amount of artificial diffusion introduced while stabilizing any oscillatory behavior. \revv{We want to ensure that the resulting stabilized scheme satisfies Def. \ref{def.system-led} in regions where the local \ac{dmp} is violated (see Def. \ref{def.local-dmp}) by a given set of components}. $\vdetector[i](\boldsymbol{w}_h)$ must be a positive real number that takes value 1 when $\bunk(\x_i)$ is an inadmissible value of $\bunk$, and smaller than 1 otherwise. To this end, we define 
\begin{equation}\label{eq.system-detector}
\vdetector[i](\bunk)\doteq \textstyle\max \{\detector[i](\unk^\beta)\}_{\beta\in \detectorComps}, \quad\forall i\in\nodes
\end{equation}
where $\detectorComps$ is the set of components that are used to detect inadmissible values of $\bunk$, e.g. density and total energy in the case of Euler equations. For simplicity, we restrict ourselves to the components of $\bunk$. However, derived quantities such as the pressure or internal energy can be also used. For scalar equations, since the stabilization is defined for the assembled system, the shock detector $\vdetector[i]$ only needs to be defined for $i\in\nodes$. However, for the elemental definition used for Euler equations, it is also required for $i\in\hangingnodes$. In that case, we use the maximum of its constraining nodes, i.e., 
\begin{equation}
  \vdetector[k](\bunk)\doteq \max_{j\in\constraining(k)} \vdetector[j](\bunk) \quad \rev{\hbox{for } k\in\hangingnodes}. 
\end{equation}
Let us recall some useful notation from \cite{badia_monotonicity-preserving_2017} to introduce the scalar shock detector $\detector[i](w_h)$. Let $\rr_{\ij} = \x_j - \x_i$ be the vector pointing from node $\x_i$ to $\x_j$ with $i,j\in\nodes$ and $\hat{\rr}_{\ij} \doteq \frac{\rr_{\ij}}{|\rr_{\ij}|}$. Recall that the set of points $\x_j$ for $j\in\neighborhood[i]\backslash\{i\}$ define the macroelement $\domain_i$ around node $\x_i$. Let $\x_{\ij}^\sym$ be the point at the intersection between $\partial\domain_i$ and the line that passes through $\x_i$ and $\x_j$ that is not $\x_j$ (see Fig. \ref{fig.usym}). The set of all $\x_{\ij}^\sym$ for all $j\in\neighborhood[i]\backslash\{i\}$ is represented with $\symneigh[i]$. We define $\rr_{\ij}^\sym \doteq \x_{\ij}^\sym - \x_i$. 
We define \revv{$\contunk_\ij^\sym$} as the value of $\bunk$ at $\x_{\ij}^\sym$, i.e., $\bunk(\x_{\ij}^\sym)$.
\begin{figure}[h]
	\centering
	\includegraphics[width=0.25\textwidth]{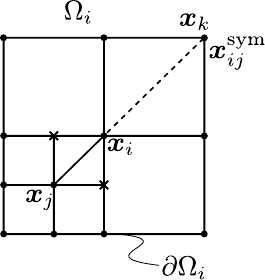}
	\caption{$u^\sym$ drawing}
	\label{fig.usym}
\end{figure}

Both $\contunk_{\ij}^\sym$ and $\x_{\ij}^\sym$ are only required to construct a linearity preserving shock detector. Let us define the jump and the mean of a linear approximation of component $\beta$ of the unknown gradient at node $\x_i$ in direction $\rr_{\ij}$ as
\begin{equation}\label{eq.jump}
\jump{\gradient \unk^\beta}_{\ij} \doteq \frac{u_j^\beta - u_i^\beta}{|\rr_{\ij}|} + \frac{\revv{u_\ij^{\sym,\beta}} - u^\beta_i}{|\rr_{\ij}^\sym|}, 
\end{equation}
\begin{equation}\label{eq.mean}
\mean{|\gradient \unk^\beta\cdot \hat{\rr}_{\ij}|}_{\ij} \doteq \frac{1}{2}
\left(\frac{|u_j^\beta-u_i^\beta|}{|\rr_{\ij}|}+\frac{|\revv{u_\ij^{\sym,\beta}}-u_i^\beta|}{|\rr_{\ij}^\sym|}\right).
\end{equation}
In the present work, for each component in $\detectorComps$, we use the same shock detector developed in \cite{badia_monotonicity-preserving_2017}. Let us recall its definition
\begin{equation}\label{eq.detector}
\detector[i](\unk^\beta)\doteq\left\{\begin{array}{cc}
\left[\dfrac{\left|\sum_{j\in\neighborhood[i]} \jump{\gradient\unk^\beta}_{\ij}\right|}{\sum_{j\in\neighborhood[i]}2\mean{\left|\gradient\unk^\beta\cdot\hat{\rr}_{\ij}\right|}_{\ij}}\right]^q & \text{if}\quad \sum_{j\in\neighborhood[i]} \mean{\left|\gradient\unk^\beta\cdot\hat{\rr}_{\ij}\right|}_{\ij} \neq 0 \\
0 & \text{otherwise}
\end{array}
\right. ,
\end{equation}
\revv{where $q>1$ is a parameter that minimizes the amount of artificial diffusion introduced.}
We know from \cite[Lm. 3.1]{badia_monotonicity-preserving_2017} that \eqref{eq.detector} gets values between 0 and 1, and it is only equal to one if $\unk^\beta(\x_i)$ is a local discrete extremum.
Since the linear approximations of the unknown gradients are exact for $\unk^\beta\in\mathcal{P}_1$, the shock detector vanishes when the solution is linear. Thus, it is also linearly preserving for every component in $\detectorComps$. This result follows directly from \cite[Th. 4.5]{badia_monotonicity-preserving_2017}.

\revv{The shock detector is properly defined for any interior node. However, $u_\ij^\beta$ might not seem easy to compute for nodes at the boundary. For nodes at Dirichlet boundaries, $\x_i\in\inflowboundary$, the shock detector is set to 0, since the value of the unknown is fixed. For weak Dirichlet boundary conditions, we refer the reader to \cite{badia_differentiable_2017}, where appropriate definitions are developed such that monotonicity and linearity are preserved. In the case of $\x_i\in\outflowboundary$, $\x_\ij^\sym$ might lay on top of node $\x_i$ for some directions. Hence, the second fraction in \eqref{eq.jump} and \eqref{eq.mean} becomes undefined. In these cases, this term is dropped and we use
\begin{equation}
\jump{\gradient \unk^\beta}_{\ij} \doteq \frac{u_j^\beta - u_i^\beta}{|\rr_{\ij}|} \quad \text{and}\quad
\mean{|\gradient \unk^\beta\cdot \hat{\rr}_{\ij}|}_{\ij} \doteq \frac{1}{2}
\left(\frac{|u_j^\beta-u_i^\beta|}{|\rr_{\ij}|}\right).
\end{equation}
This definition still ensures that the shock detector takes value 1 for extreme values at $\x_i$ (see \cite[Lm. 3.1]{badia_monotonicity-preserving_2017}). However, linearity preservation in the direction normal to $\outflowboundary$ is lost unless the unknown is constant in that direction.}

The final stabilized problem in matrix form reads as follows. Find $\bunk\in\bfespace$ such that \revv{$\unk^\beta=\uboundary_h^\beta$ on $\inflowboundary^\beta$, $\bunk=\boldsymbol{u}_{0h}$ at $t=0$}, and
\begin{equation}\label{eq.stabilized-matrix-problem}
\oM(\bunk^\np)\derd_t\U^\np+\oK_{\ij}(\bunk^\np)\U^\np = \G
\end{equation}
for $n=1,...,n^{ts}$, where 
\begin{equation}
\oM_{\ij}(\bunk^\np)\doteq \left[1-\max\left(\vdetector[i],\vdetector[j]\right)\right]\M_\ij
+ \delta_{\ij}\rev{\sum_{k\in\nodes}\max(\detector[i],\detector[k])\M_{ik}},
\end{equation}
\revv{$\M^L_i = \sum_j \M_\ij$}, and $\oK_{\ij}(\bunk^\np) \doteq \K_{\ij} + \B_{\ij}$, 
where \revv{$\B_{\ij}$ is the stabilization matrix that takes the form
\begin{equation}
\B_{\ij}(\boldsymbol{w}_h)\doteq \left\{\begin{array}{ll}
\artdif_\ij(w_h)\ \graphl(i,j), & \text{for } m=1, \\
\sum_{\element[e]\in\mesh}\artdif^e_{\ij}(\boldsymbol{w}_h)\ \graphl(i,j)  I_{m\times m} , & \text{for } m>1,
\end{array}\right..
\end{equation}
}

Let us show that adapted non-conforming meshes do not jeopardize any of the stability properties defined in Sect.\ \ref{sec.stability} and proved for conforming meshes in \cite{badia_differentiable_2017,Badia2019}.

\begin{corollary}[\ac{dmp}]\label{thm.genDMP}
	The solution of the discrete problem \eqref{eq.stabilized-matrix-problem} with $m=1$ and using the shock detector \eqref{eq.detector} satisfies the local \ac{dmp} in Def. \ref{def.local-dmp} if $\force = 0$ and, for every control point $i\in\nodes$ such that $u_i$ is a local discrete extremum, it holds:
	\begin{equation}\label{dmp-cond}
	\oK_\ij(\unk) \leq 0,  \,\forall \, j \in \neighborhood[i] \backslash\{i\},\qquad \sum_{j\in\neighborhood[i]} \oK_\ij(\unk)  = 0.
	\end{equation}
	Moreover, the resulting scheme is linearity-preserving as defined in Def. \ref{def.lp}, i.e., $\B_\ij(\unk)=0$ for $\unk\in\mathcal{P}_1(\domain_i)$.
\end{corollary}
\begin{proof}
	The stabilization scheme for scalar problems is defined on the assembled system. Hence, the modifications introduced in the assembly procedure do not affect the reasoning in the proof of \cite[Thm. 5.2]{badia_differentiable_2017}.
\end{proof}

\begin{lemma}[Local bounds preservation]\label{thm-discreteLED}
	Consider $\bunk\in\bfespace$ with component $\beta$ in the set of tracked variables $\detectorComps$. The stabilized problem \eqref{eq.stabilized-matrix-problem}, \revv{with $\G=0$,} is local bounds preserving as defined in Def. \ref{def.system-led} at any region where $\unk^\beta$ has extreme values.
\end{lemma}
\begin{proof}
	If component $\beta\in \detectorComps$ of $\bunk$ has an extremum at $\x_i$, we know from \cite[Lm. 3.1]{badia_monotonicity-preserving_2017} that $\detector[i](\unk^\beta)=1$. Moreover, it can be checked from \eqref{eq.system-detector} that $\vdetector[i](\bunk)=1$. In this case, $\oM_{\ij}(\bunk) = \delta_{\ij}\sum_j\M_\ij$. Hence, $\oM_{\ij}(\bunk) = 0$ for $j\neq i$ and $\oM_{ii}(\bunk) = m_i$. Therefore, we can rewrite the system as follows
	\begin{align}
	m_i\partial_t \contunk_i + \sum_{j\in\neighborhood[i]\backslash\{i\}} \oK_{\ij}(\contunk_{\ij})(\contunk_j - \contunk_i) &= \\
	m_i\partial_t \contunk_i + \sum_{j\in\neighborhood[i]\backslash\{i\}}\sum_{\element[e]\in\mesh} \Big[(\conv^\prime(\contunk_{\ij}))\cdot \Big ( &
	(\gradient\shapef[j],\shapef[i])_{\element[e]} \\ 
&+\sum_{k\in\constraint(j)} C_{kj} (\gradient\shapef[k],\shapef[i])_{\element[e]} \\
&+\sum_{k\in\constraint(i)} C_{ki} (\gradient\shapef[j],\shapef[k])_{\element[e]} \\
&+\sum_{k\in\constraint(i)\cap\constraint(j)} C_{ki} C_{kj} (\gradient\shapef[k],\shapef[k])_{\element[e]} \Big) 
	-\artdif^e_{\ij} I_{m\times m} \Big](\contunk_j - \contunk_i) = \0. 
	\end{align}
	We need to prove that the eigenvalues of $\oK_{\ij}(\contunk_{\ij})$ are non-positive. To this end, let us show that the following inequality holds
	\begin{align}
	\sum_{\element[e]\in\mesh} & \Big ( \rho\left(\conv^\prime(\contunk_{\ij})\cdot(\gradient\shapef[j],\shapef[i])_{\element[e]}\right) \\
	&+\sum_{k\in\constraint(j)} C_{kj} \ \rho\left(\conv^\prime(\contunk_{\ij})\cdot(\gradient\shapef[k],\shapef[i])_{\element[e]}\right) \\
	&+\sum_{k\in\constraint(i)} C_{ki} \ \rho\left(\conv^\prime(\contunk_{\ij})\cdot(\gradient\shapef[j],\shapef[k])_{\element[e]}\right) \\
	&+\sum_{k\in\constraint(i)\cap\constraint(j)} C_{ki} C_{kj} \ \rho\left(\conv^\prime(\contunk_{\ij})\cdot(\gradient\shapef[k],\shapef[k])_{\element[e]}\right) 
	\Big )
	\geq\rho(\conv^\prime(\contunk_{\ij})\cdot(\gradient\shapef[j],\shapef[i])).
	\end{align}
	One can observe from \eqref{eq.eigenvalues}  that $\rho\left(\conv^\prime(\contunk_{\ij})\cdot(\gradient\shapef[j],\shapef[i])_{\element[e]}\right) = |\vel_{\ij}\cdot\cc_{\ij}^e| + \roesound\|\cc_{\ij}^e\|$, where $\cc_{\ij}^e=(\gradient\shapef[j],\shapef[i])_{\element[e]}$. We have that
	$$
	\cc_{\ij}=(\gradient\shapef[j],\shapef[i])=
	\sum_{\element[e]\in\mesh} \left(\cc_{\ij}^e +
	\sum_{k\in\constraint(j)}C_{kj}\cc_{ik}^e + 
	\sum_{k\in\constraint(i)}C_{ki}\cc_{kj}^e + 
	\sum_{k\in\constraint(i)\cap\constraint(j)}C_{ki}C_{kj}\cc_{kk}^e \right).
	$$ 
	 Thus,
	\begin{align}
	\sum_{\element[e]\in\mesh}	\Big( \left|\vel_{\ij}\cdot\cc_{\ij}^e\right| +
&	\sum_{k\in\constraint(j)}C_{kj} \left|\vel_{\ij}\cdot\cc_{ik}^e \right|\\ 
&+	\sum_{k\in\constraint(i)}C_{ki} \left|\vel_{\ij}\cdot\cc_{kj}^e \right|\\  
&+	\sum_{k\in\constraint(i)\cap\constraint(j)}C_{ki}C_{kj} \left|\vel_{\ij}\cdot\cc_{kk}^e \right| \Big)
	\geq \left|\vel_{\ij}\cdot\cc_{\ij}\right|,
	\end{align}
	and \rev{
	\begin{align}
	\sum_{\element[e]\in\mesh} \Big( \roesound\|\cc_{\ij}^e\| +
	  \sum_{k\in\constraint(j)}C_{kj} \roesound\|\cc_{ik}^e\| &+
	 \sum_{k\in\constraint(i)}C_{ki} \roesound\|\cc_{kj}^e\| \\  
	&+ \sum_{k\in\constraint(i)\cap\constraint(j)}C_{ki}C_{kj} \roesound \|\cc_{kk}^e\| \Big)
	\geq \roesound\|\cc_{\ij}\|.
	\end{align}}
	Therefore, $\sum_e\rho(\K_{\ij}^e(\contunk_{\ij}))\geq\rho(\K_{\ij}(\contunk_{\ij}))$. 
	Moreover, by definition (see \eqref{eq.artificial-diffusion}),
	\begin{align}
	\artdif_{\ij}^e\geq \rho\left(\conv^\prime(\contunk_{\ij})\cdot(\gradient\shapef[j],\shapef[i])_{\element[e]}\right) &+\sum_{k\in\constraint(j)} C_{kj} \ \rho\left(\conv^\prime(\contunk_{\ij})\cdot(\gradient\shapef[k],\shapef[i])_{\element[e]}\right) \\
	 &+\sum_{k\in\constraint(i)} C_{ki} \ \rho\left(\conv^\prime(\contunk_{\ij})\cdot(\gradient\shapef[j],\shapef[k])_{\element[e]}\right) \\
	 &+\sum_{k\in\constraint(i)\cap\constraint(j)} 
	 C_{ki} C_{kj} \ \rho\left(\conv^\prime(\contunk_{\ij})\cdot(\gradient\shapef[k],\shapef[k])_{\element[e]}\right) 
	 \quad \text{for} \quad j\neq i.
	\end{align}
	Furthermore, one can infer from \eqref{eq.stab-term} that $\rho(\B_{\ij}^e(\contunk_{\ij}))\geq\rho(\K_{\ij}^e(\contunk_{\ij}))$, \revv{where $\B_\ij^e(\contunk_{\ij})$ and $\K_{\ij}^e(\contunk_{\ij})$ are the elemental stabilization and stiffness matrices, respectively.}  Hence, $\rho(\B_{\ij}(\contunk_{\ij}))\geq\rho(\K_{\ij}(\contunk_{\ij}))$. Finally, 
	since $\oK_{\ij}= \K_{\ij} + \B_{\ij}$ and $\B_{\ij} = \sum_e\B^e_{\ij} = -\sum_e\artdif_{\ij}^e I_{m\times m}$ for all $j\neq i$. Then, the maximum eigenvalue of $\oK_{\ij}(\contunk_{\ij})$ is non-positive, which completes the proof.
\end{proof}

Notice that it is essential to apply properly the constraints at the flux \ac{fe} approximation, i.e., $\conv^\prime(\contunk_k) = \sum_{\rev{i\in\constraining(k)}} C_{ki} \conv^\prime(\contunk_i)$. Otherwise, it is not possible to formally prove local bound preservation. However, experimental results in the present work show that using $\conv^\prime(\contunk_k)$ does not affect the overall performance of the scheme.

\subsection{Differentiable stabilization}
In the case of steady, or implicit time integration, differentiability plays a role in the convergence behavior of the nonlinear solver. This is especially important if one wants to use Newton's method. The authors show in \cite{badia_differentiable_2017,badia_monotonicity-preserving_2017,Badia2019c} that nonlinear convergence can be improved after few modifications to make the scheme twice-differentiable. In this section, we introduce a set of regularizations applied to all non-differentiable functions present in the stabilized scheme introduced above. In order to regularize these functions, we follow the same strategy as in \cite{badia_differentiable_2017,badia_monotonicity-preserving_2017,Badia2019c}. Absolute values are replaced by
\begin{equation}
\absn{x} =  \sqrt{x^2 + \varepsilon_h},\qquad
\absd{x} =  \frac{x^2}{\sqrt{x^2 + \varepsilon_h}},
\end{equation}
\revv{where $\varepsilon_h$ is a small positive value.} Note that $\absd{x}\leq |x|\leq\absn{x}$. Next, we also use the smooth maximum function
\begin{equation}\label{eq:smax}
\smax(x,y) \doteq \frac{ \absn[\sigma_h]{x -y}}{2} +\frac{x +y}{2} \geq \max(x,y),
\end{equation}
\revv{where $\sigma_h$ is a small positive value.}
In addition, we need a smooth function to limit the value of any given quantity to one. To this end, we use
\begin{equation}\label{eq.smthlimit}
\smthlimit{x} \doteq \left\{\begin{array}{ll}
2x^4-5x^3+3x^2+x, & x<1,\\
1, & x\geq 1.
\end{array}\right.
\end{equation}

The set of twice-differentiable functions defined above allows us to redefine the stabilization term introduced in Sect. \ref{sec.stabilization}. In particular, we define 
\begin{equation}\label{eq.smth-stab-term}
\tilde{B}_h(\boldsymbol{w}_h;\bunk,\btest)\doteq 
\left\{\begin{array}{ll}
\sum_{i\in\nodes}\sum_{j\in\neighborhood[i]} \smthartdif_\ij(w_h) v_i u_j \graphl(i,j), & \text{for } m=1, \\
\sum_{\element[e]\in\mesh}\sum_{i,j \in \nodes(\element[e])} \smthartdif^e_{\ij}(\boldsymbol{w}_h)\graphl(i,j) \boldsymbol{v}_i \cdot  I_{m\times m} \boldsymbol{u}_j, & \text{for } m>1,
\end{array}\right.
\end{equation}
where
\begin{align}\label{eq.smth-artificial-diffusion-scalar}
\revv{\smthartdif_\ij(w_h)}&\revv{\doteq\smax\Big(\smax\big(\smthdetector[i](w_h)\K_\ij,\smthdetector[j](w_h)\K_\ji\big),0\Big) \quad \text{for}\quad j\in\neighborhood[i]\backslash\{i\}}, \\
\smthartdif_{ii}(w_h)&\doteq\displaystyle\sum_{j\in\neighborhood[i]\backslash\{i\}} \smthartdif_\ij(w_h),
\end{align}
and
\begin{equation}\label{eq.smth-artificial-diffusion}
\begin{aligned}
\smthartdif^e_{\ij}(\boldsymbol{w}_h) \doteq & \smax\left(\vsmthdetector[i](\boldsymbol{w}_h)\lambdamax[ij],\vsmthdetector[j](\boldsymbol{w}_h)\lambdamax[ji]\right) \\ 
& + \sum_{k\in\constraint(i)} C_{ki} \smax\left(\vsmthdetector[k](\boldsymbol{w}_h)\lambdamax[kj],\vsmthdetector[j](\boldsymbol{w}_h)\lambdamax[jk]\right) \\ 
& + \sum_{k\in\constraint(j)} C_{kj} \smax\left(\vsmthdetector[i](\boldsymbol{w}_h)\lambdamax[ik],\vsmthdetector[k](\boldsymbol{w}_h)\lambdamax[ki]\right) \\
& + \sum_{k\in\constraint(i)\cap\constraint(j)} C_{ki} C_{kj} \vsmthdetector[k](\boldsymbol{w}_h)\lambdamax[kk], \quad  
\text{for } j\in\neighborhood[i]\backslash\{i\}, \\
\smthartdif^e_{ii}(\boldsymbol{w}_h)&\doteq\displaystyle \sum_{j\in\neighborhood[i]\backslash\{i\}} \smthartdif^e_{\ij}(\boldsymbol{w}_h).\quad
\end{aligned}
\end{equation}
Let us note that $\lambdamax[\ij]$ needs to be regularized as $\lambdamax[\ij]=\rev{\absn{\vel_{\ij} \cdot\cc_{\ij}^e}} + c\|\cc_{\ij}^e\|$. 
The shock detector is also regularized as follows:
\begin{equation}\label{eq.smth-system-detector}
\vsmthdetector[i](\bunk)\doteq \textstyle\smax \{\smthdetector[i](\unk^\beta)\}_{\beta\in \detectorComps}.
\end{equation}
\revv{Notice that the regularized maximum has only been defined for two arguments. However, for $|\detectorComps|>2$, one can chain several times the regularized functions, i.e., $\smax(\smax(...))$.} 
In the case of the component shock detector we recall the definition in \cite[Eq. 18]{badia_monotonicity-preserving_2017}
\begin{equation}\label{eq.smth-detector}
\smthdetector[i](\unk^\beta)\doteq
\left[\smthlimit{\dfrac{\absn{\sum_{j\in\neighborhood[i]} \jump{\gradient\unk^\beta}_{\ij} }+\zeta_h}{\sum_{j\in\neighborhood[i]}2\mean{\absd{\gradient\unk^\beta\cdot\hat{\rr}_{\ij}}}_{\ij}+\zeta_h}}\right]^q,
\end{equation}
where $\zeta_h$ is a small value for preventing division by zero. Finally, the twice-differentiable stabilized scheme reads: find $\bunk\in\bfespace$ such that \revv{$\unk^\beta=\uboundary_h^\beta$ on $\inflowboundary^\beta$}, $\bunk=\boldsymbol{u}_{0h}$ at $t=0$, and
\begin{equation}\label{eq.smth-stabilized-matrix-problem}
\sM(\bunk^\np)\derd_t\U^\np+\sK_{\ij}(\bunk^\np)\U^\np = \G \quad\text{for}\ n=1,...,n^{ts},
\end{equation}
where 
\begin{align}
\sM_{\ij}(\bunk^\np) & \doteq \left[1-\smax\left(\vsmthdetector[i],\vsmthdetector[j]\right)\right]\M_\ij + \rev{\delta_{\ij}\sum_{k\in\nodes}\smax(\detector[i],\detector[k])\M_{ik}},\\
\sK_{\ij}(\bunk^\np) & \doteq \K_{\ij}(\bunk^\np) + \tilde{\B}_{\ij}(\bunk^\np),
\end{align}
and \revv{$\tilde\B_{\ij}$ is the regularized stabilization matrix that takes the form
\begin{equation}
	\tilde\B_{\ij}(\boldsymbol{w}_h)\doteq \left\{\begin{array}{ll}
	\smthartdif_\ij(w_h)\ \graphl(i,j), & \text{for } m=1, \\
	\sum_{\element[e]\in\mesh}\smthartdif^e_{\ij}(\boldsymbol{w}_h)\ \graphl(i,j)  I_{m\times m} , & \text{for } m>1.
	\end{array}\right.
\end{equation}
}

\begin{corollary}\label{cor.lbp}
	The scheme in \eqref{eq.discrete-matrix-problem} \revv{with $\G=0$ and} the differentiable stabilization in \eqref{eq.smth-stab-term} is local bounds preserving, as defined in Def. \ref{def.system-led}, at any region where $\unk^\beta$ has extreme values for every $\beta$ in $\detectorComps$. 
\end{corollary}
\begin{proof}	
	For an extreme value of $\unk^\beta$, since $\absd{x}\leq |x|\leq\absn{x}$ the quotient of \eqref{eq.smth-detector} is larger than one. Hence, by definition of $Z(x)$,  $\smthdetector[i]$ is equal to 1. At this point, it is easy to check that $\smthartdif_{\ij}^e\geq\artdif_{\ij}^e$ in virtue of the definition of $\smax$. Therefore, $\rho(\tilde{\B}^e_{\ij}(\bunk)) \geq \rho(\B^e_{\ij}(\bunk))$, completing the proof.
\end{proof}

Moreover, it is important to mention that the differentiable shock detector is weakly linearly-preserving as $\zeta_h$ tends to zero. This result follows directly from \cite{badia_monotonicity-preserving_2017}. In order to obtain a differentiable operator, we have added a set of regularizations that rely on different parameters, e.g., $\sigma_h,\ \varepsilon_h,\ \zeta_h$. Giving a proper scaling of these parameters is essential to recover theoretic convergence rates. In particular, we use the following relations
\begin{equation}\label{eq.scaling}
\sigma_h=\sigma|\lambdamax|^2  L^{2(d-3)} h^{4},\quad \varepsilon_h=\varepsilon L^{-4} h^2,\quad \zeta_h=L^{-1}\zeta,
\end{equation}
where \revv{$\sigma$, $\varepsilon$, and $\zeta$ are small positive parameters,} $d$ is the spatial dimension of the problem, $L$ is a characteristic length. \revv{The value $|\lambdamax|$ being used is the maximum on the whole domain of the Euclidean norm. In the case of a scalar problem, it is simply the maximum convection velocity, i.e. $\max_{\x\in\domain} |\vel(\x)|$.}

\section{Adaptive mesh refinement}\label{sec.adaptivity}

The motivation of an adaptive \ac{fe} method is to solve \eqref{eq.stabilized-matrix-problem} up to a certain tolerance (or resolution) using the \emph{minimum} number of \acp{dof}. To this end, the solution error ($\berror = \contunk-\bunk$) is estimated at each element. With this information at hand, it is possible to iteratively adapt the resolution of the mesh at certain regions. This process can be divided into two parts: estimating the error at every cell, and deciding which and how many cells need to be refined or coarsened. 
This procedure is performed iteratively until a desired tolerance is achieved or, alternatively, a number of elements is reached. 
In the present work, we start with a rather coarse mesh and perform the following steps till reaching a stopping criterion:
\begin{enumerate}
	\item Compute solution $\bunk$;
	\item Estimate the error $\berror$;
	\item Select all cells that need to be refinement or coarsened;
	\item Update the mesh, and project the solution to the new mesh.
\end{enumerate}
In some cases, the refinement might be driven by features of the solution instead of a classical error estimator. For instance, one may decide to refine the regions around discontinuities. In this scenario, one could use \revv{an} expression that does not estimate the error, but it allows to concentrate the elements around discontinuities.

\subsection{Error estimators}
One of the keys of \ac{amr} is the ability to provide a good estimation of the error. Several error estimators have been proposed to date \cite{Ainsworth1997,Johnson1995,Zienkiewicz1987,Kelly1983,Eriksson1995}. These can be classified, at least, in two main types.  Some authors \cite{Eriksson1995,Johnson1995,Suli1999,Nazarov2010,Nazarov2011} try to compute an upper bound of the error for every cell. Then, provided a user defined tolerance, one can decide to refine or coarsen each cell. However, an adjoint problem needs to be solved in order to compute this upper bound \cite{Burman2000,Johnson1995}. It is possible to approximate the error bounds without solving an adjoint problem only for simple cases, see \cite{Johnson1995}. Therefore, this kind of error estimators increases the computational cost substantially. Alternatively, one can simply determine the distribution of the error in the mesh and use this information to drive an adaptivity algorithm. In this scenario, some authors \cite{Moller2006,Bittl2013,Kuzmin2018,Ainsworth1997,Zienkiewicz1987,Zienkiewicz1992} drive the adaptivity process with the solution gradient. In this case, explicit expressions of the estimated error are possible, requiring less computational resources than the previous option.

In general, the adaptive procedure can be described as follows. Given a finite element solution $\bunk$, the error $\berror$ is approximated as $\berror\approx\gradient\contunk-\gradient\bunk$.  Then, the reconstruction is used as an approximation of the exact gradient. This strategy is based on \emph{superconvergence} of special recovery techniques (see \cite{Zienkiewicz1992} and refs. in \cite{Moller2006,Ainsworth1997}). Kuzmin and co-workers \cite{Moller2006,Bittl2013} follow \cite{Zienkiewicz1987} to reconstruct an approximation of the exact gradient. Kelly et al. \cite{Kelly1983} proposed a well-known estimator based on gradient recovery:
\begin{equation}
\estimator_K^2\doteq\frac{h_K}{24}\int_{\partial K} \jump{\frac{\partial \bunk}{\partial n}}^2 d\Gamma,
\end{equation}
where $\estimator_K$ is the estimated error at every element, $\element$. \revv{In this case, the jump $\jump{\cdot}$ does not correspond with the linear approximation of the gradient jump $\jump{\gradient(\cdot)}_\ij$ defined in \eqref{eq.jump}. Instead, it takes the classical definition, i.e.,  $\jump{u} = u^+\cdot n^+ + u^-\cdot n^-$.} The main advantages of this estimator are its simplicity and its low computational cost. For these reasons, this estimator is used in the present work.

It is worth mentioning that our problems of interest are characterized by exhibiting discontinuities, where the error concentrates. These regions are susceptible to develop instabilities, and thus, these are the regions in which the shock capturing is activated. Therefore, it is natural to use the shock capturing to drive the adaptivity procedure. We propose an \rev{indicator} based on the graph Laplacian $\ell(i,j)$ present in the stabilization term \eqref{eq.stab-term}. This way, we reuse available information and reduce the computational overhead associated with \rev{the selection of cells to be refined}. The indicator reads: 
\begin{equation}
  \rev{(\lestimator_{\element}^\beta)^2 \doteq h_K^{d-2} \sum_{i\in\nodes(\element)} \sum_{j\in\neighborhood[i]} (u_i^\beta - u_j^\beta)^2,}
\end{equation}
where $\beta\in \detectorComps$ is the index of the specific component analyzed. This \rev{expression} is expected to yield high values around shocks and low values in smooth regions.

\subsection{Refinement strategy}

After the error has been estimated for every element \rev{(or the indicator has been computed at every cell)}, one needs to decide which element needs to be refined and which one coarsened. If an upper bound of the error is computed, then one may use a given tolerance to make this decision. However, in the present case this is not available. A classical alternative is to refine/coarsen a fixed amount of elements at every iteration \cite{bangerth_algorithms_2012,Badia2019d}. In the present study, a 30\% of the elements with higher error estimates \rev{(or higher indicator values)} are refined whereas a 10\% of the elements with lower \rev{values} are coarsened. This percentages are arbitrary and other choices are valid. Notice that using this setting in two dimensions the number of elements is almost doubled at every iteration. We make use of the parallel $n$th element algorithm \cite{Tikhonova2005,Badia2019d} to efficiently determine the \rev{indicator} thresholds for refining or coarsening the elements.

\section{Nonlinear solver}\label{sec.solvers}

In this section, we describe the method used for solving the nonlinear system of equations arising from the scheme introduced above. In particular, we use a hybrid Picard--Newton approach in order to increase the robustness of the nonlinear solver. Moreover, we also make use of a line--search method to improve the nonlinear convergence.

We define the residual of the equation \eqref{eq.smth-stabilized-matrix-problem} at the $k$-th iteration as 
\begin{equation}\label{eq.residual}
\R(\bunk^\nk) \doteq \sM(\bunk^\nk)\derd_t\U^\nk+\sK_{\ij}(\bunk^\nk)\U^\nk - \G.
\end{equation}
Hence, the Jacobian is defined as
\begin{align}\label{eq.jacobian}
\J(\bunk^\nk)&\doteq \frac{\partial\R(\bunk^\nk)}{\partial \U^\nk} \\
&= \Delta t^{-1}_{t+1} \sM(\bunk^\nk)+\sK_{\ij}(\bunk^\nk) + \Delta t^{-1}_{t+1}\frac{\partial\sM(\bunk^\nk)}{\partial \U^\nk}\derd_t\U^\nk+\frac{\partial \sK_{\ij}(\bunk^\nk)}{\partial \U^\nk}\U^\nk.
\end{align}
Therefore, Newton method consists in solving $\J(\bunk^\nk) \Delta\U^{k+1,n+1} = - \R(\bunk^\nk)$. It is well known that Newton method can diverge if the initial guess of the solution $\bunk^{0,n+1}$ is not close enough to the solution. In order to improve robustness, we use a line--search method to update the solution at every time step. The new approximation is computed as $\U^{k+1,n+1} = \U^\nk + \lambda\Delta\U^{k+1,n+1}$, where $\lambda$ is obtained using a standard cubic backtracking algorithm. 

As introduced at the beginning of the section, we also use a hybrid approach combining Newton method with Picard linearization. Picard nonlinear iterator can be obtained removing the last two terms of \eqref{eq.jacobian}, i.e.,
\begin{equation}\label{eq.picard}
\left(\Delta t^{-1}_{t+1}\sM(\bunk^\nk)+\sK_\ij(\bunk^\nk)\right)\Delta\U^{k+1,n+1} = -\R(\bunk^\nk).
\end{equation}
Clearly, it is equivalent to 
\begin{equation}
\left(\Delta t^{-1}_{t+1}\sM(\bunk^\nk)+\sK_\ij(\bunk^\nk)\right)\U^{k+1,n+1} = \Delta t^{-1}_{t+1}\sM(\bunk^\nk)\U^{n} + \G.
\end{equation}
Moreover, we modify the left hand side terms in \eqref{eq.picard}; we use $\vdetector[i]=1$ for computing these terms while we use the value obtained from \eqref{eq.system-detector} for the residual. Using this strategy, the solution remains unaltered but the obtained approximations $\bunk^\nk$ for intermediate values of $k$ are more diffusive. Even though this modification slows the nonlinear convergence, it is essential at the first iterations. Otherwise, the robustness of the method might be jeopardized.

The resulting iterative nonlinear solver consists in the following steps. We iterate Picard method in \eqref{eq.picard}, with the modification described above, until the \rev{Euclidean} norm of the residual \rev{vector} is smaller than a given tolerance. In the present work, we use a tolerance of $10^{-2}$. Afterwards, Newton method with the exact Jacobian in \eqref{eq.jacobian} is used until the desired nonlinear convergence criteria is satisfied. We summarize the nonlinear solver introduced above in Alg.\ \ref{alg.scheme}.

\begin{algorithm}[h]
	\caption{Hybrid Picard--Newton method.}\label{alg.scheme}
	\KwIn{$\U^{0,n+1}$, ${\rm tol_1}$, ${\rm tol_2}$, $\varepsilon$}
	\KwOut{$\U^{k,n+1}$, $k$}
	$k=1$, $\varepsilon^1 = \varepsilon$ \\
	\While{$ \|\R(\U^\nk)\|/ \|\R(\U^{0,n+1})\| \geq \rm tol_1$ }{
		Compute $\vdetector[i](\U^\nk)$ using \eqref{eq.system-detector} \\
		Compute $\Delta \U^{k+1,n+1}$ using \eqref{eq.picard} \\
		Minimize $\|\R(\U^{k+1,n+1})\|$, where $\U^{k+1,n+1} = \lambda\Delta \U^{k+1,n+1} + \U^\nk$, with respect to $\lambda$ \\
		Set $\U^{k+1,n+1}=\lambda\Delta \U^{k+1,n+1} + \U^\nk$ \\
		Update $k= k+1$
	}
	
	\While{$ \|\R(\U^\nk)\|/ \|\R(\U^{0,n+1})\| \geq \rm tol_2$ }{
		Compute $\vdetector[i](\U^\nk)$ using \eqref{eq.system-detector} \\
		Solve $\J(\U^\nk)\Delta \U^{k+1,n+1}=-\R(\U^\nk)$ with $\J$ in \eqref{eq.jacobian} \\
		Minimize $\|\R(\U^{k+1,n+1})\|$, where $\U^{k+1,n+1} = \lambda\Delta \U^\nk + \U^\nk$, with respect to $\lambda$ \\
		Set $\U^{k+1,n+1}=\lambda\Delta \U^\nk + \U^\nk$ \\
		Update $k= k+1$
	}
\end{algorithm}

\section{Numerical results}\label{sec.numerical-exp}
In this section, we perform several numerical experiments to assess the numerical scheme introduced in the previous sections. First, we perform a convergence analysis to assess its implementation. Then, we use steady benchmark tests to analyze the effectiveness of the high--order scheme in the context of \ac{amr}. In particular, we compare the nonlinear scheme in \eqref{eq.smth-stabilized-matrix-problem} with its linear (first order) counterpart, i.e., using $\vsmthdetector[i](\bunk)\equiv 1$.

From previous experience \cite{badia_differentiable_2017,badia_monotonicity-preserving_2017,Bonilla2018a,Bonilla2019a}, we choose the following regularization parameters: $\sigma=10^{-2}$, $\varepsilon=10^{-4}$, and $\zeta = 10^{-10}$.
In addition, the density is discontinuous at all shocks and contacts for all Euler tests below. Therefore, we use $\detectorComps=\{1\}$ in \eqref{eq.system-detector}, i.e., the shock detector is based on the density behavior in all Euler tests below.

\subsection{Convergence}

First, the convergence to a discontinuous solution is analyzed. To this end, we solve two different problems. On the one hand, the following scalar problem is solved
\begin{equation}\label{eq.6.steady_transport}
\begin{array}{rlcl}
\gradient\cdot(\mathbf{v}u) &= 0 &\text{ in }& \domain=[0,1]\times[0,1], \\
u &= u_D &\text{ on } & \inflowboundary,
\end{array}
\end{equation}
where $\mathbf{v}(x,y) \doteq \left(\nicefrac{1}{2},\, \sin\nicefrac{-\pi}{3}\right)$, and inflow boundary conditions $u_D=1$ on $\{x=0\}\cap \{y>0.7\}$ and $y=1$, while $u_D=0$ at the rest of the inflow boundary. This problem has the following analytical solution
\begin{equation}
u(x,y) =  \left\{ \begin{array}{cl}
1 & \text{if}\quad y > 0.7 + 2x\sin\nicefrac{-\pi}{3}, \\
0 & \text{otherwise} .
\end{array}\right. 
\end{equation}

For the Euler equations, the problem is the well known compression corner test \cite{AndersonJr.1990,Kuzmin2012}, also known as oblique shock test \cite{shakib_new_1991,tezduyar_stabilization_2006}. This benchmark consists in a supersonic flow impinging to a wall at an angle. We use a $[0,1]^2$ domain with a $M=2$ flow at $10^{\circ}$ with respect to the wall. This leads to two flow regions separated by an oblique shock at 29.3$^{\circ}$, see Fig.\ \ref{fig.corner}. 
\begin{figure}[h]
	\centering
	\includegraphics[width=0.35\textwidth]{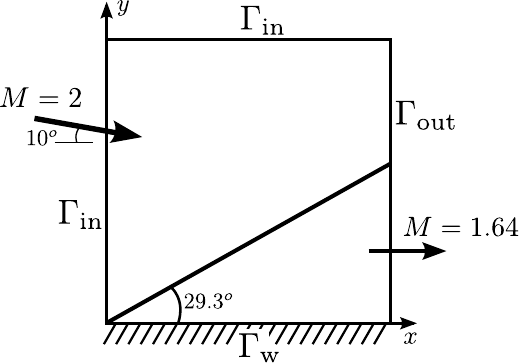}
	\caption{Compression corner scheme.}
	\label{fig.corner}
\end{figure}

\begin{figure}[h]
	\centering
	\begin{subfigure}[t]{0.42\textwidth}
		\includegraphics[width=\textwidth]{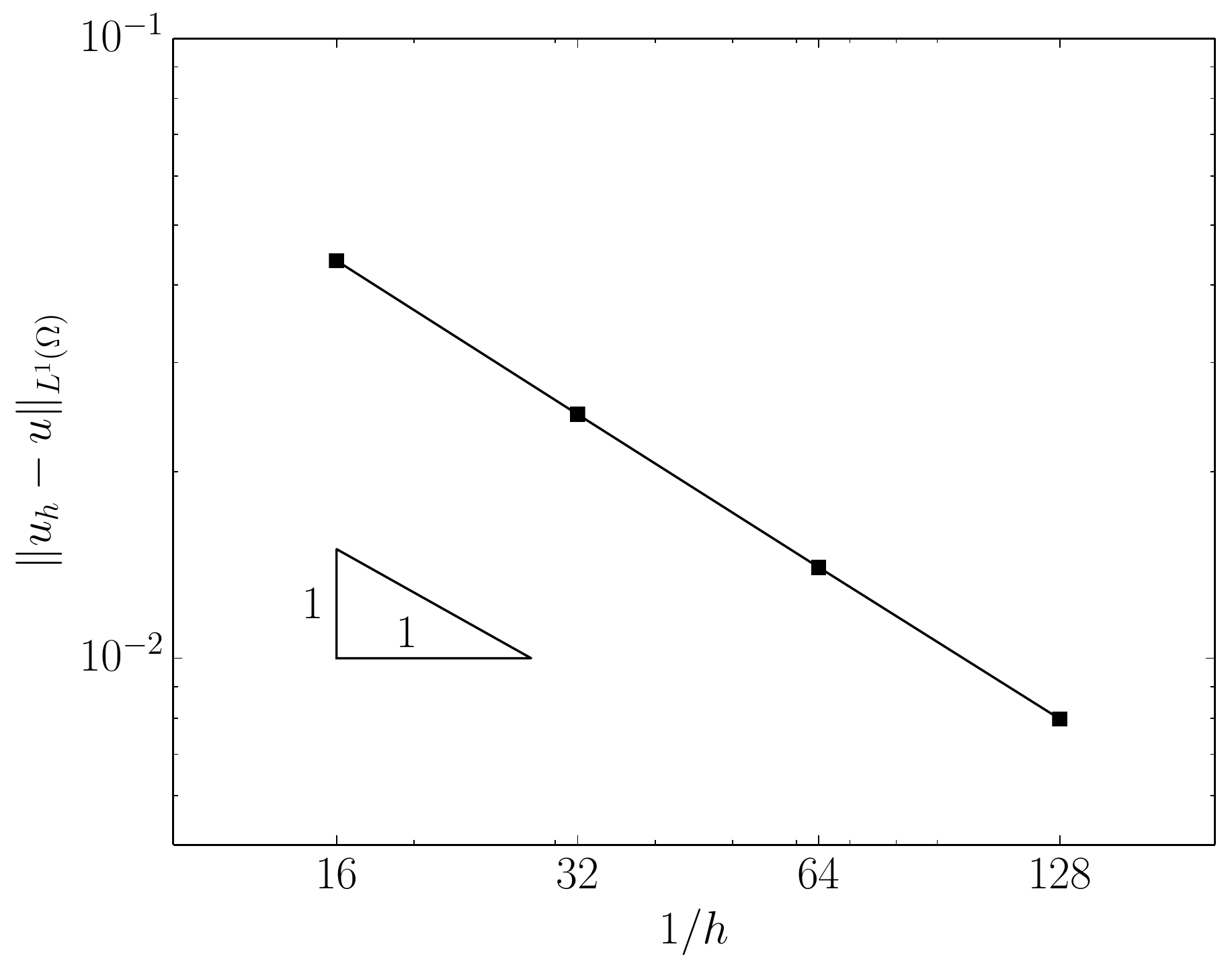}
		\caption{Scalar transport problem.}
	\end{subfigure}
	\begin{subfigure}[t]{0.44\textwidth}
		\includegraphics[width=\textwidth]{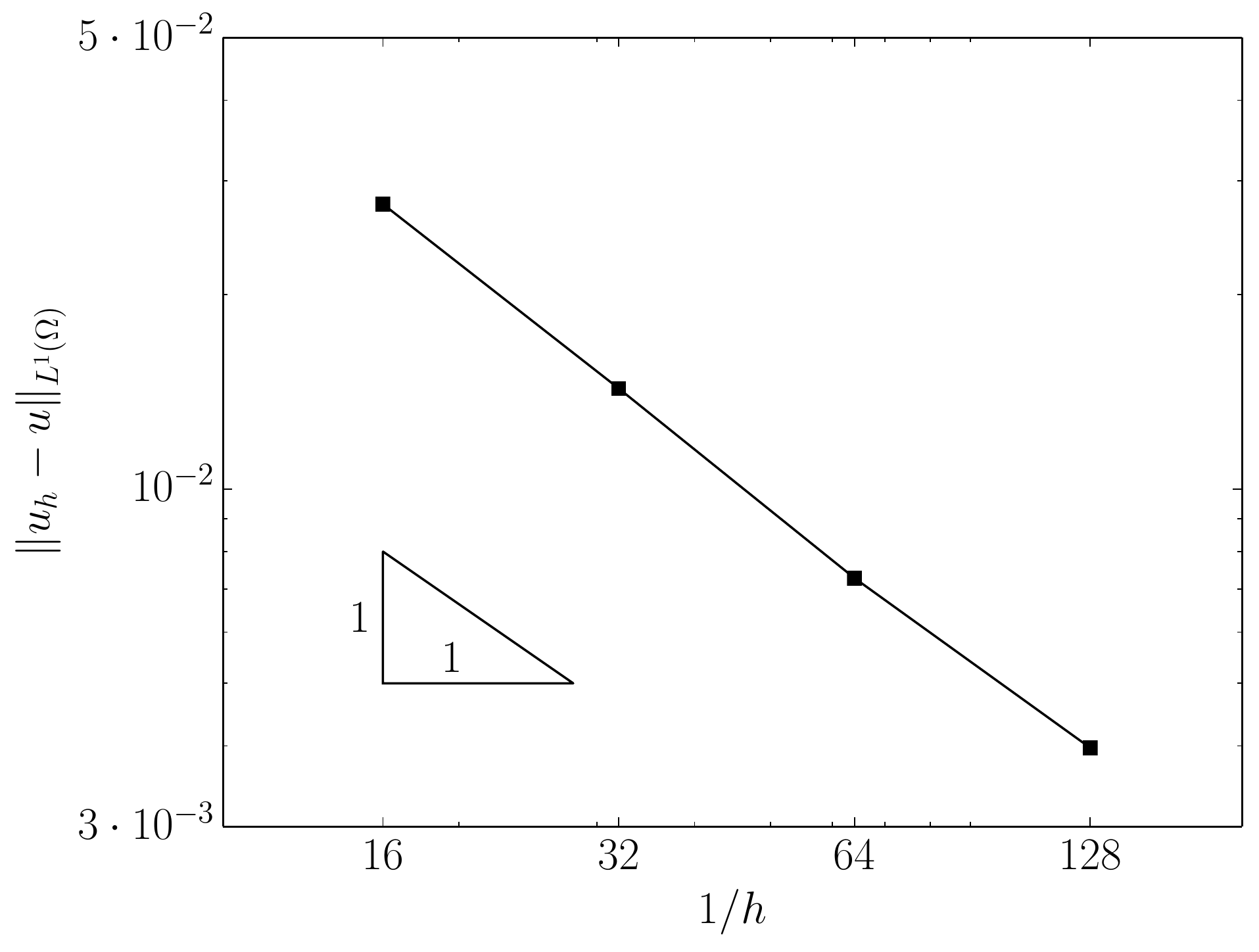}
		\caption{Euler equations.}
	\end{subfigure}
	\caption{Convergence of $\|u-\unk\|_{\lone(\domain)}$ to a solution with a discontinuity.}
	\label{fig.convergence}
\end{figure}

Since the solution is not smooth, we expect linear convergence rates in the $\lone$-norm. Fig.\ \ref{fig.convergence} shows the convergence behavior of both problems with uniform mesh refinements. The experimental convergence rate measured for the scalar transport problem is 0.82, whereas the convergence rate measured is 0.94 for the compression corner test. Therefore, both tests exhibit the expected convergence behavior.

\subsection{Linear discontinuity}

For this test, we use again the problem in \eqref{eq.6.steady_transport}. The purpose of this test is twofold. On the one hand, we analyze the effectiveness of the proposed \rev{indicator}. On the other hand, we compare the effectiveness of the linear and nonlinear stabilization methods. Specifically, this effectiveness is measured as follows. For a given error, we consider a method more effective if it requires less computational time, independently of the number of elements required. In addition, we also solve the problem for successive uniformly refined meshes in order to evaluate the effect of \ac{amr}.

For all comparisons, we start with a coarse mesh of $16\times 16$ elements, and proceed adapting the mesh up to a maximum number of elements. For the nonlinear stabilization, we set a maximum of \revv{$5\cdot10^4$} elements. The maximum number of elements for the low-order method is \revv{$2\cdot10^6$}. The uniform mesh is refined up to a \revv{$1024\times 1024$} mesh. We use a nonlinear tolerance of $\|\Delta \bunk\|/\|\bunk\|<10^{-4}$, and a maximum of 500 iterations. 

Fig.\ \ref{fig.linear-amr} shows the evolution of the \ac{amr} algorithm for \rev{the Kelly error estimator and the proposed refinement strategy based on the graph Laplacian}. The results shown in this picture have been obtained using the linear stabilization, and the left-most column using the nonlinear one. It can be observed that both Kelly ($\estimator_K$) \rev{estimator} and graph Laplacian ($\lestimator_K$) \rev{indicator} refine in the vicinity of the shock. However, the graph Laplacian operator clearly outperforms Kelly estimator.

Figs.\ \ref{fig.linear-q1}--\ref{fig.linear-q10} compare the effectiveness of the low-order and the high-order stabilization schemes. The results are obtained for the stabilization parameter $q=1$, $q=2$, and $q=10$, respectively.

At Fig.\ \ref{fig.linear-q2}, the nonlinear stabilization is able to converge the nonlinear problem efficiently and the overhead of solving a nonlinear problem does not strongly affect the overall performance. We note that for the linear scheme the problem is linear. It can be observed that the convergence rate (against time) is much higher for the nonlinear scheme. The linear scheme requires less computational time for coarser meshes but the nonlinear scheme is more effective for tighter accuracies, \revv{specially for $q>1$.}

\begin{figure}[!]
	\centering
	\newlength{\figwidth}
	\setlength{\figwidth}{0.8\textwidth}
	\begin{subfigure}[t]{\figwidth}
		\includegraphics[width=\textwidth,trim={0 135mm 0 25mm},clip]{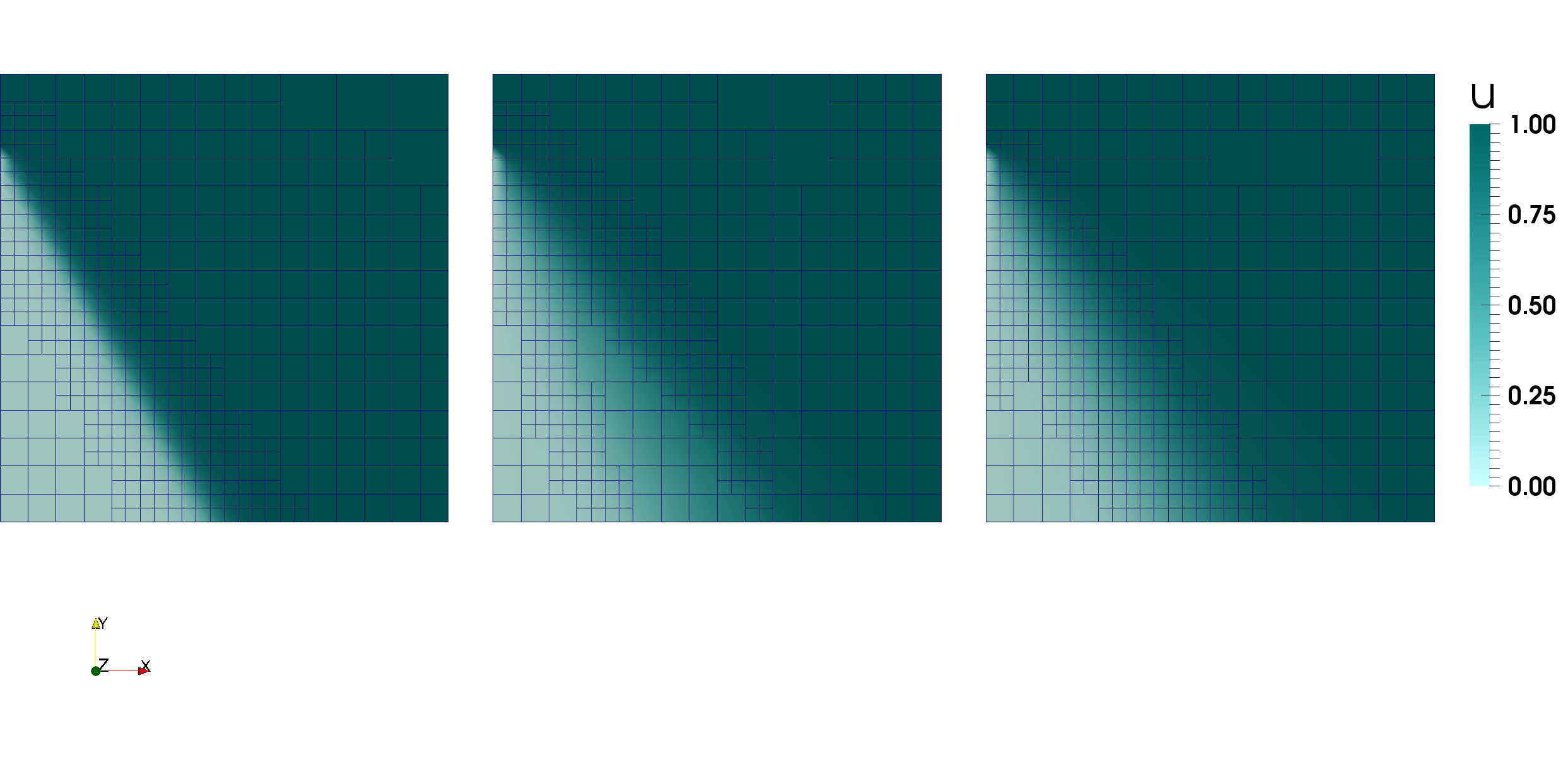}
	\end{subfigure}
	\begin{subfigure}[t]{\figwidth}
		\includegraphics[width=\textwidth,trim={0 135mm 0 25mm},clip]{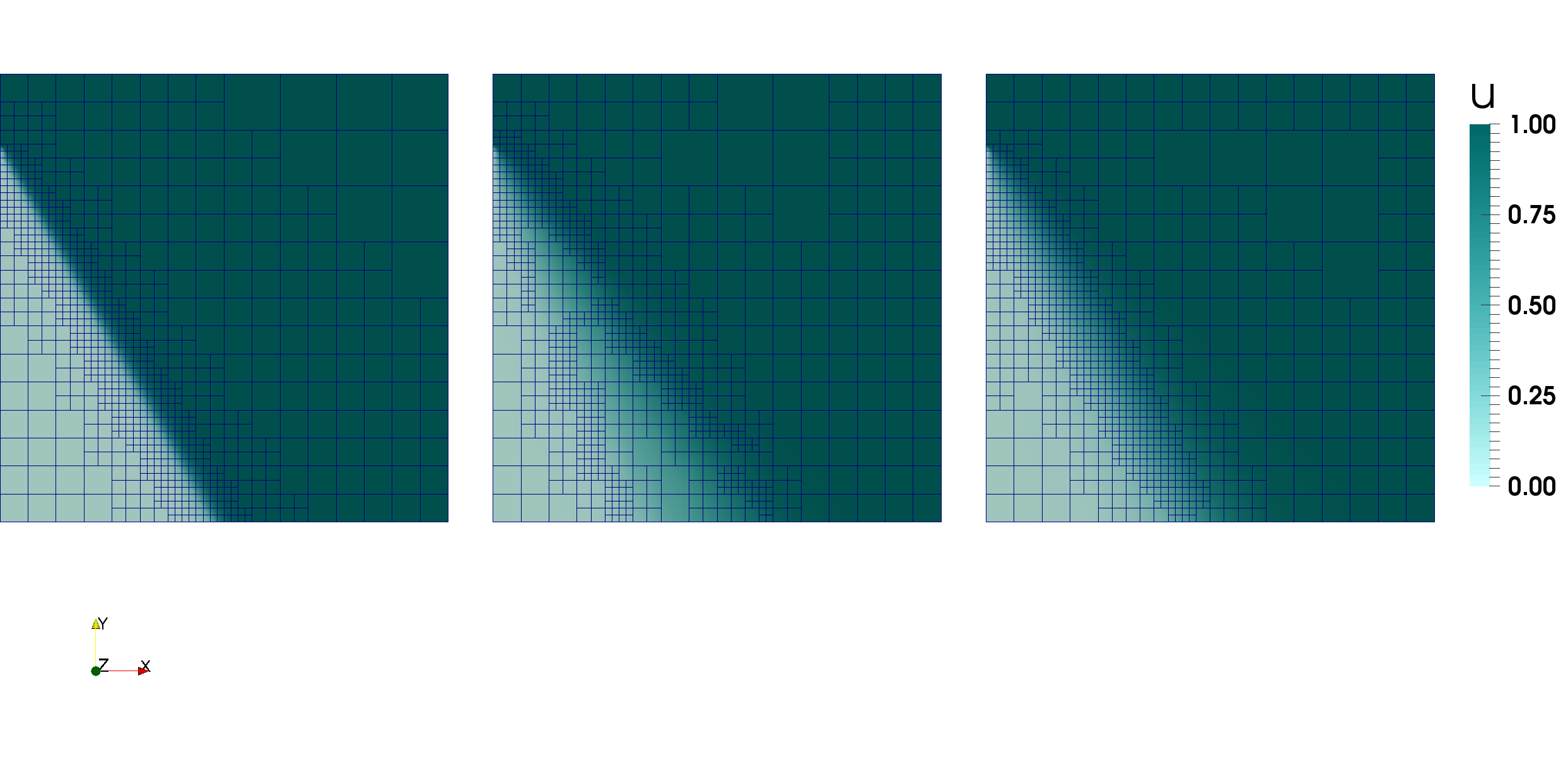}
	\end{subfigure}
	\begin{subfigure}[t]{\figwidth}
		\includegraphics[width=\textwidth,trim={0 135mm 0 25mm},clip]{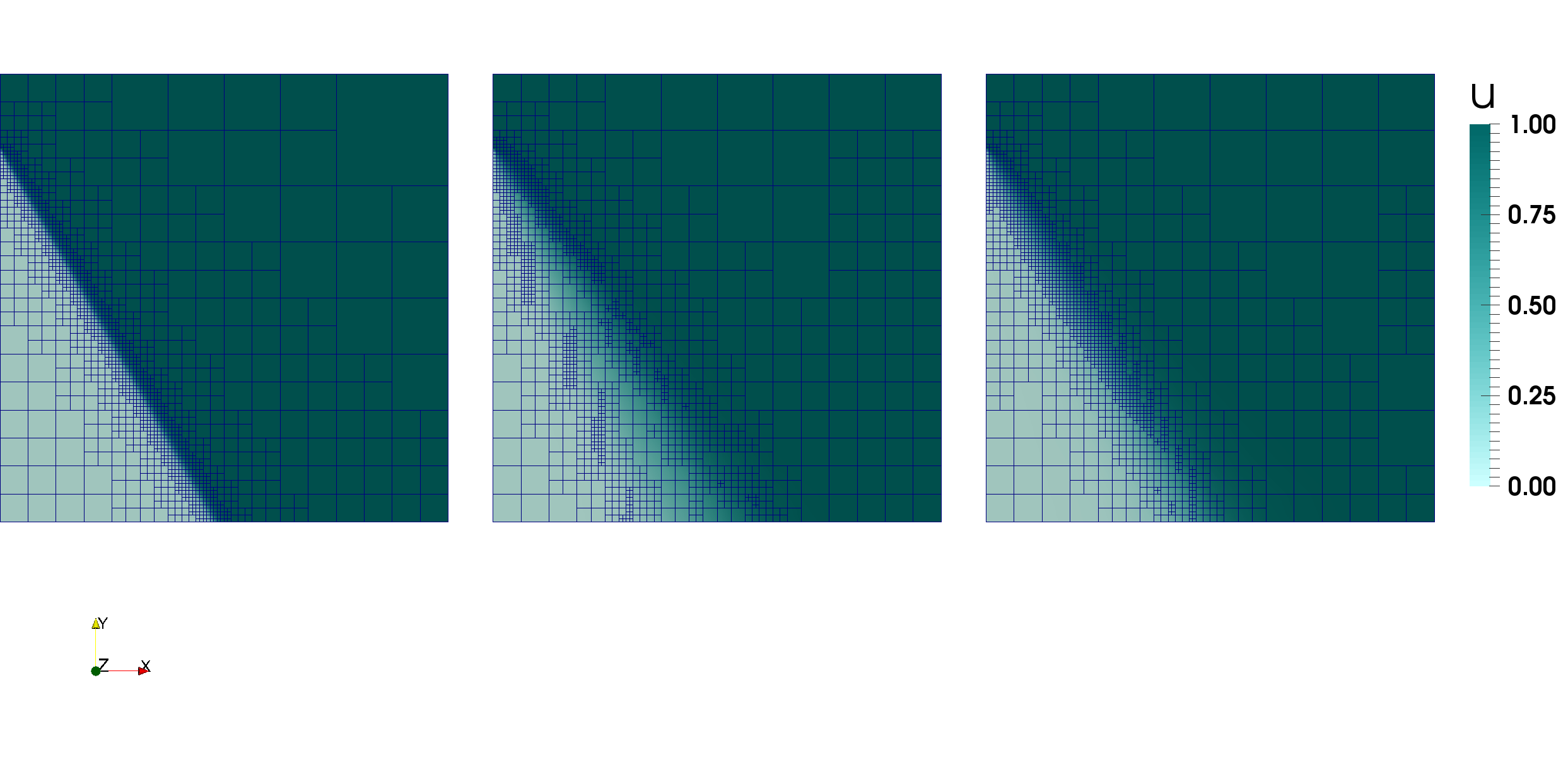}
	\end{subfigure}
	\begin{subfigure}[t]{\figwidth}
		\includegraphics[width=\textwidth,trim={0 135mm 0 25mm},clip]{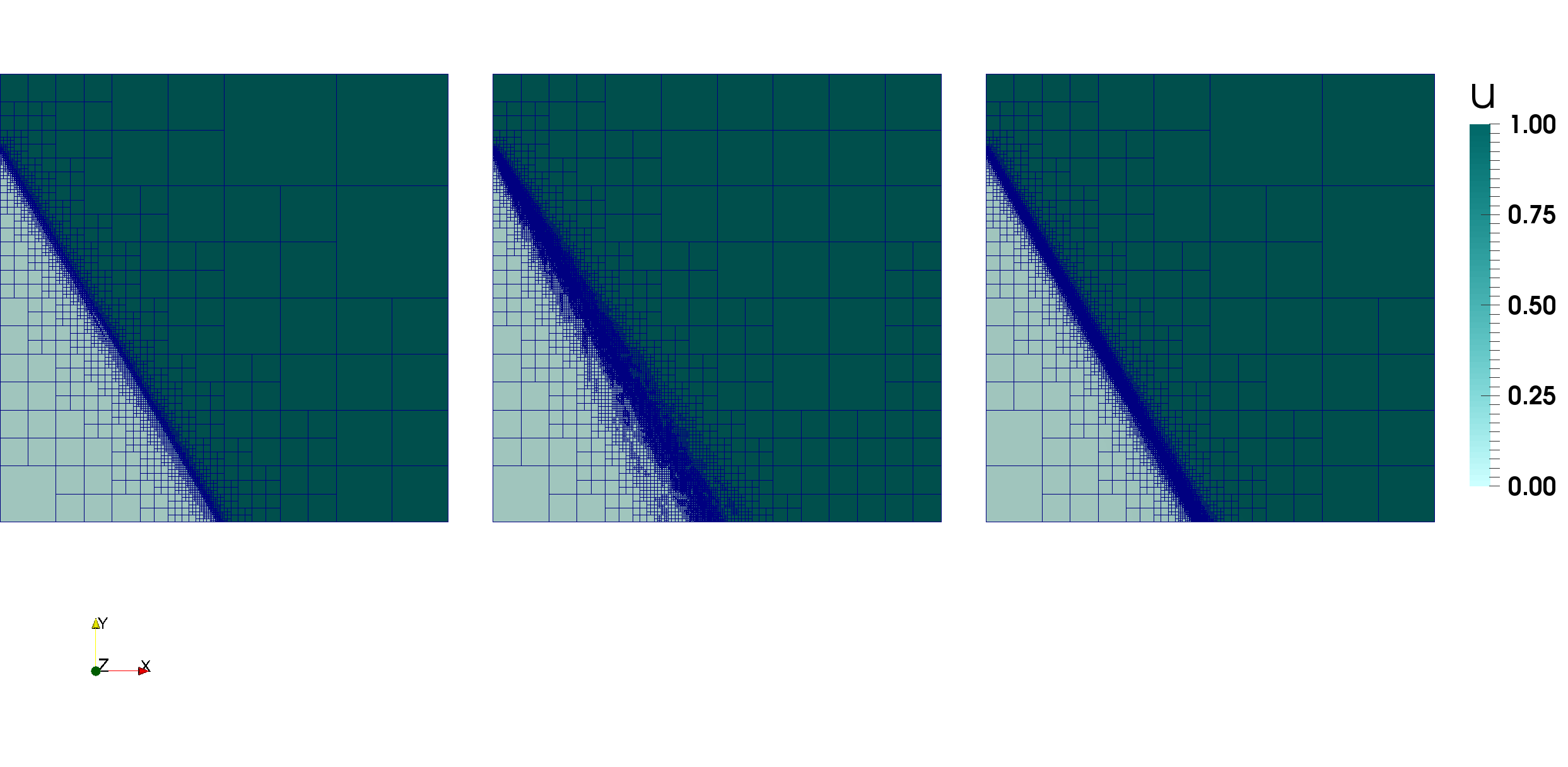}
	\end{subfigure}
	\begin{subfigure}[t]{\figwidth}
		\includegraphics[width=\textwidth,trim={0 135mm 0 25mm},clip]{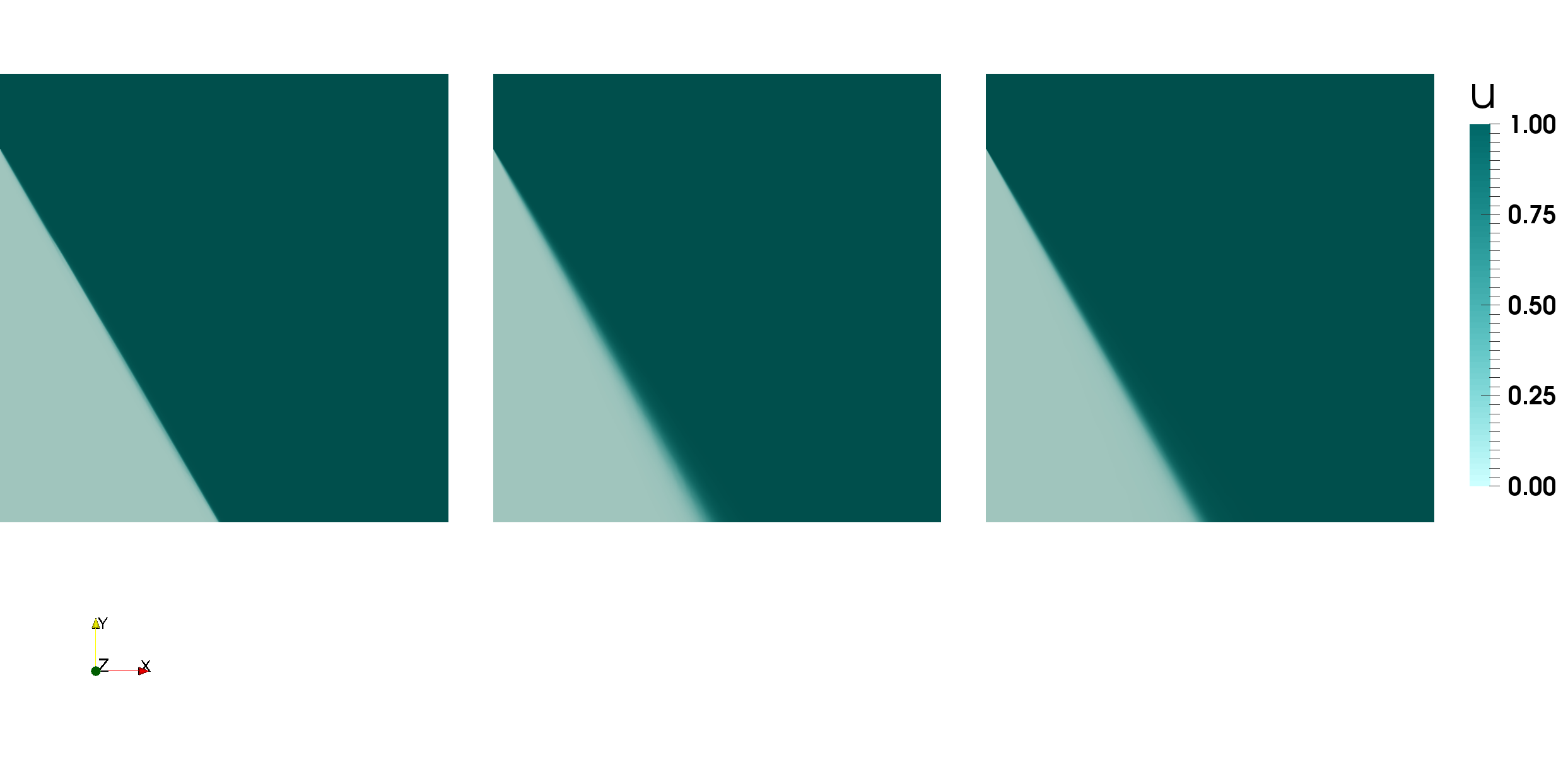}
	\end{subfigure}
	\caption{Evolution of the mesh refinement process. $\lestimator_K$ with high--order scheme is used in the left column. Low--order scheme with Kelly estimator is used in the central column. $\lestimator_K$ with low--order scheme is used in the right column. For the low--order scheme from top to bottom results have been obtained at refinement step 1, 2, 3, 9, and 9. For the high--order refinement steps are 1, 2, 3, 5, and 5.}
	\label{fig.linear-amr}
\end{figure}

We can observe in Fig.\ \ref{fig.linear-q10} the convergence problems of the nonlinear stabilization at some steps of the refinement procedure. Even though using $q=10$ improves the accuracy of the method, it also increases the computational cost since the nonlinear problem is harder to solve. \revv{As a consequence, the nonlinear stabilization needs a very refined mesh to overcome the performance of the linear stabilization.}

\begin{figure}[h]
	\centering
	\includegraphics[width=0.75\textwidth]{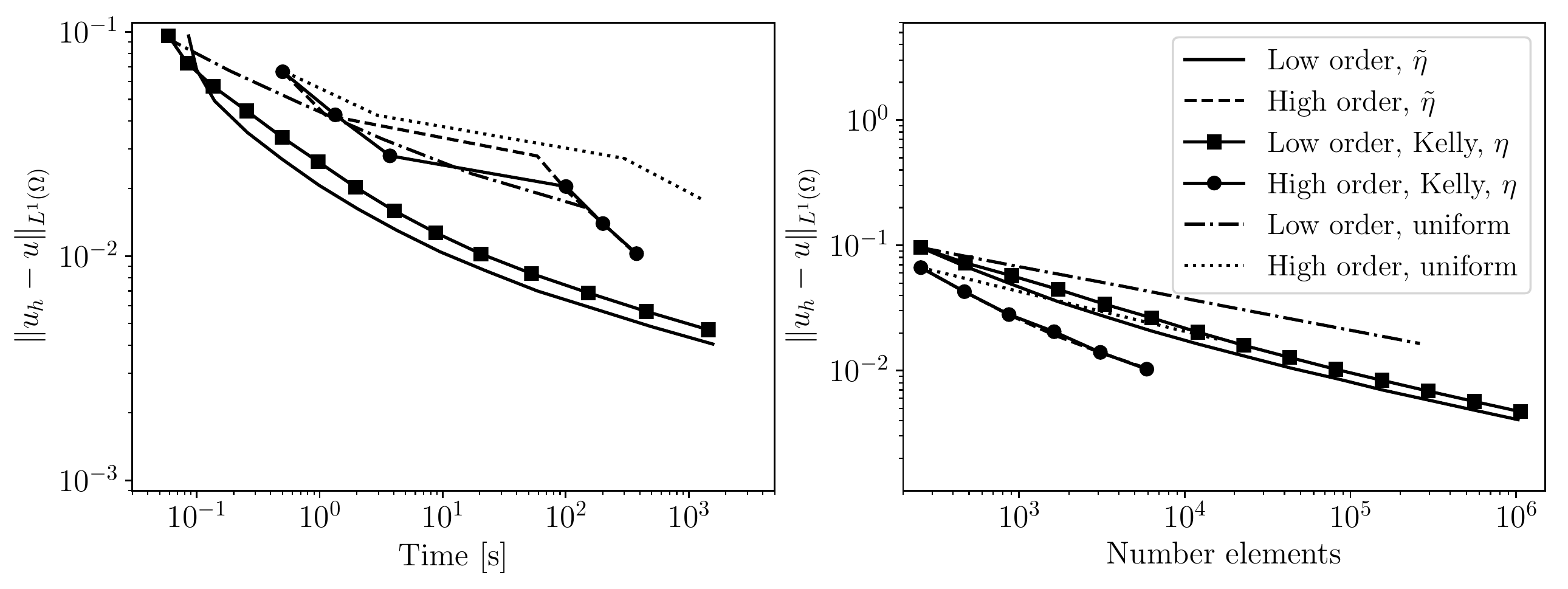}
	\caption{Time and elements convergence comparison for the transport problem with a linear discontinuity, $q=1$.}
	\label{fig.linear-q1}
\end{figure}

\begin{figure}[h]
	\centering
	\includegraphics[width=0.75\textwidth]{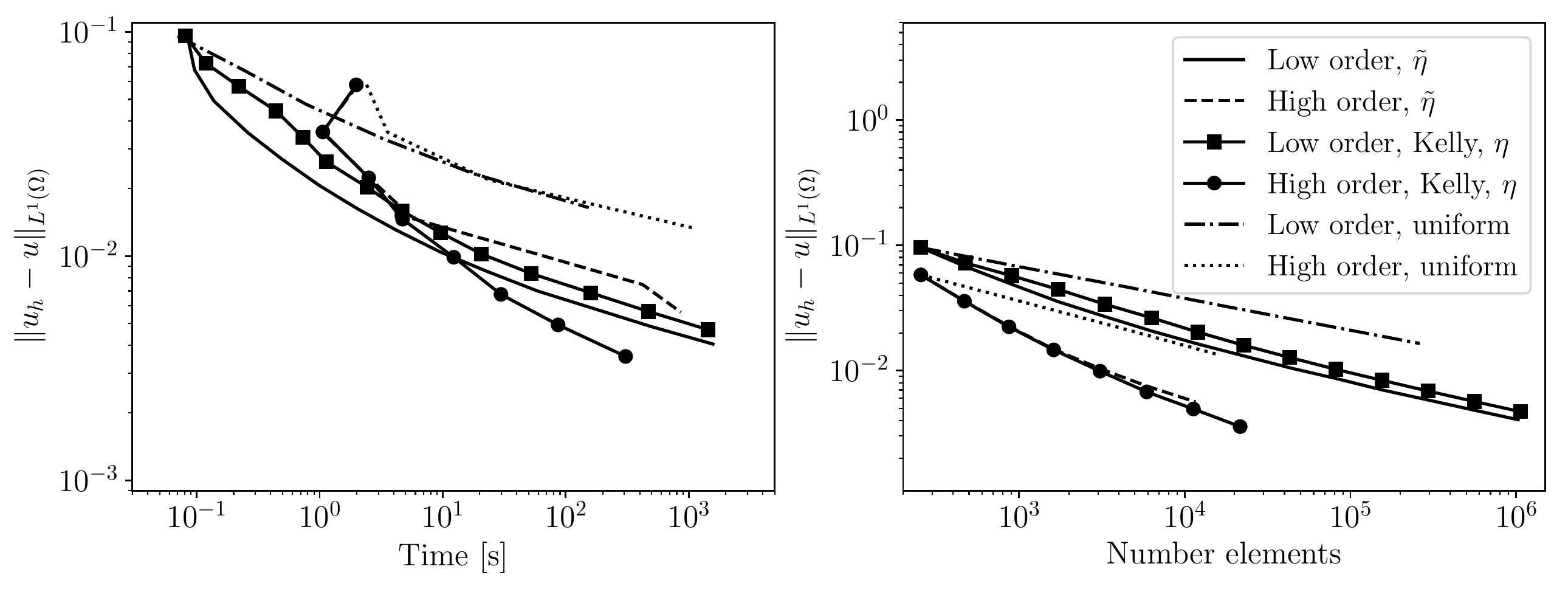}
	\caption{Time and elements convergence comparison for the transport problem with a linear discontinuity, $q=2$.}
	\label{fig.linear-q2}
\end{figure}

\begin{figure}[h]
	\centering
	\includegraphics[width=0.75\textwidth]{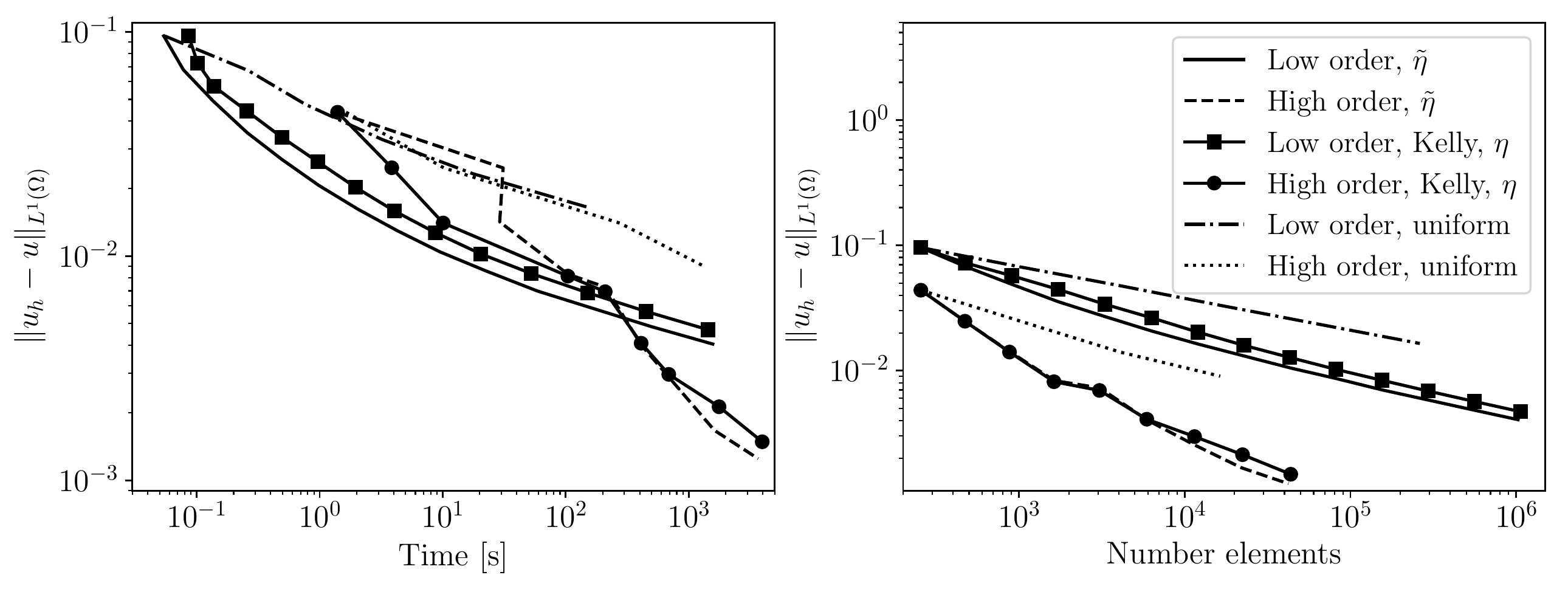}
	\caption{Time and elements convergence comparison for the transport problem with a linear discontinuity, $q=10$.}
	\label{fig.linear-q10}
\end{figure}

\subsection{Circular discontinuity}
We analyze again the effectiveness of the proposed \rev{refinement strategy} and the effectiveness of the linear and nonlinear stabilization methods for a slightly more complicated convective field.  For this test, we use \eqref{eq.6.steady_transport} with $\mathbf{v}(x,y) \doteq (y, -x)$, and inflow boundary conditions
\begin{align}
\uboundary(0,y) = \left\{\begin{array}{ll}
1 & y\in[0.15,0.45],\\
\cos^2\left(\frac{10}{3}\pi(y-0.4)\right) & y\in[0.55,0.85],\\
0 & {\rm elsewhere}.
\end{array} \right.
\end{align}
The analytical solution of this particular configuration consists in the transport of the inflow profile in the direction of the convection. As a result, the solution at the outflow boundary, corresponding to $y=0$, is 
$
u(x,0) = \uboundary(0,x).
$
We start with a coarse mesh of $16\times 16$ elements in all cases, and proceed adapting the mesh up to a maximum number of elements. For the nonlinear stabilization, we set a maximum of \revv{$5\cdot10^4$} elements. The maximum number of elements for the linear stabilization is \revv{$2\cdot10^6$}. We use a nonlinear tolerance of $\|\Delta \bunk\|/\|\bunk\|<10^{-4}$, and a maximum of 500 iterations. 

\begin{figure}[h]
	\centering
	\includegraphics[width=0.75\textwidth]{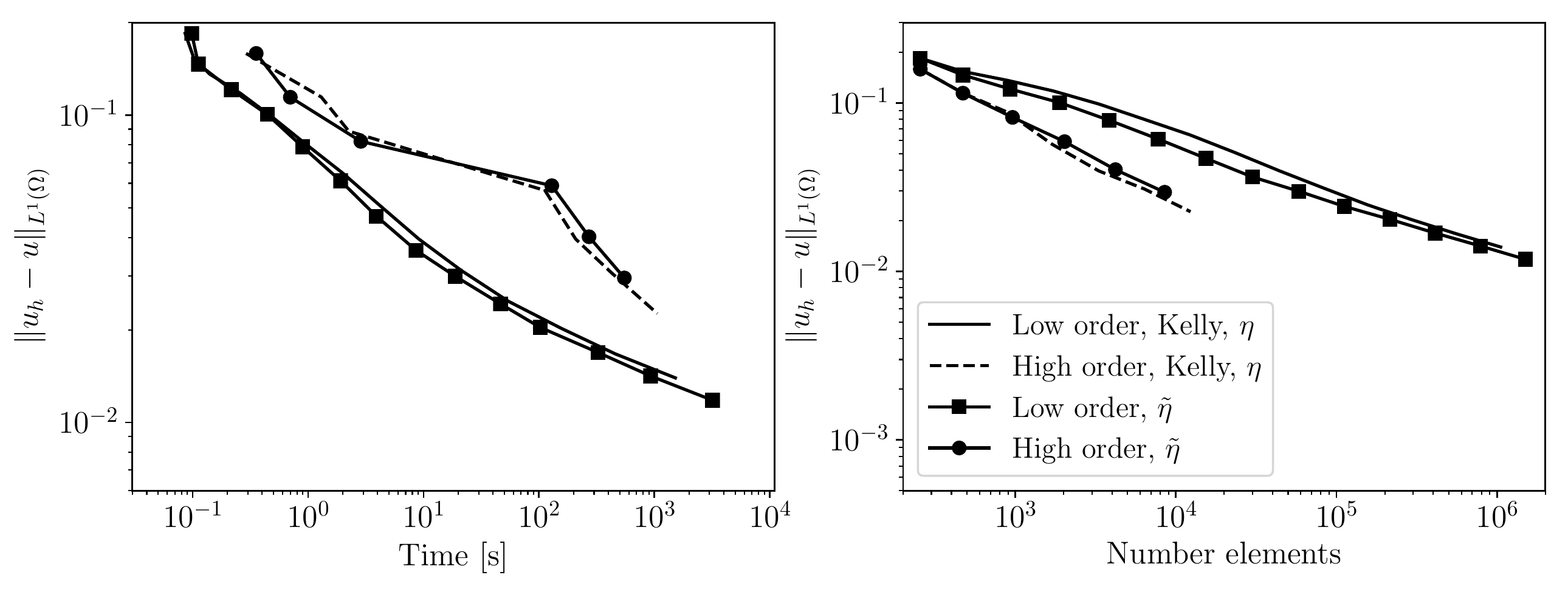}
	\caption{Time and elements convergence comparison for the transport problem with a circular convection field, $q=1$.}
	\label{fig.circular-q1}
\end{figure}

\begin{figure}[h]
	\centering
	\includegraphics[width=0.75\textwidth]{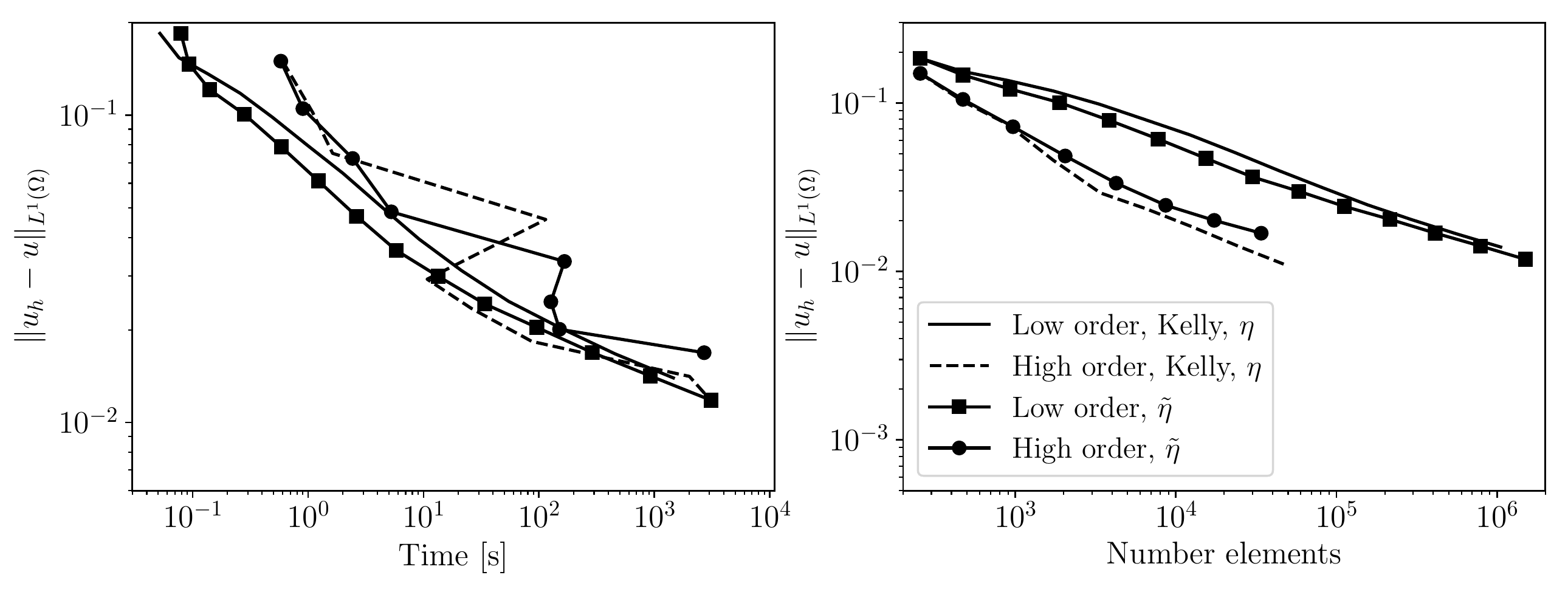}
	\caption{Time and elements convergence comparison for the transport problem with a circular convection field, $q=2$.}
	\label{fig.circular-q2}
\end{figure}

\begin{figure}[h]
	\centering
	\includegraphics[width=0.75\textwidth]{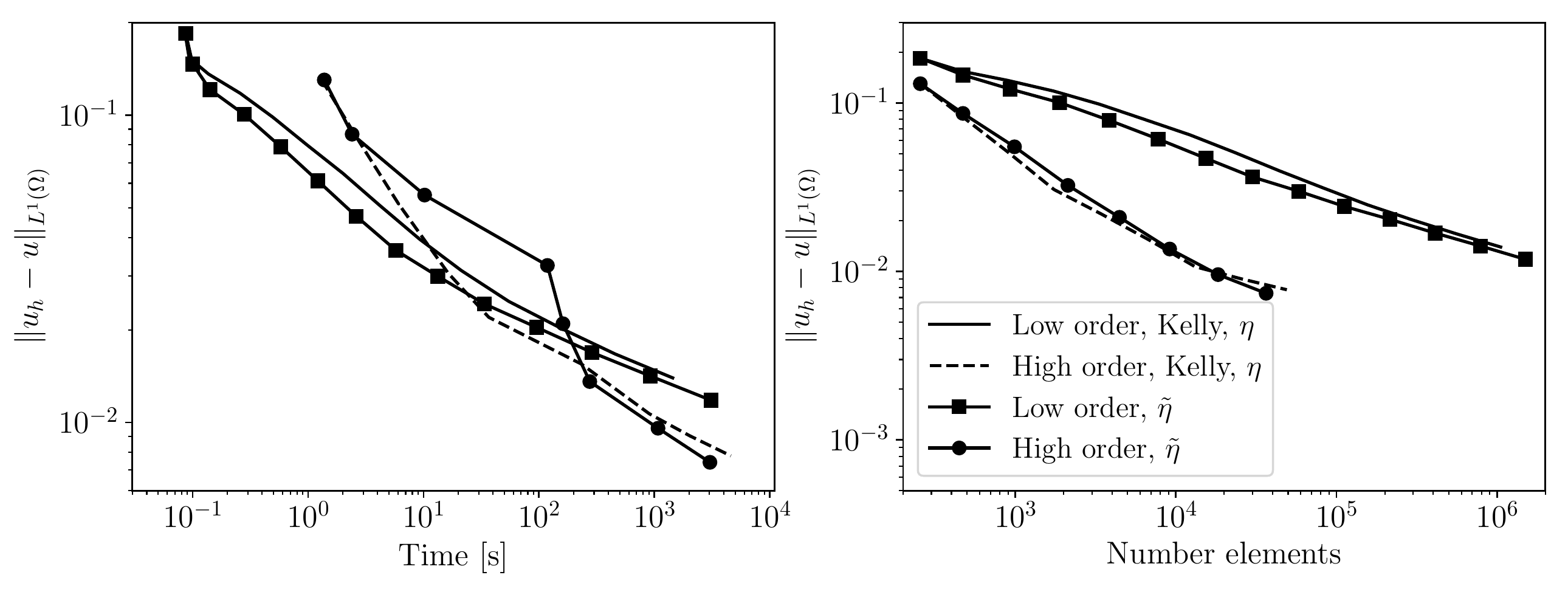}
	\caption{Time and elements convergence comparison for the transport problem with a circular convection field, $q=10$.}
	\label{fig.circular-q10}
\end{figure}

\begin{figure}[!]
	\centering
	\setlength{\figwidth}{0.81\textwidth}
	\begin{subfigure}[t]{\figwidth}
		\includegraphics[width=\textwidth,trim={0 135mm 0 25mm},clip]{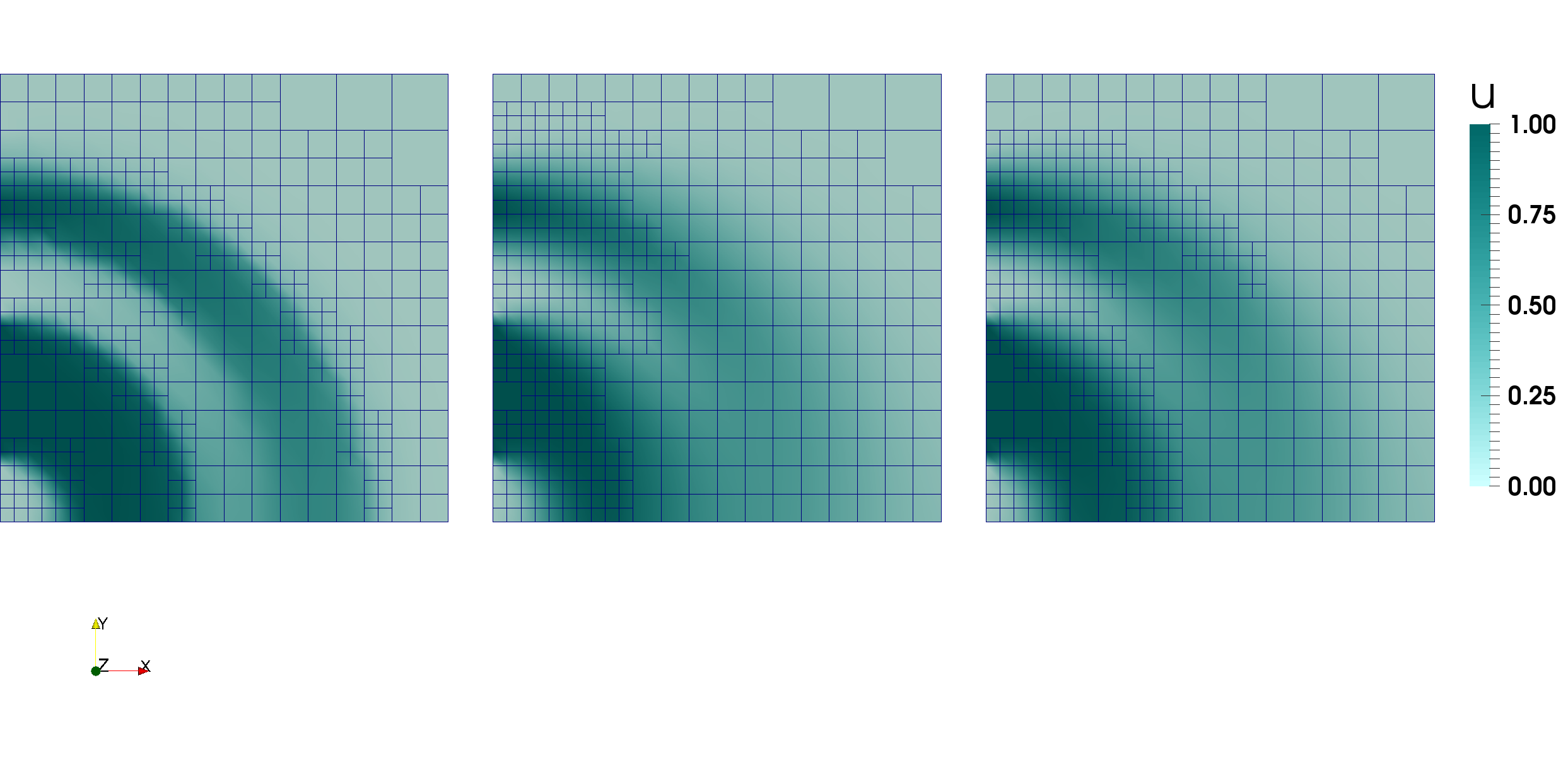}
	\end{subfigure}
	\begin{subfigure}[t]{\figwidth}
		\includegraphics[width=\textwidth,trim={0 135mm 0 25mm},clip]{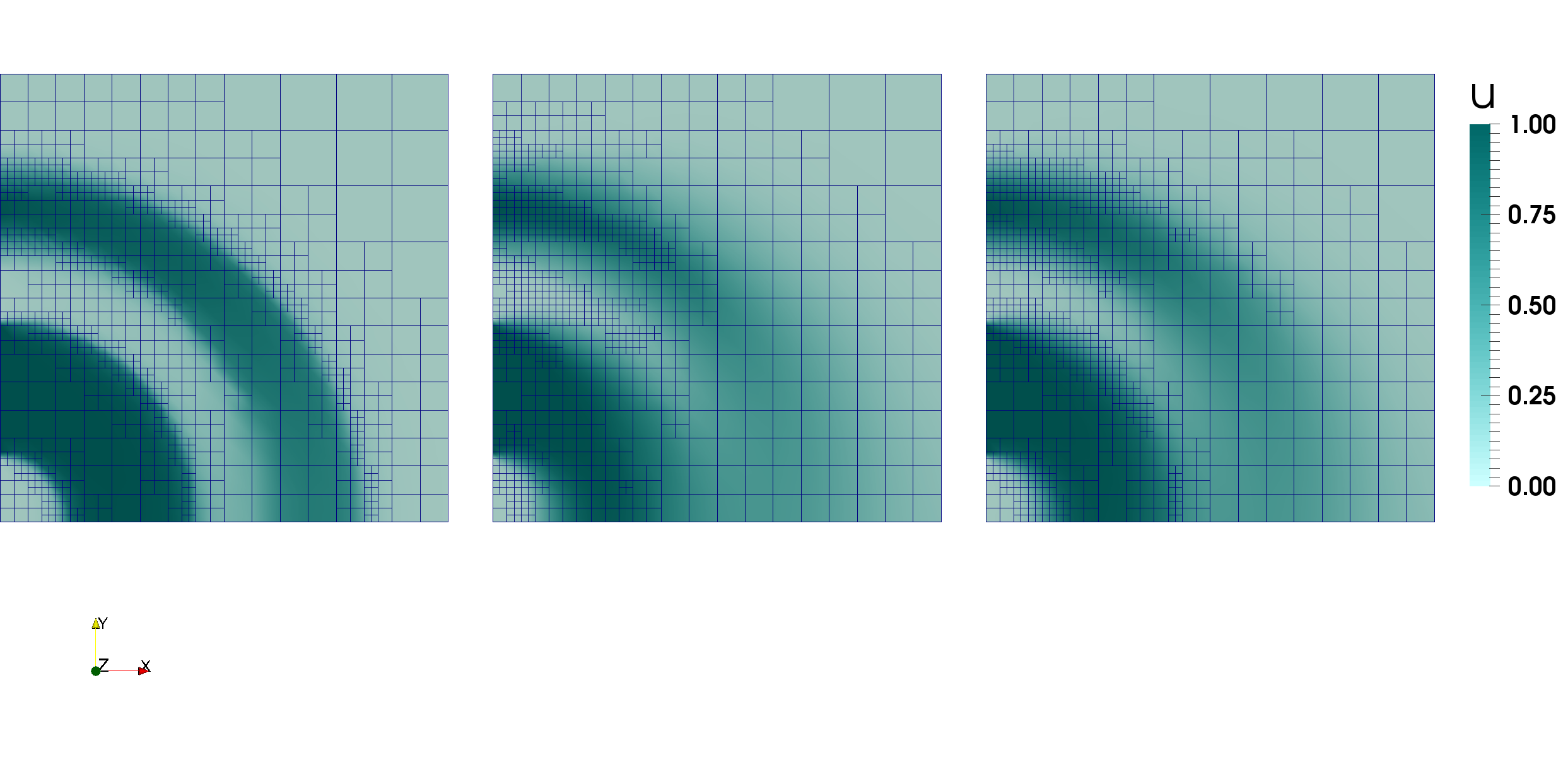}
	\end{subfigure}
	\begin{subfigure}[t]{\figwidth}
		\includegraphics[width=\textwidth,trim={0 135mm 0 25mm},clip]{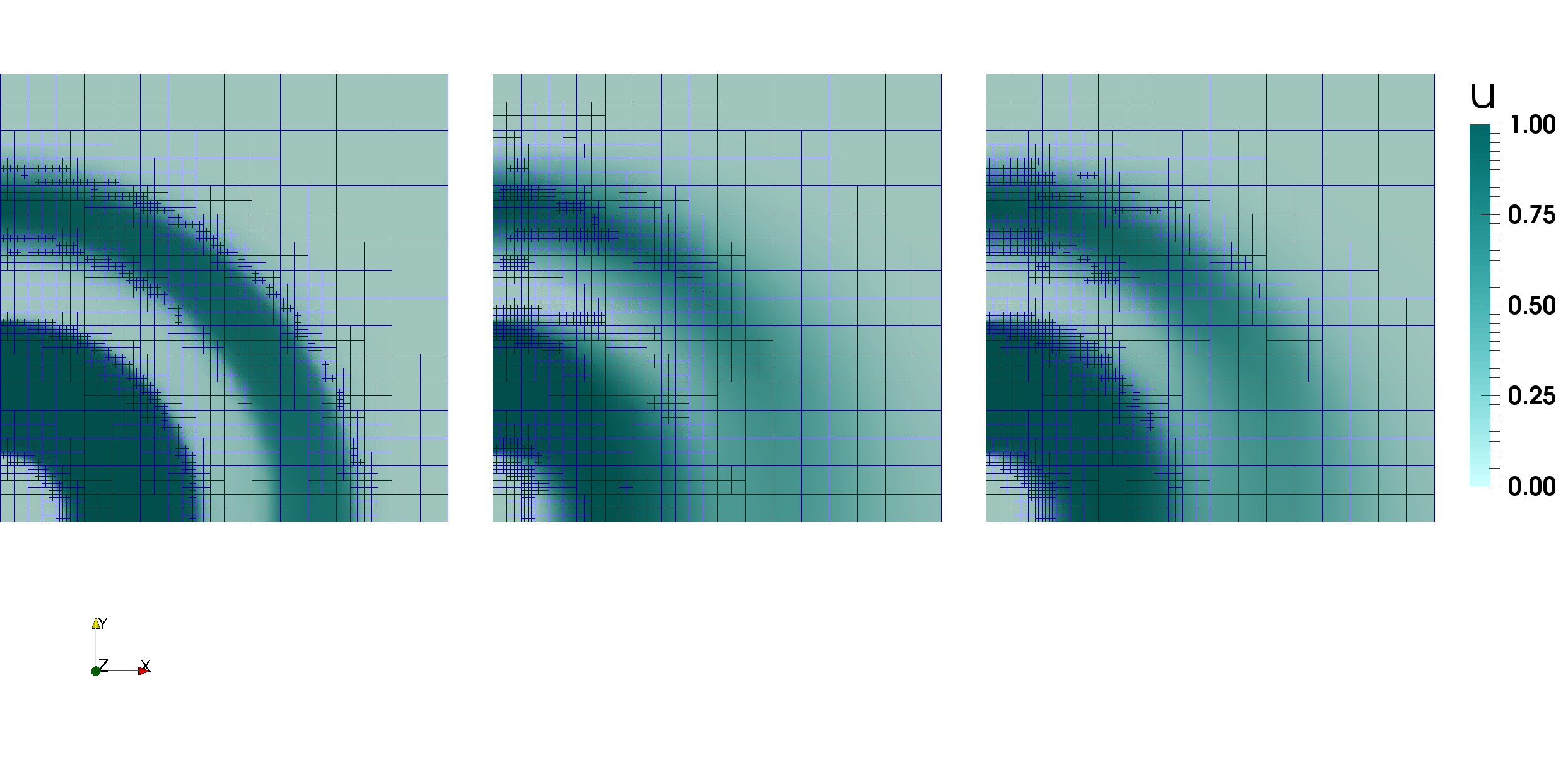}
	\end{subfigure}
	\begin{subfigure}[t]{\figwidth}
		\includegraphics[width=\textwidth,trim={0 135mm 0 25mm},clip]{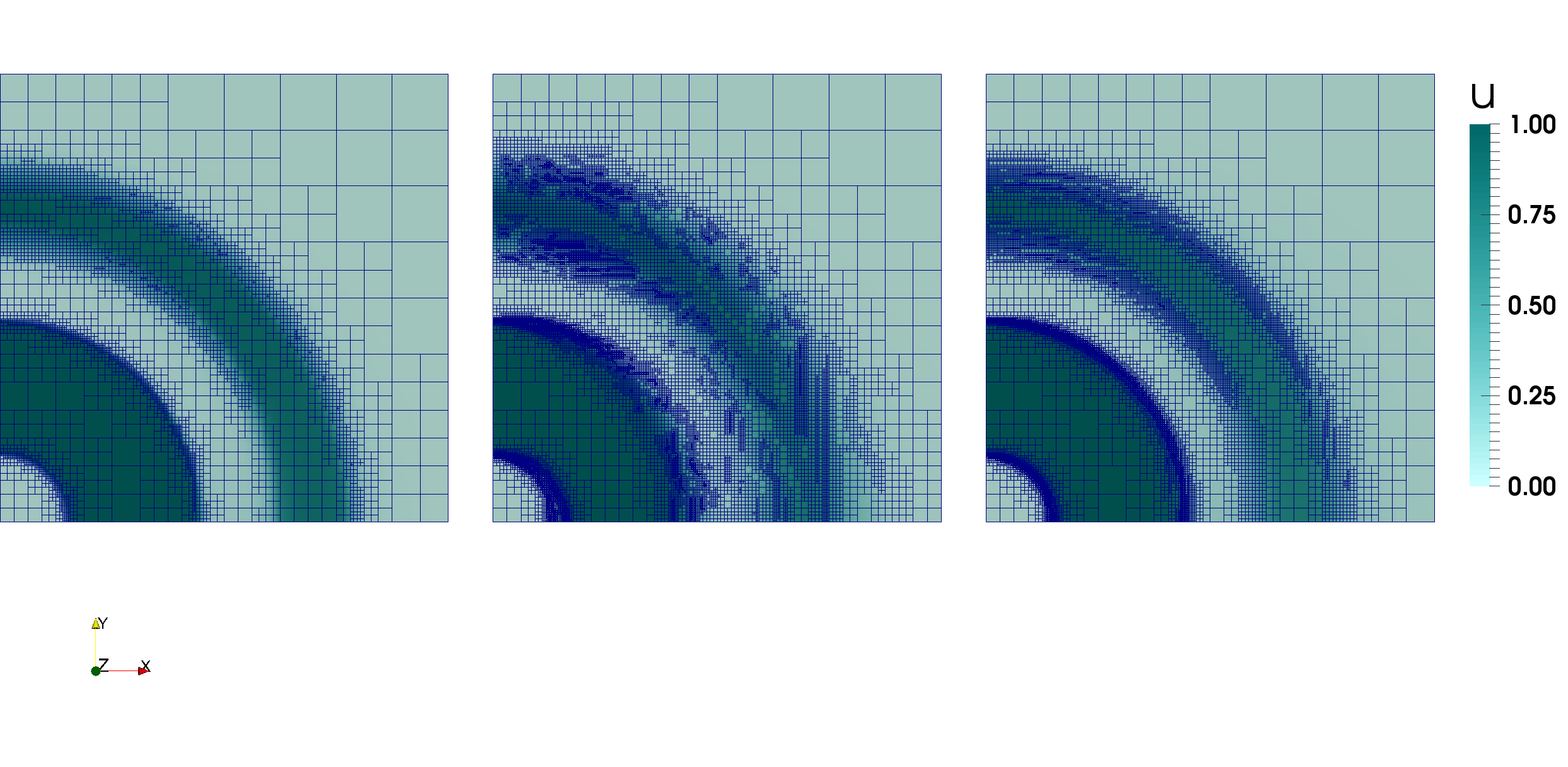}
	\end{subfigure}
	\begin{subfigure}[t]{\figwidth}
		\includegraphics[width=\textwidth,trim={0 135mm 0 25mm},clip]{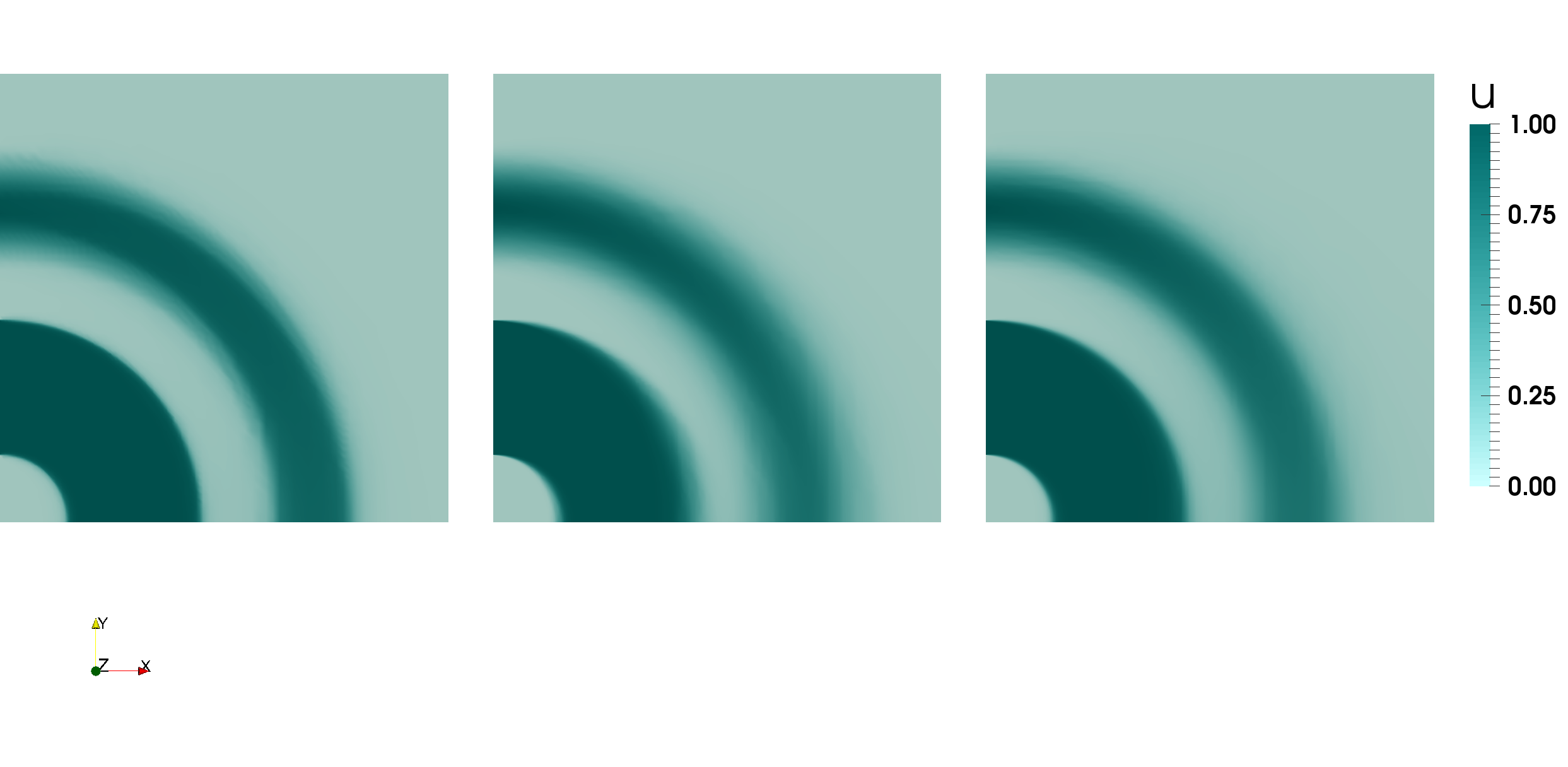}
	\end{subfigure}
	\caption{Evolution of the mesh refinement process. $\lestimator_K$ with high--order scheme is used in the left column. Low--order scheme with Kelly estimator is used in the central column. $\lestimator_K$ with low--order scheme is used in the right column. For the low--order scheme from top to bottom results have been obtained at refinement step 1, 2, 3, 7, and 7. For the high--order refinement steps are 1, 2, 3, 4, and 4.}
	\label{fig.circular-amr}
\end{figure}

Figs.\ \ref{fig.circular-q1}--\ref{fig.circular-q10} compare the effectiveness of the linear and nonlinear stabilization. These results use the stabilization parameter $q=1$, $q=2$, and $q=10$, respectively. In Fig.\ \ref{fig.circular-q2}, the high-order scheme is able to converge efficiently and the overhead of solving a nonlinear problem does not strongly affect the overall performance. Nevertheless, the low-order scheme \revv{usually} requires \revv{similar} computational time for any given error. However, it can be observed that the convergence rate (in time) \revv{using the Kelly error estimator is slightly} higher for the high-order scheme. \revv{Actually, it outperforms the low-order scheme for the finer meshes.}

In contrast, we do not observe the significant convergence problems in Fig.\ \ref{fig.circular-q10} even though the linear stabilization is slightly more efficient \revv{for coarse meshes}. The convergence rate (in time) is higher for the nonlinear stabilization and it \revv{is actually more efficient for the finer meshes}.

Fig.\ \ref{fig.circular-amr} shows the evolution of the \ac{amr} algorithm for both $\estimator_K$ and $\lestimator_K$ with the linear stabilization and $\lestimator_K$ with the nonlinear one. It can be observed that both Kelly ($\estimator_K$) \rev{estimator} and graph Laplacian indicator detect the regions that require more resolution. In any case, as in the previous example, the graph Laplacian operator ($\lestimator_K$) performs slightly better.

\subsection{Compression corner}

Let us consider now the Euler equations. We start with the compression corner test (see Fig.\ \ref{fig.corner}). We analyze the effectiveness of the high-order scheme, and evaluate the performance of the graph Laplacian \rev{indicator}. We start with a coarse mesh of $16\times 16$ elements, and adapt it up to a maximum number of elements. For the high-order method, we set a maximum of $5\cdot10^3$ elements. The maximum number of elements for the low-order method is $5\cdot10^4$. We use a nonlinear tolerance of $\|\Delta \bunk\|/\|\bunk\|<10^{-4}$ and a maximum of 500 iterations. 

In Fig.\ \ref{fig.compression-amr}, we depict the refinement evolution for the graph Laplacian \rev{indicator} ($\lestimator_{\element}$) for linear and nonlinear stabilization. As expected, we can observe that for the high-order method the scheme is able to resolve the shock with less refinement steps. The linear stabilization is able to provide well-resolved shocks at the final refinement step.

\begin{figure}[!]
	\centering
	\setlength{\figwidth}{0.58\textwidth}
	\begin{subfigure}[t]{\figwidth}
		\includegraphics[width=\textwidth,trim={270mm 135mm 0 25mm},clip]{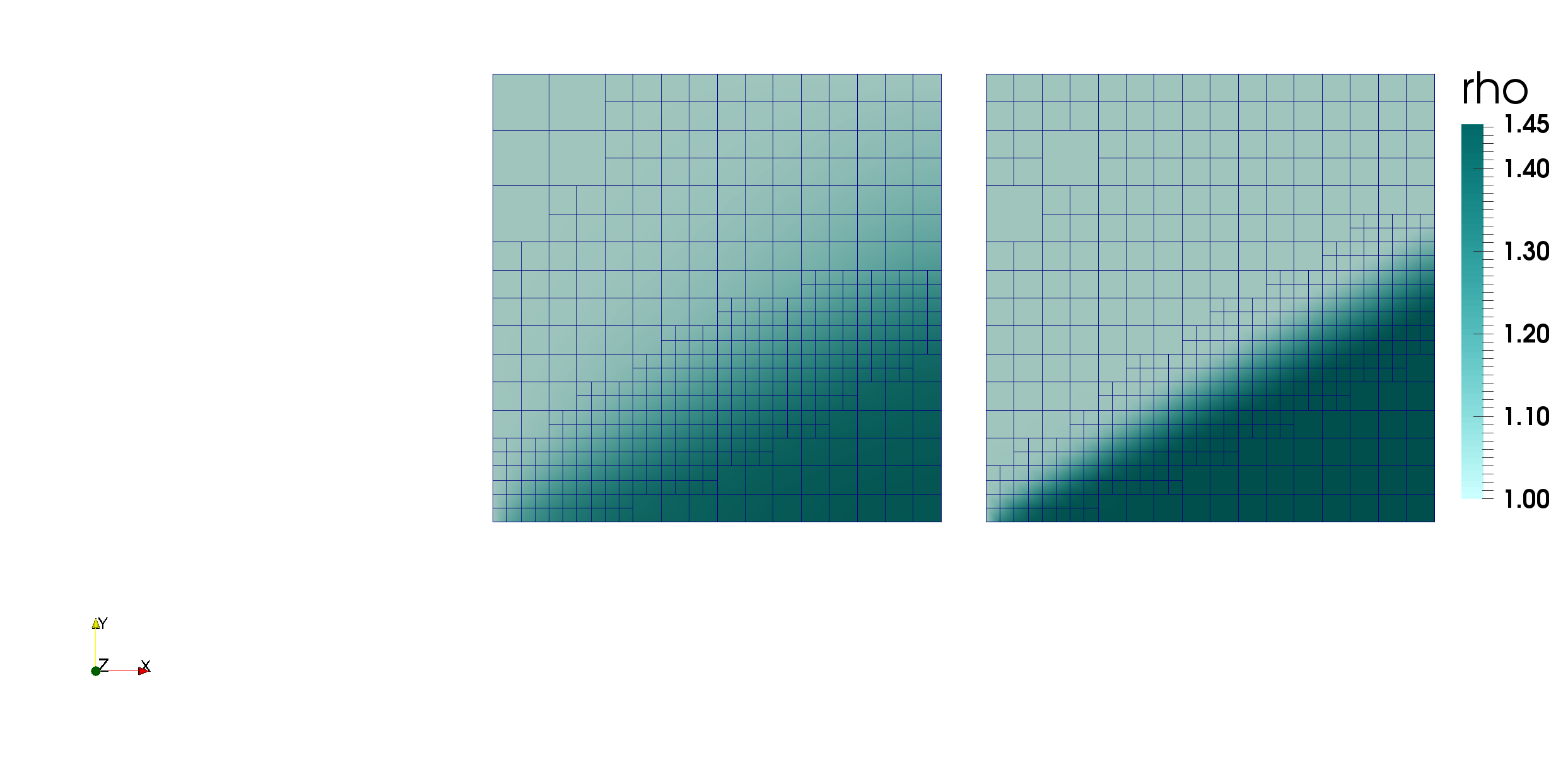}
	\end{subfigure}
	\begin{subfigure}[t]{\figwidth}
		\includegraphics[width=\textwidth,trim={270mm 135mm 0 25mm},clip]{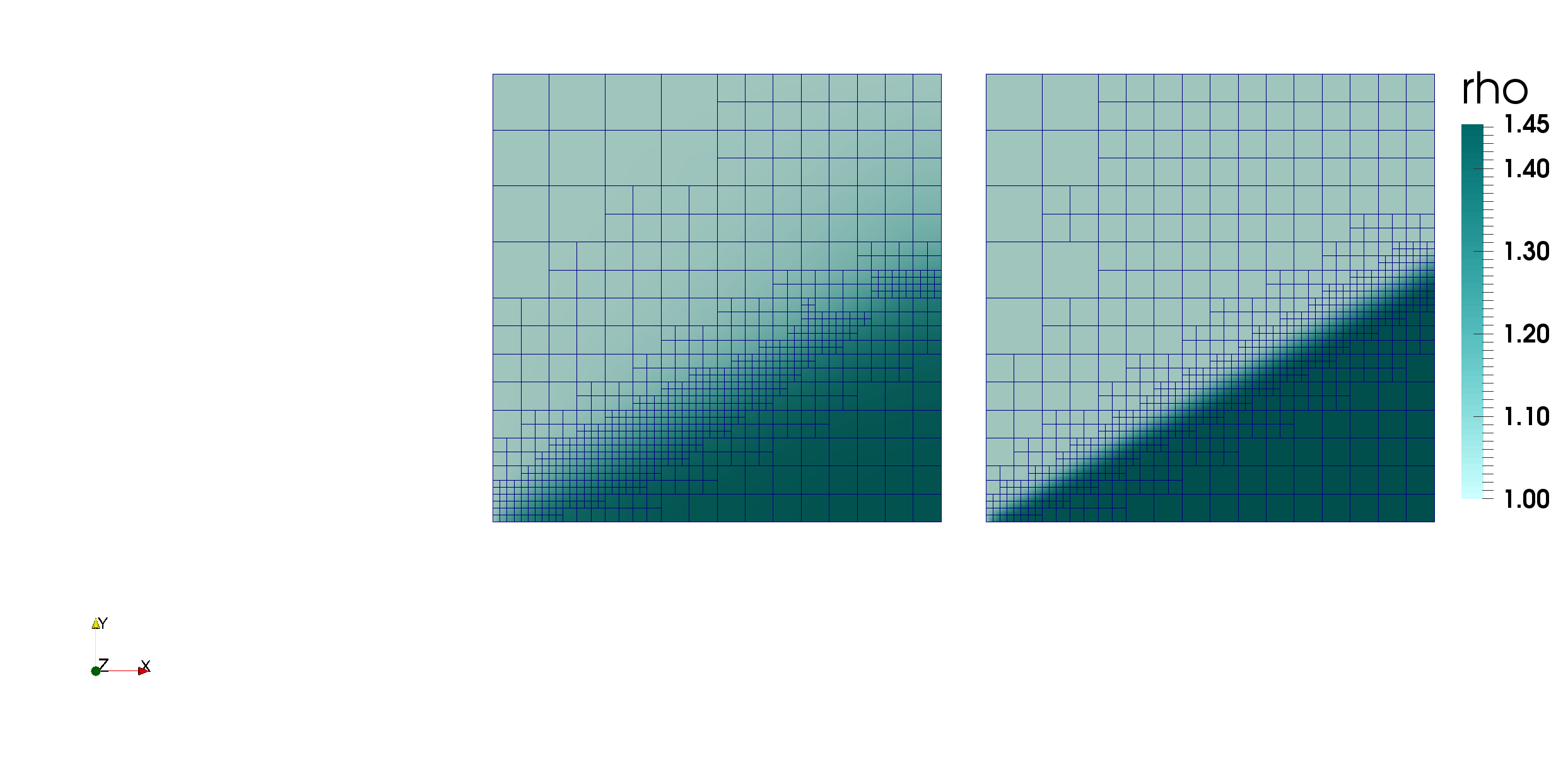}
	\end{subfigure}
	\begin{subfigure}[t]{\figwidth}
		\includegraphics[width=\textwidth,trim={270mm 135mm 0 25mm},clip]{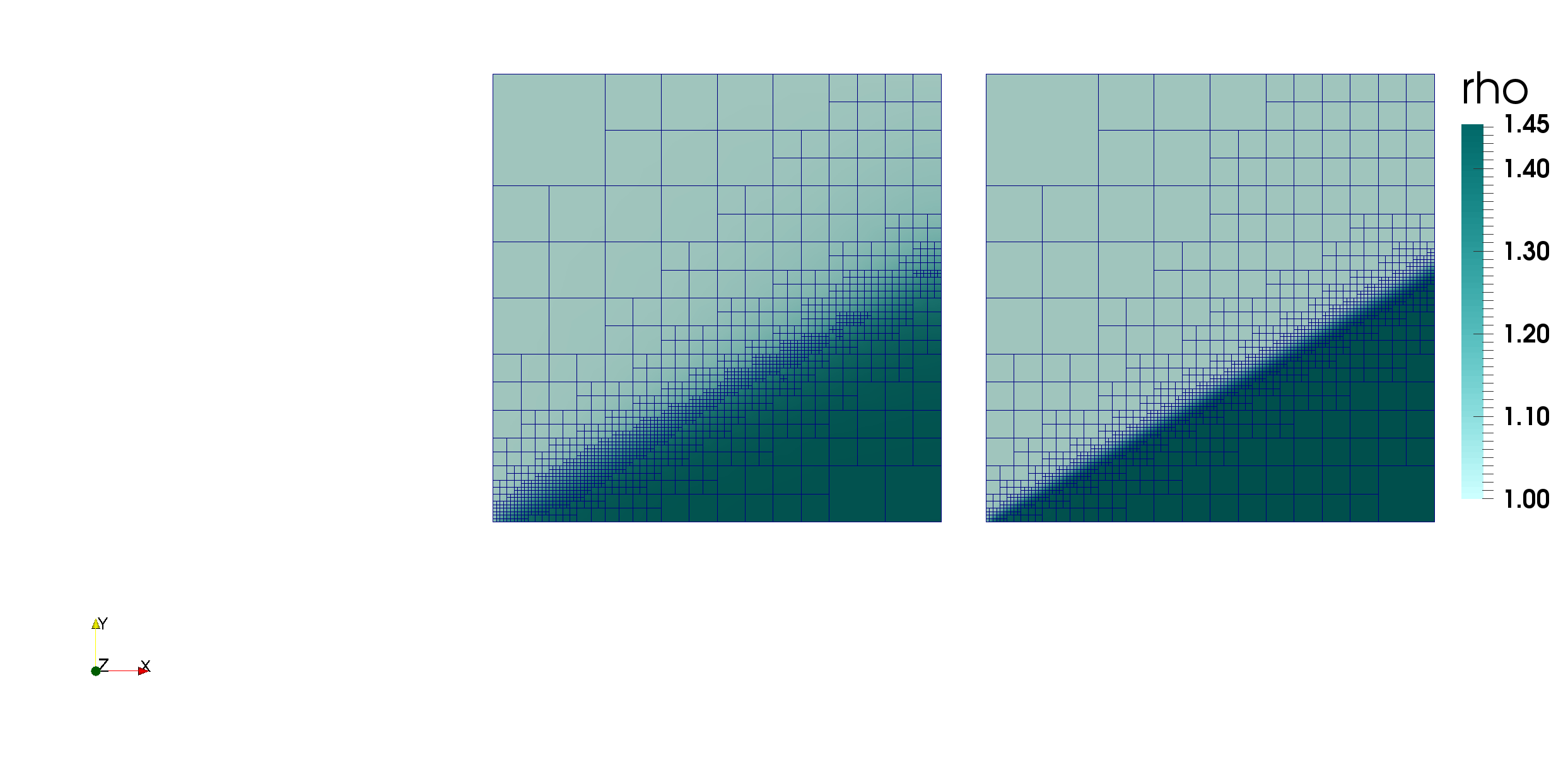}
	\end{subfigure}
	\begin{subfigure}[t]{\figwidth}
		\includegraphics[width=\textwidth,trim={270mm 135mm 0 25mm},clip]{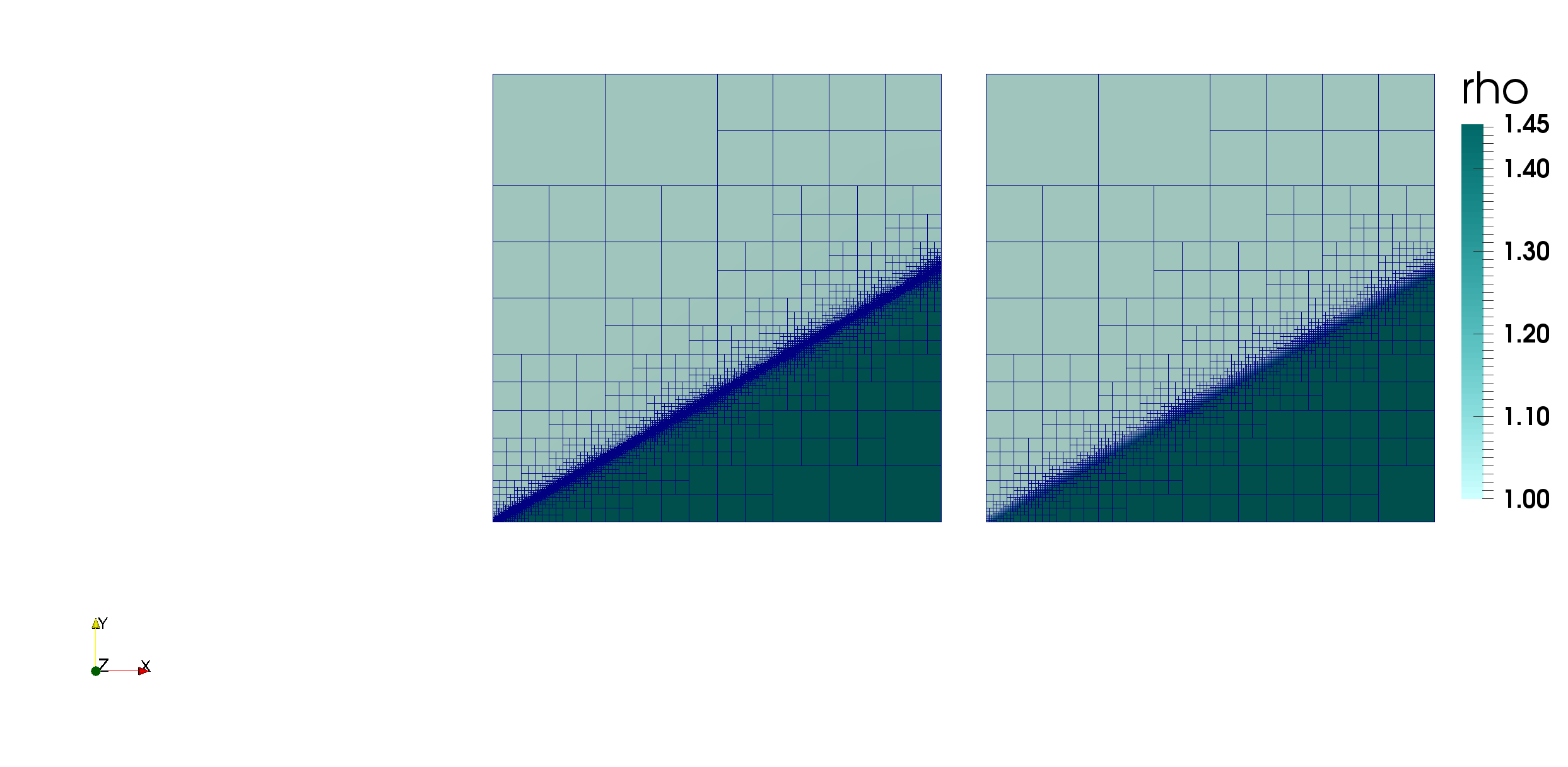}
	\end{subfigure}
	\begin{subfigure}[t]{\figwidth}
		\includegraphics[width=\textwidth,trim={270mm 135mm 0 25mm},clip]{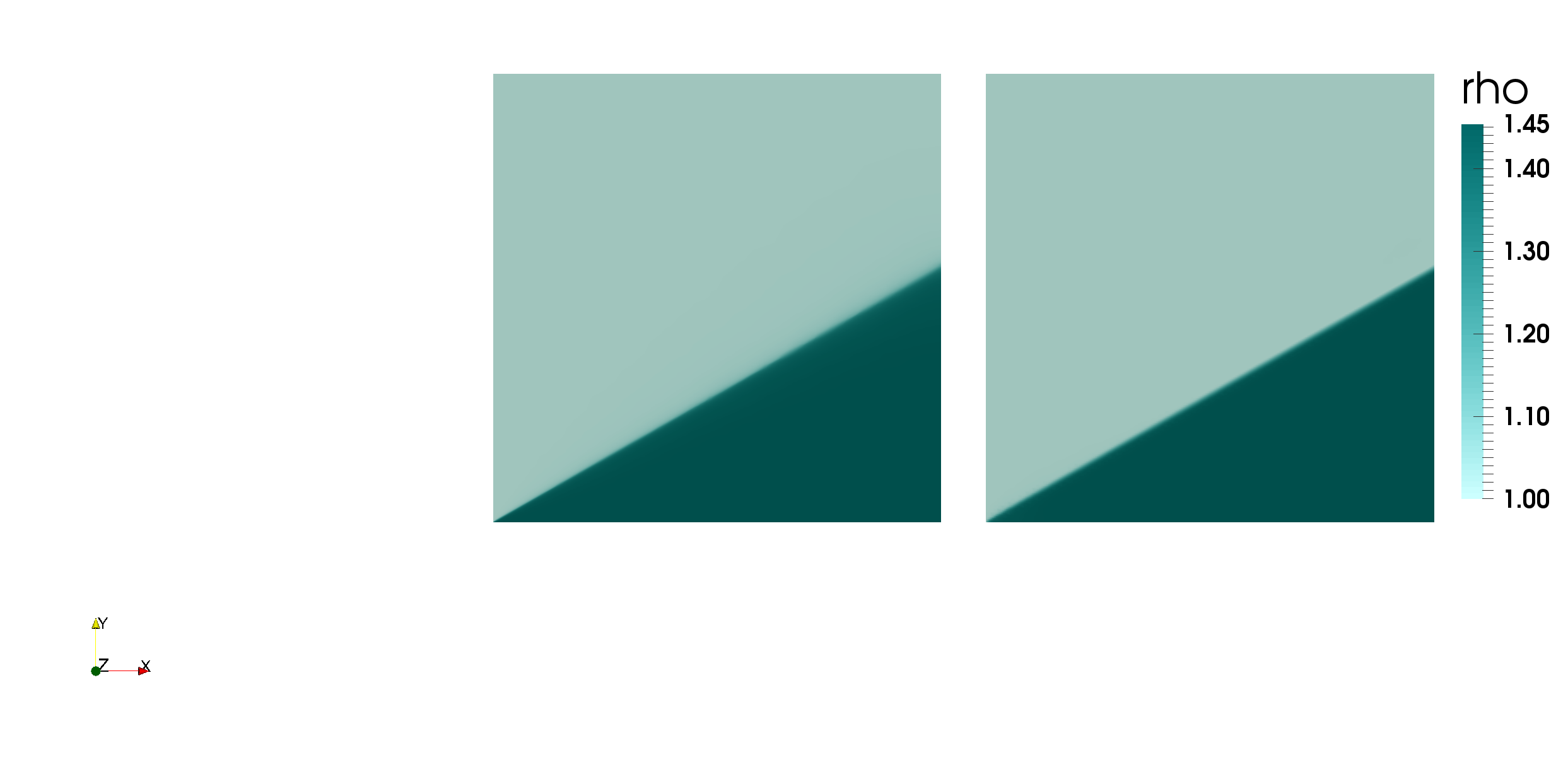}
	\end{subfigure}
	\caption{Evolution of the mesh refinement process. $\lestimator_K$ with high--order (right) and low--order (left) schemes are used. For the low--order scheme from top to bottom results have been obtained at refinement step 1, 2, 3, 8, and 8. For the high--order refinement steps are 1, 2, 3, 4, and 4.}
	\label{fig.compression-amr}
\end{figure}

Fig.\ \ref{fig.corner-conv} compares the effectiveness of the low-order and the high-order stabilization schemes for different values of $q$. The high-order scheme is able to converge efficiently and the overhead of solving a nonlinear problem does not affect the overall performance. In this case, the low-order and the high-order schemes require similar computational time for any given error. Actually, for the finer meshes, the high-order scheme with either $q=1$ or $q=2$ already performs better than the low-order scheme. However, for some meshes the high-order scheme exhibits convergence problems. In the case of $q=10$ the cost of converging the nonlinear problem does not compensate the increase in computational cost.

\begin{figure}[h]
\centering
\includegraphics[width=0.9\textwidth]{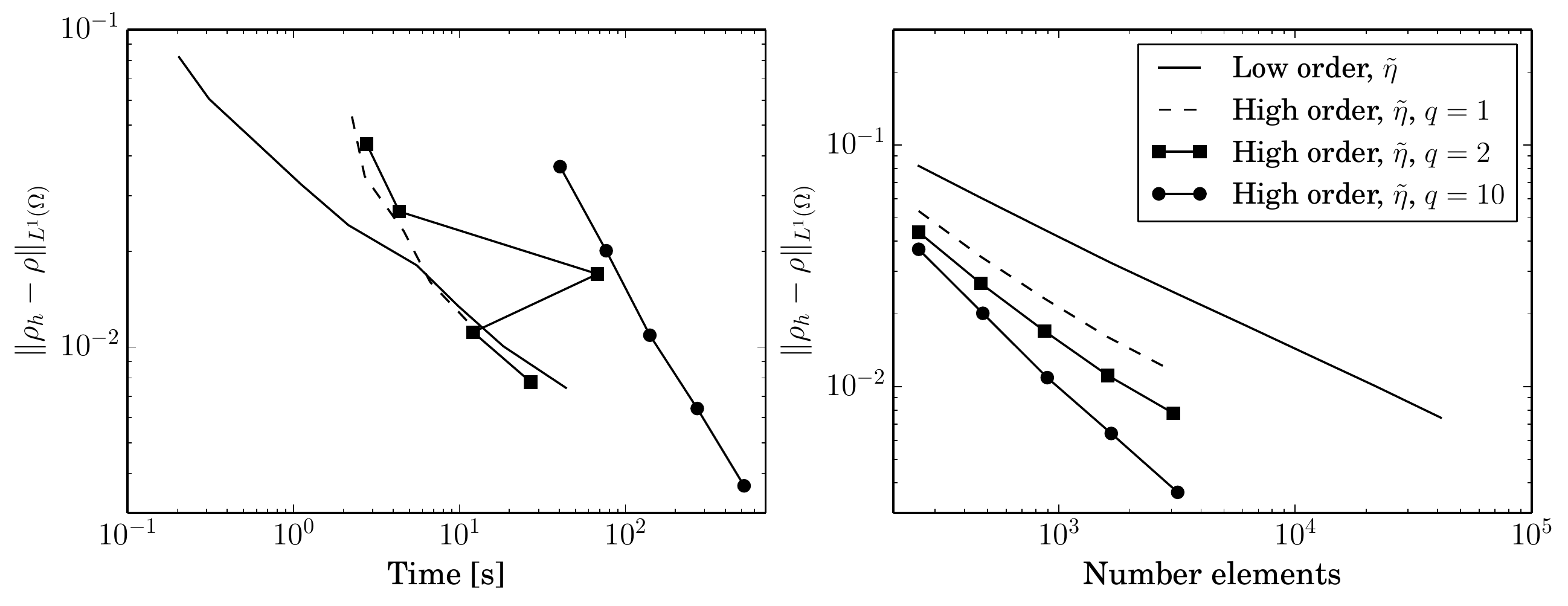}
\caption{Time and elements convergence comparison for the compression corner problem.}
\label{fig.corner-conv}
\end{figure}

\subsection{Reflected shock}

This benchmark consists in two flow streams colliding at different angles. The domain has dimensions $[0.0,1.0]\times[0.0,4.1]$ and a solid wall at its lower boundary. This configuration leads to a steady shock separating both flow regimes that is reflected at the wall producing a third different flow state behind it. A sketch of this benchmark test is given in Fig.\ \ref{fig.reflected-scheme}. The flow states at each region have been collected in Tab.\ \ref{tab.reflected-shock}.

\begin{table}[h]
	\centering
	\caption{Reflected shock solution values at every region.}
	\label{tab.reflected-shock}
	\begin{tabular}{cccc} \hline
		Region & Density [Kg\,m$^{-3}$] & Velocity [m\,s$^{-1}$] & Total energy [J] \\ \hline
		\textcircled{a} & 1.0 & (2.9, 0.0) & 5.99075 \\
		\textcircled{b} & 1.7 & (2.62, -0.506)& 5.8046 \\
		\textcircled{c} & 2.687 & (2.401, 0.0) & 5.6122 \\ \hline
	\end{tabular}
\end{table}

\begin{figure}[!h]
	\centering
	\includegraphics[width=0.6\textwidth]{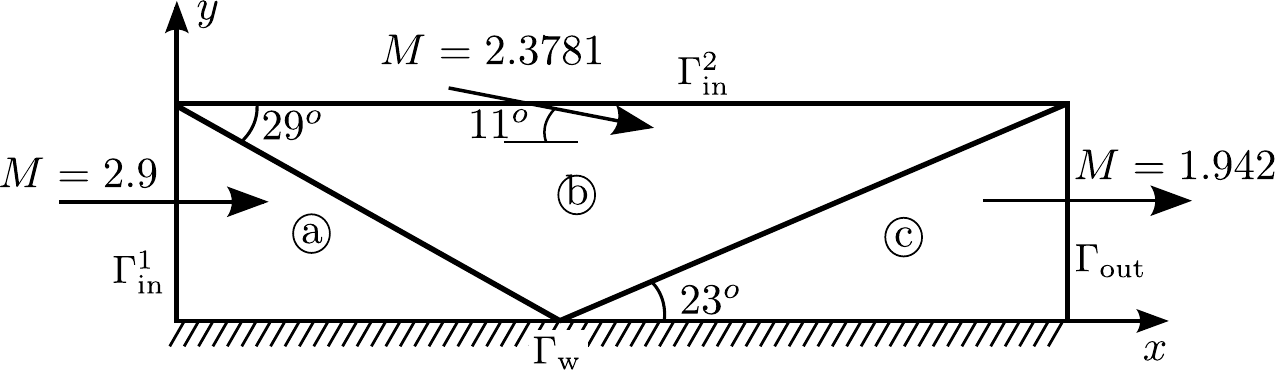}
	\caption{Reflected shock scheme.}
	\label{fig.reflected-scheme}
\end{figure}

We analyze the effectiveness of the high-order scheme, and evaluate the performance of the graph Laplacian \rev{refinement strategy}. We start with a coarse mesh of $16\times 64$ elements and adapt the mesh till a certain number of elements is reached. For the high-order method, we set a maximum of $10^4$ elements. The maximum number of elements for the low-order method is $3\cdot10^5$. We use a nonlinear tolerance of $\|\Delta \bunk\|/\|\bunk\|<10^{-4}$ and a maximum of 500 iterations. 

Fig.\ \ref{fig.reflected-conv} compares the effectiveness of the low-order and the high-order stabilization schemes for different values of $q$. The high-order scheme converges efficiently and the overhead of solving a nonlinear problem does not affect the overall performance. Actually, for the most refined meshes the high-order method is more efficient than the low-order one. As for the previous problem, Fig.\ \ref{fig.reflected-conv} shows that the high-order scheme can present nonlinear convergence problems at some steps of the refinement process. However, as the mesh becomes more adapted to the problem this issue is reduced. 

\begin{figure}[h]
	\centering
	\includegraphics[width=0.9\textwidth]{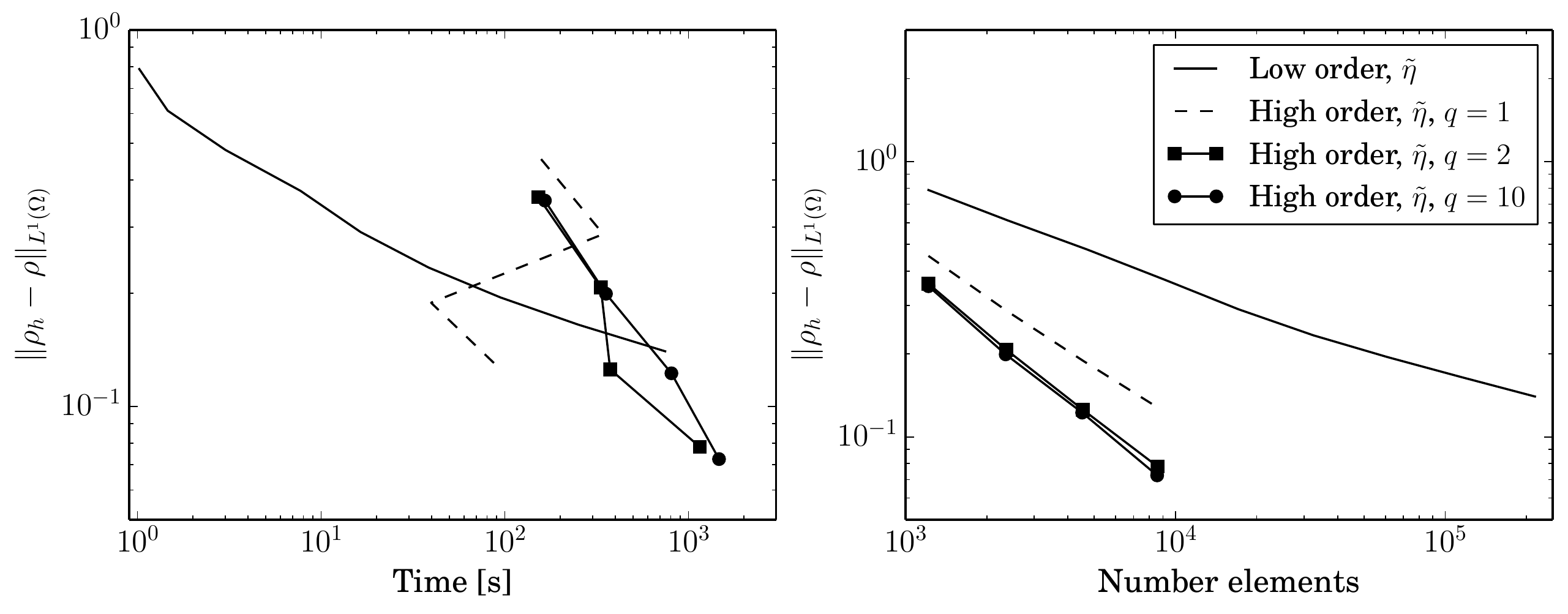}
	\caption{Time and elements convergence comparison for the reflected shock problem.}
	\label{fig.reflected-conv}
\end{figure}

\begin{figure}[!]
	\centering
	\setlength{\figwidth}{0.58\textwidth}
	\begin{subfigure}[t]{0.838\figwidth}
		\includegraphics[width=\textwidth,trim={200mm 170mm 120mm 100mm},clip]{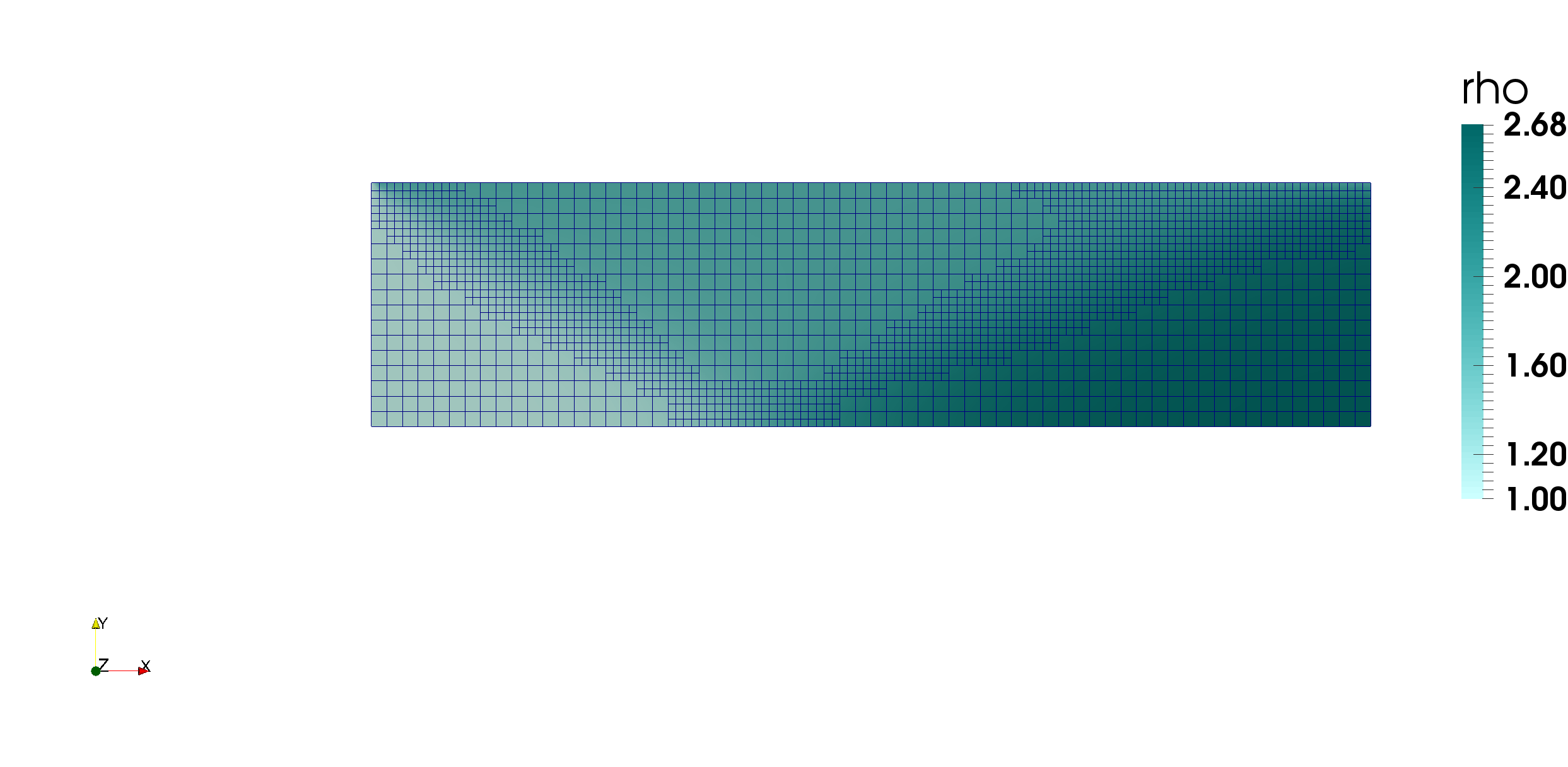}
	\end{subfigure}\hspace*{0.162\figwidth}
	
	\begin{subfigure}[t]{0.838\figwidth}
		\includegraphics[width=\textwidth,trim={200mm 170mm 120mm 100mm},clip]{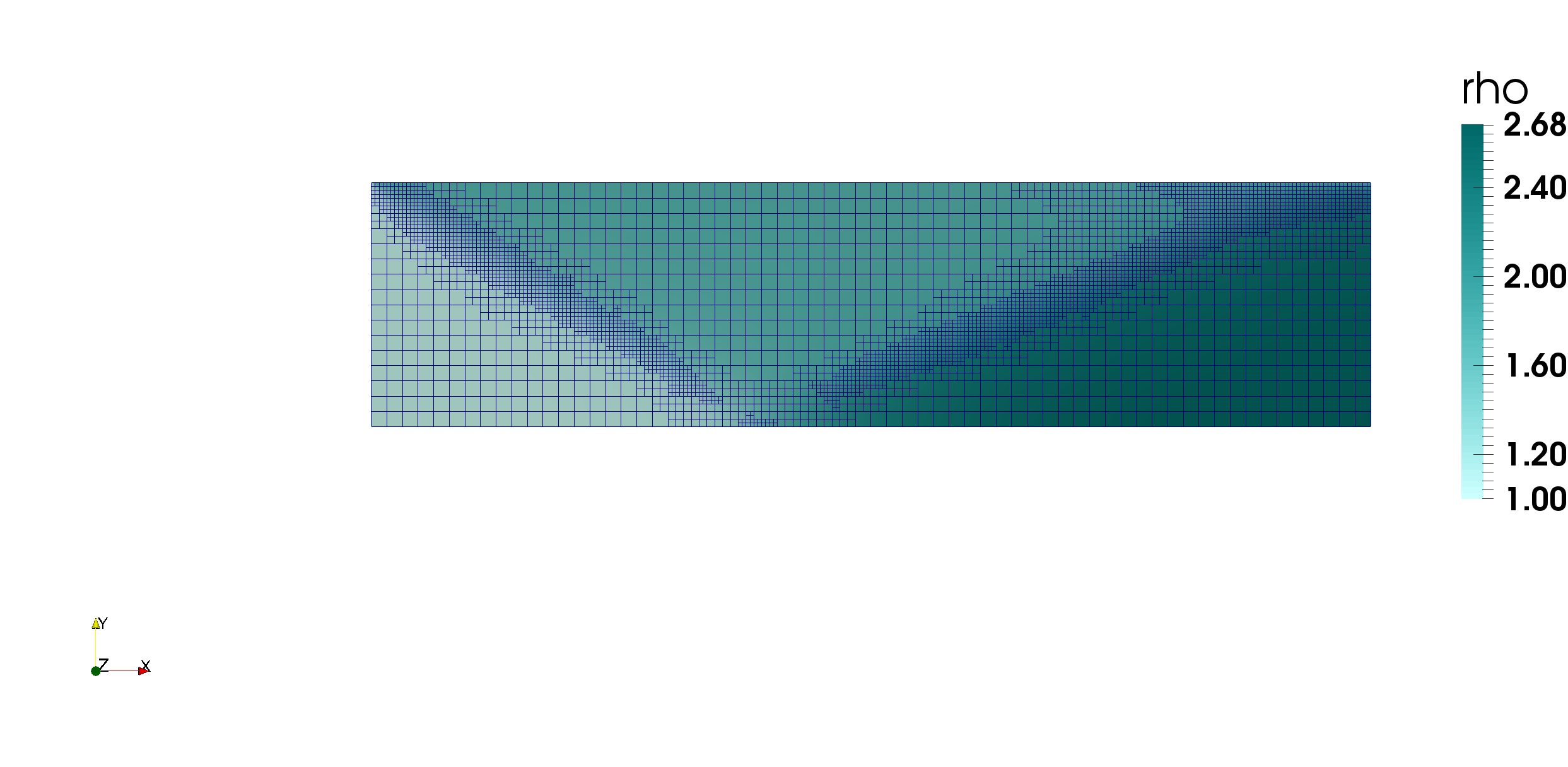}
	\end{subfigure}\hspace*{0.162\figwidth}
	
	\begin{subfigure}[t]{0.838\figwidth}
		\includegraphics[width=\textwidth,trim={200mm 170mm 120mm 100mm},clip]{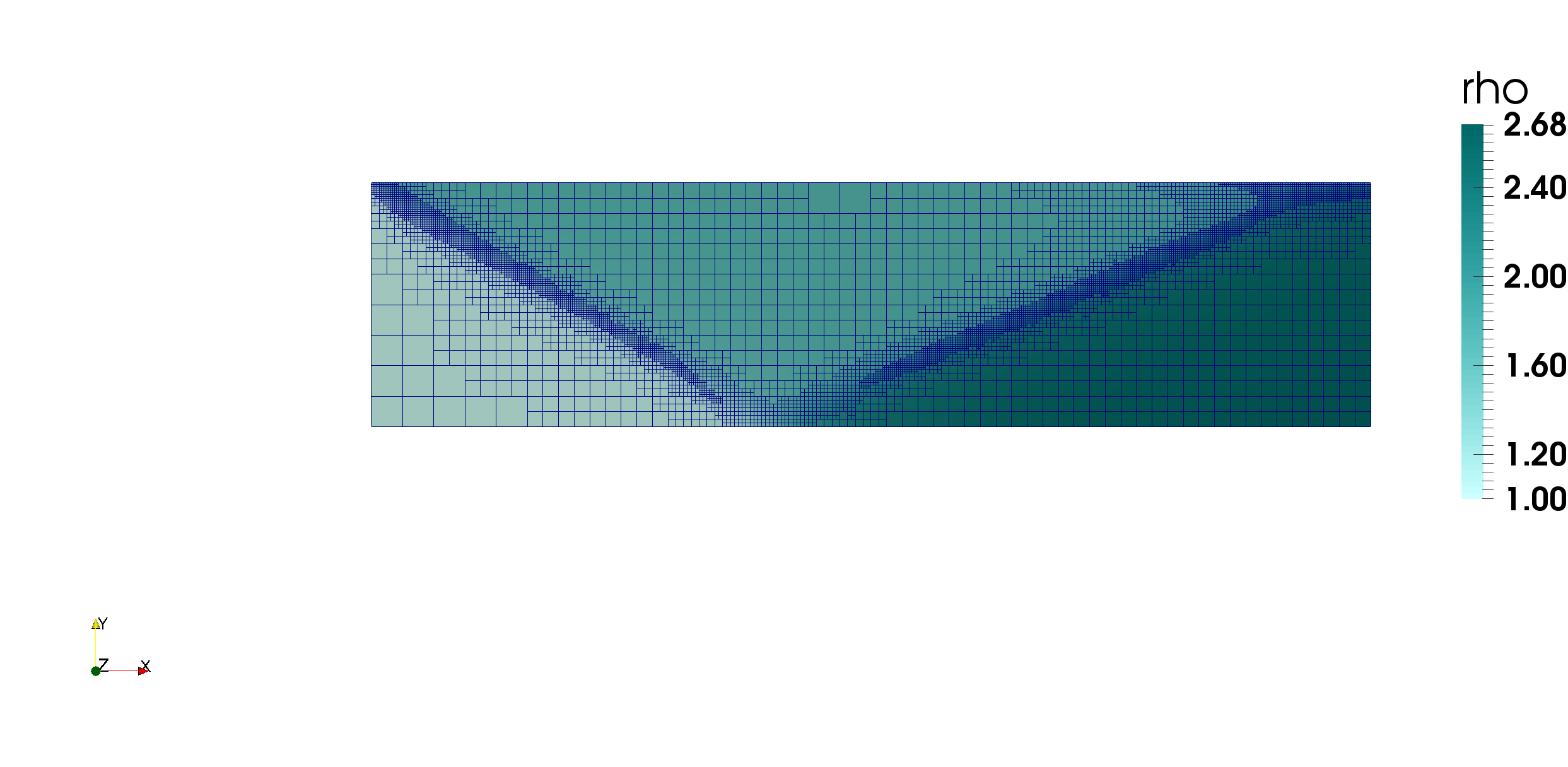}
	\end{subfigure}\hspace*{0.162\figwidth}
	
	\begin{subfigure}[t]{0.838\figwidth}
		\includegraphics[width=\textwidth,trim={200mm 170mm 120mm 100mm},clip]{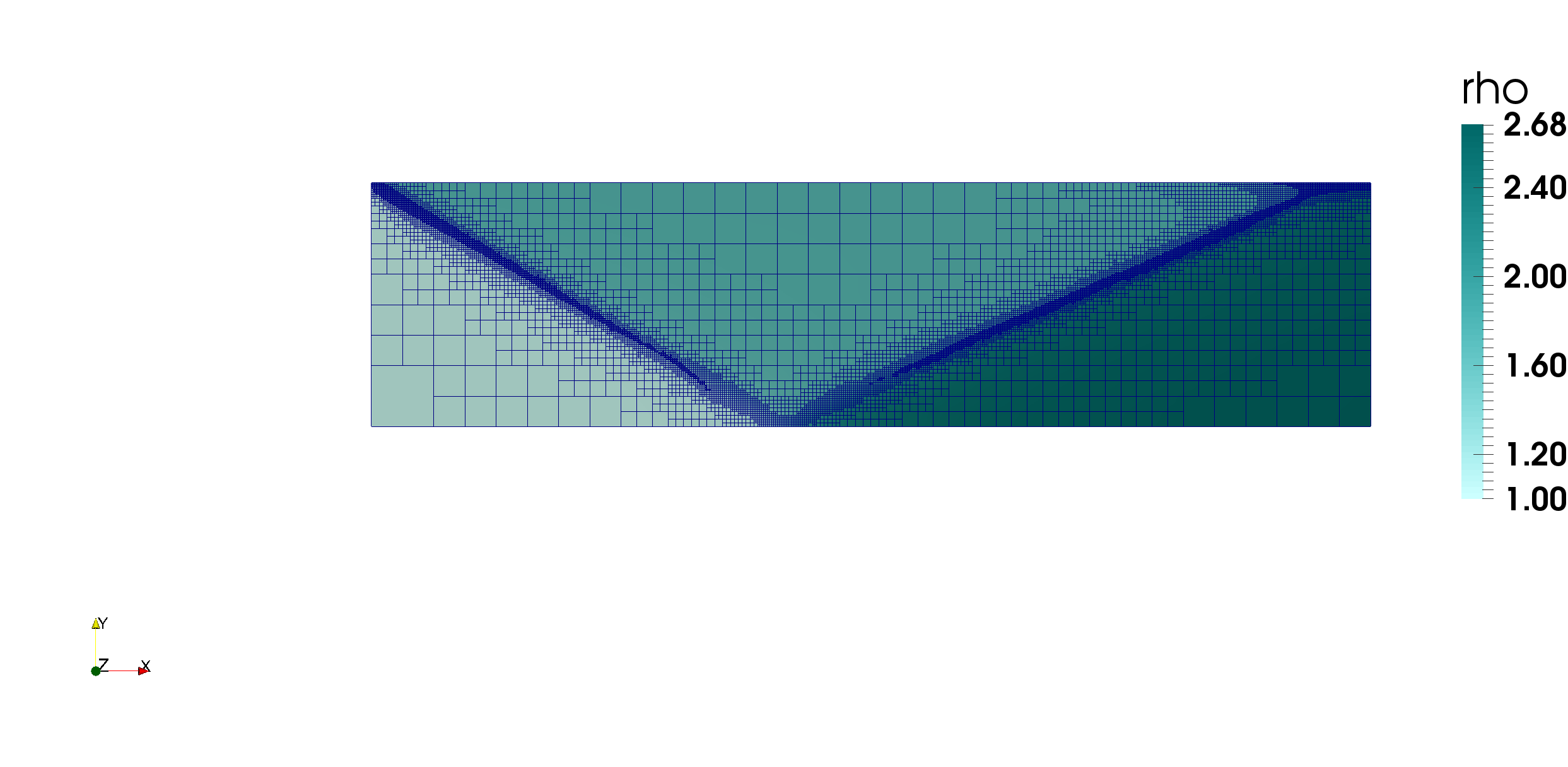}
	\end{subfigure}\hspace*{0.162\figwidth}
	
	\begin{subfigure}[t]{0.838\figwidth}
		\includegraphics[width=\textwidth,trim={200mm 170mm 120mm 100mm},clip]{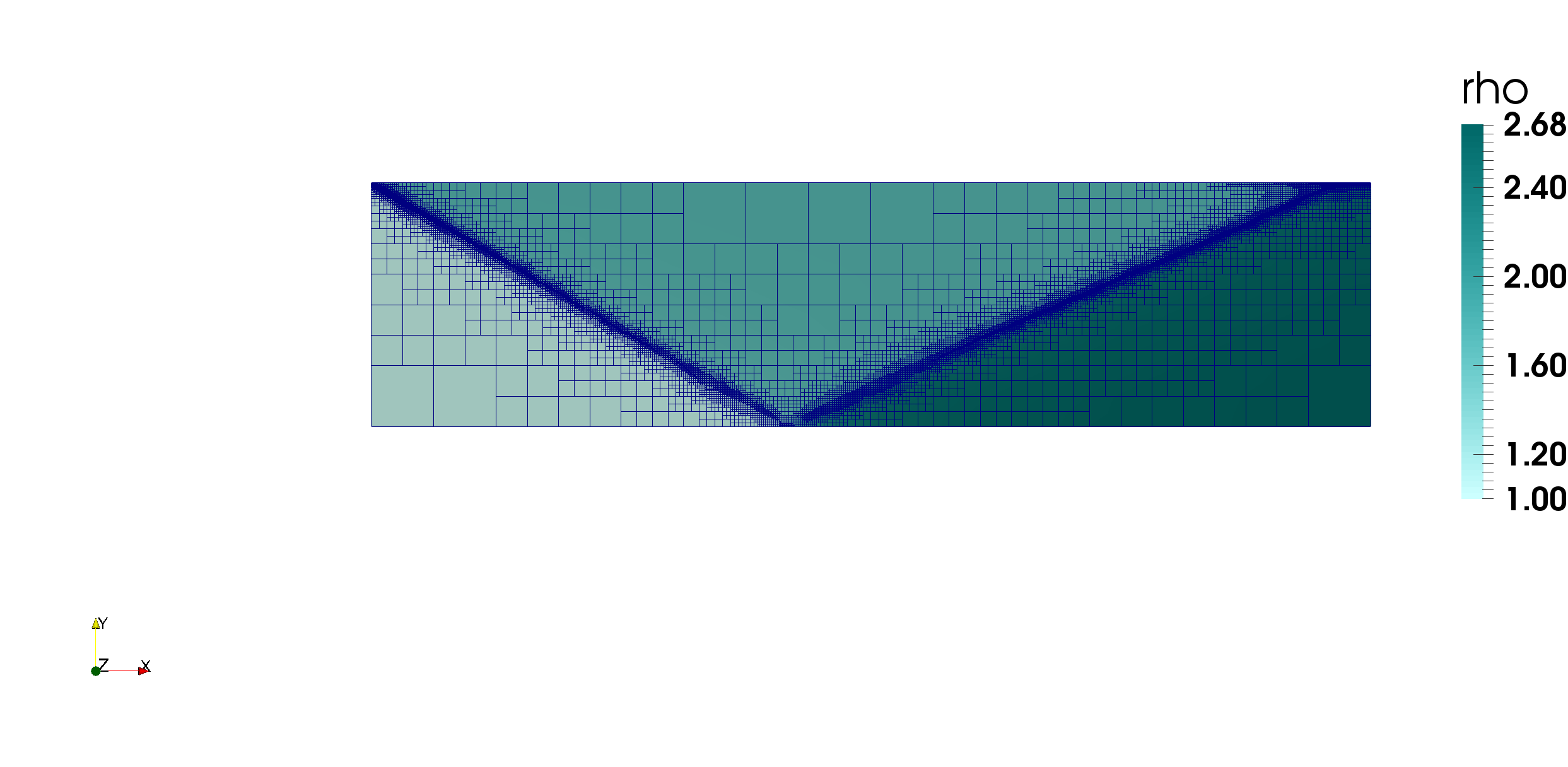}
	\end{subfigure}\hspace*{0.162\figwidth}
	
	\begin{subfigure}[t]{0.838\figwidth}
		\includegraphics[width=\textwidth,trim={200mm 170mm 120mm 100mm},clip]{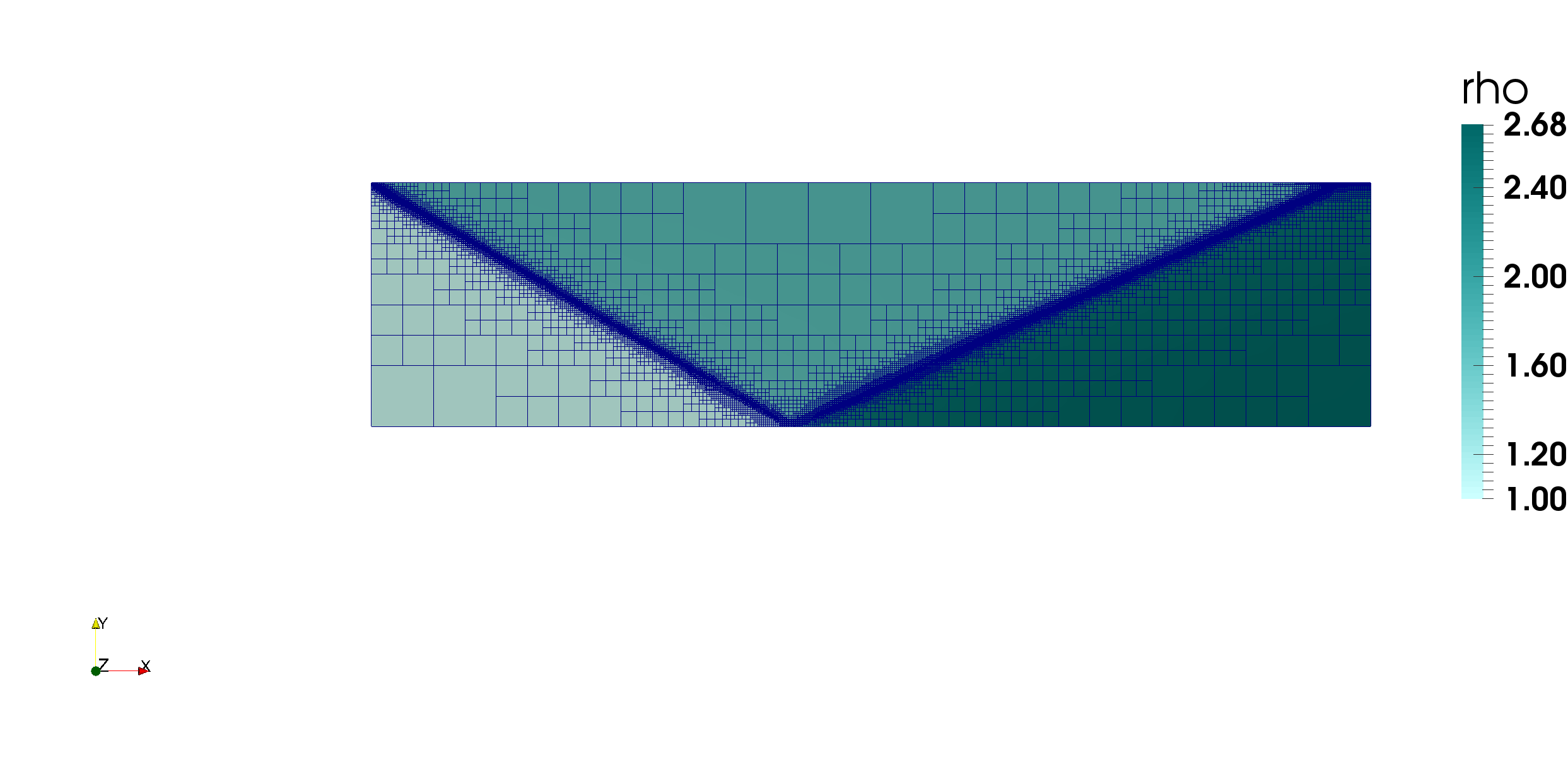}
	\end{subfigure}\hspace*{0.162\figwidth}
	
	\begin{subfigure}[t]{0.838\figwidth}
		\includegraphics[width=\textwidth,trim={200mm 170mm 120mm 100mm},clip]{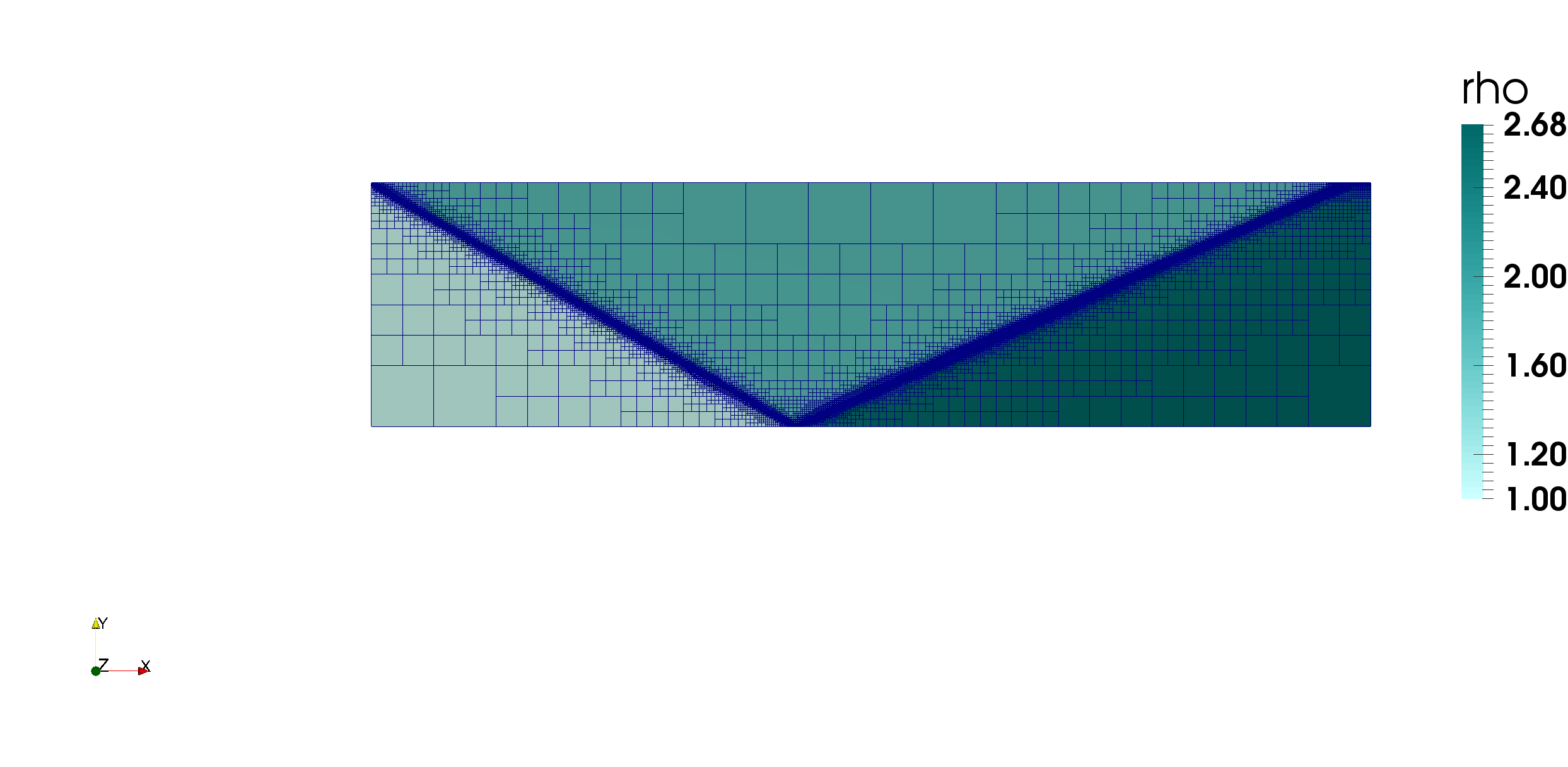}
	\end{subfigure}\hspace*{0.162\figwidth}
	
	\begin{subfigure}[t]{\figwidth}
		\includegraphics[width=\textwidth,trim={200mm 110mm 0 25mm},clip]{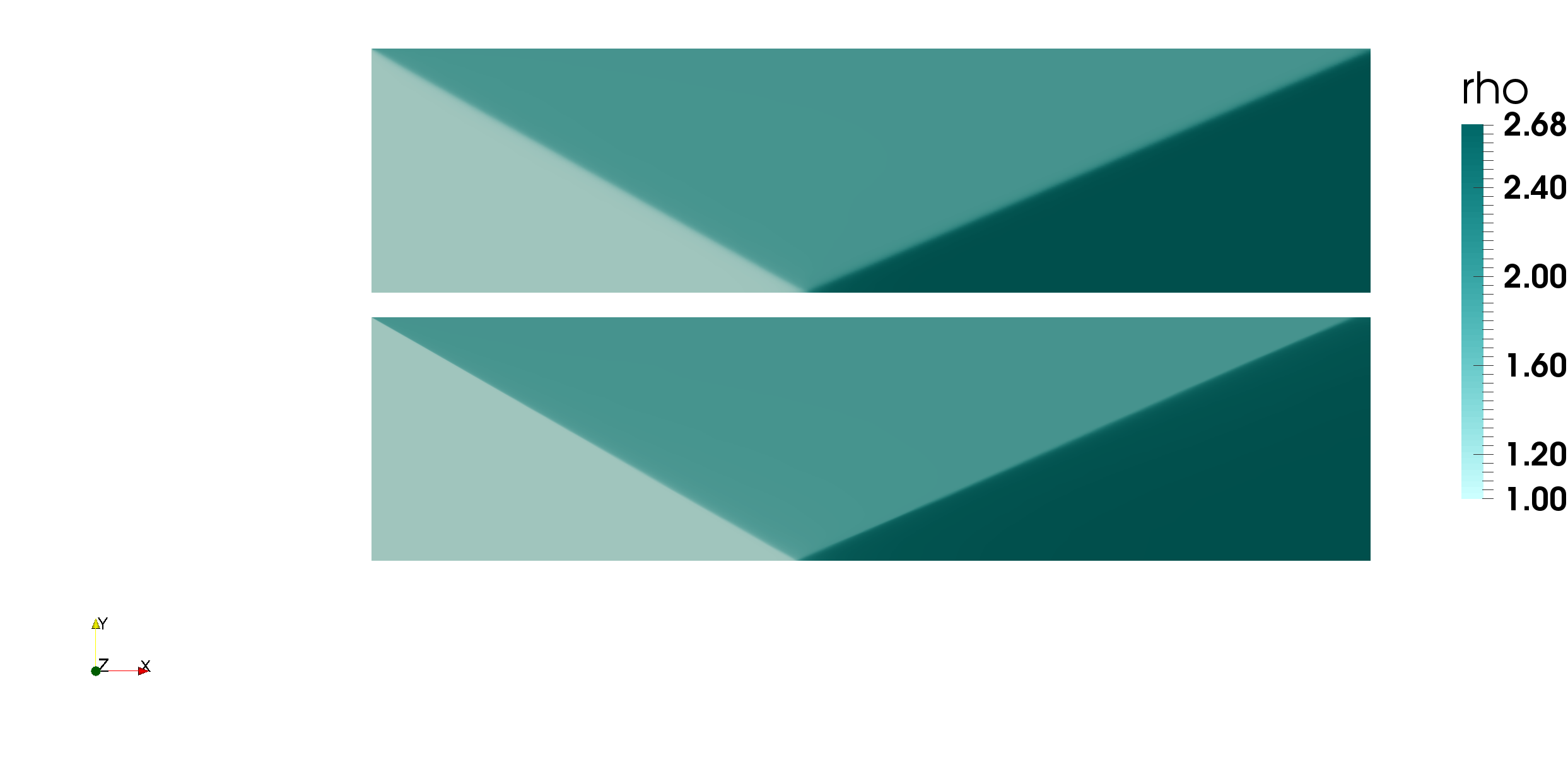}
	\end{subfigure}
	\caption{Evolution of the mesh refinement process. $\lestimator_K$ with low--order scheme is used. For the low--order scheme from top to bottom results have been obtained at refinement step 1, 2, 3, 4, 5, 6, and 7. The lower two figures are the high-order (top) and low-order (bottom) results at their last refinement step.}
	\label{fig.reflected-amr}	
\end{figure}

In Fig.\ \ref{fig.reflected-amr} we depict the refinement evolution for the graph Laplacian \rev{indicator} ($\lestimator_{\element}$) for the low-order scheme. In these figures it can be observed how the graph Laplacian \rev{strategy} is able to concentrate all the resolution at the shock location. Finally, we can conclude from the lower two figures that both schemes resolve the shocks properly after the mesh has been refined enough.

\section{Conclusions}\label{sec.conclusions}

The stabilization schemes in \cite{badia_monotonicity-preserving_2017,Badia2019c} have been extended and assessed in the \ac{amr} context for nonconforming hierarchical octree meshes. The work focuses in assessing the effectiveness of linear (first-order) and nonlinear (higher-order) stabilization. We focus the comparison in terms of accuracy versus computational time.

The results indicate that linear stabilization is more effective for coarse meshes. In this case, the computational cost required to solve the stiff nonlinear problem due to the nonlinear stabilization does not compensate the improvement in the accuracy. This is especially evident for linear systems of \acp{pde}. On the contrary, as the mesh is refined and properly adapted to the shocks, nonlinear stabilization pays the price. Even though increasing the value of $q$ in the nonlinear stabilization (a parameter that makes shocks sharper but hinders nonlinear convergence) improves accuracy, it turns to be more effective to refine the mesh further for low values of $q$. Nevertheless, it is worth mentioning that high-order method might exhibit nonlinear convergence problems for some meshes.

In addition, a new \rev{refinement strategy} have been proposed. The proposed \rev{indicator} is based on the graph Laplacian used in the definition of the stabilization method. Numerical results show that this shock detector is able to perform better that the well known Kelly estimator for problems with shocks or discontinuities.

\section*{Acknowledgments}
J. Bonilla gratefully acknowledges the support received from ''la Caixa'' Foundation through
its PhD scholarship program (LCF/BQ/DE15/10360010). S. Badia gratefully acknowledges the support received 
from the Catalan Government through the ICREA Acad\`emia Research Program. We acknowledge the financial support to CIMNE via the CERCA Programme  / Generalitat de Catalunya.

\bibliographystyle{siam}
\bibliography{references}

\begin{thebibliography}{10}

\bibitem{Ainsworth1997}
{\sc M.~Ainsworth and J.~{Tinsley Oden}}, {\em {A posteriori error estimation
  in finite element analysis}}, Computer Methods in Applied Mechanics and
  Engineering, 142 (1997), pp.~1--88.

\bibitem{AndersonJr.1990}
{\sc J.~D. {Anderson Jr.}}, {\em {Modern Compressible Flow}}, McGraw-Hill,
  2nd~ed., 1990.

\bibitem{badia_monotonicity-preserving_2017}
{\sc S.~Badia and J.~Bonilla}, {\em {Monotonicity-preserving finite element
  schemes based on differentiable nonlinear stabilization}}, Computer Methods
  in Applied Mechanics and Engineering, 313 (2017), pp.~133--158.

\bibitem{badia_differentiable_2017}
{\sc S.~Badia, J.~Bonilla, and A.~Hierro}, {\em {Differentiable
  monotonicity-preserving schemes for discontinuous Galerkin methods on
  arbitrary meshes}}, Computer Methods in Applied Mechanics and Engineering,
  320 (2017), pp.~582--605.

\bibitem{Badia2019c}
{\sc S.~Badia, J.~Bonilla, S.~Mabuza, and J.~N. Shadid}, {\em {Differentiable
  local bounds preserving stabilization for first order hyperbolic problems}},
  Submitted,  (2019).

\bibitem{badia_monotonicity-preserving_2014}
{\sc S.~Badia and A.~Hierro}, {\em {On Monotonicity-Preserving Stabilized
  Finite Element Approximations of Transport Problems}}, SIAM Journal on
  Scientific Computing, 36 (2014), pp.~A2673--A2697.

\bibitem{badia_discrete_2015}
{\sc S.~Badia and A.~Hierro}, {\em {On discrete maximum principles for
  discontinuous Galerkin methods}}, Computer Methods in Applied Mechanics and
  Engineering, 286 (2015), pp.~107--122.

\bibitem{Badia2019}
{\sc S.~Badia and A.~F. Mart{\'{i}}n}, {\em {A tutorial-driven introduction to
  the parallel finite element library FEMPAR v1.0.0}},  (2019).

\bibitem{Badia2019d}
{\sc S.~Badia, A.~F. Mart{\'{i}}n, E.~Neiva, and F.~Verdugo}, {\em {A generic
  finite element framework on parallel tree-based adaptive meshes}}, Submitted,
   (2019).

\bibitem{badia_fempar:_2017}
{\sc S.~Badia, A.~F. Mart{\'{i}}n, and J.~Principe}, {\em {FEMPAR: An
  Object-Oriented Parallel Finite Element Framework}}, Archives of
  Computational Methods in Engineering, 25 (2018), pp.~195--271.

\bibitem{bangerth_algorithms_2012}
{\sc W.~Bangerth, C.~Burstedde, T.~Heister, and M.~Kronbichler}, {\em
  {Algorithms and data structures for massively parallel generic adaptive
  finite element codes}}, ACM Trans. Math. Softw., 38 (2012), pp.~14:1--14:28.

\bibitem{barrenechea_edge-based_2016-1}
{\sc G.~R. Barrenechea, E.~Burman, and F.~Karakatsani}, {\em {Edge-based
  nonlinear diffusion for finite element approximations of convection-diffusion
  equations and its relation to algebraic flux-correction schemes}}, Numerische
  Mathematik,  (2016), pp.~1--25.

\bibitem{Barrenechea2016}
{\sc G.~R. Barrenechea, V.~John, and P.~Knobloch}, {\em {Analysis of Algebraic
  Flux Correction Schemes}}, SIAM Journal on Numerical Analysis, 54 (2016),
  pp.~2427--2451.

\bibitem{Bittl2013}
{\sc M.~Bittl and D.~Kuzmin}, {\em {An hp-adaptive flux-corrected transport
  algorithm for continuous finite elements}}, Computing, 95 (2013), pp.~27--48.

\bibitem{Bonilla2019a}
{\sc J.~Bonilla and S.~Badia}, {\em {Maximum-principle preserving space-time
  isogeometric analysis}}, Computer Methods in Applied Mechanics and
  Engineering, 354 (2019), pp.~422--440.

\bibitem{Bonilla2018a}
{\sc J.~Bonilla, S.~Mabuza, J.~N. Shadid, and S.~Badia}, {\em {On
  Differentiable Linearity and Local Bounds Preserving Stabilization Methods
  for First Order Conservation Law Systems}}, in Center for Computing Research
  Summer Proceedings 2018, A.~Cangi and M.~L. Parks, eds., Sandia National
  Laboratories, 2018, pp.~107--119.

\bibitem{Burman2000}
{\sc E.~Burman}, {\em {Adaptive finite element methods for compressible flow}},
  Computer Methods in Applied Mechanics and Engineering, 190 (2000),
  pp.~1137--1162.

\bibitem{burman_nonlinear_2002}
{\sc E.~Burman and A.~Ern}, {\em {Nonlinear diffusion and discrete maximum
  principle for stabilized Galerkin approximations of the
  convection-diffusion-reaction equation}}, Computer Methods in Applied
  Mechanics and Engineering, 191 (2002), pp.~3833--3855.

\bibitem{cockburn_rungekutta_2001}
{\sc B.~Cockburn and C.-W. Shu}, {\em {Runge-Kutta Discontinuous Galerkin
  Methods for Convection-Dominated Problems}}, Journal of Scientific Computing,
  16 (2001), pp.~173--261.

\bibitem{codina_discontinuity-capturing_1993}
{\sc R.~Codina}, {\em {A discontinuity-capturing crosswind-dissipation for the
  finite element solution of the convection-diffusion equation}}, Computer
  Methods in Applied Mechanics and Engineering, 110 (1993), pp.~325--342.

\bibitem{demkowicz_computing_2006}
{\sc L.~Demkowicz}, {\em {Computing with hp-ADAPTIVE FINITE ELEMENTS: One and
  Two Dimensional Elliptic and Maxwell Problems}}, vol.~1, CRC Press, oct 2006.

\bibitem{Eriksson1995}
{\sc K.~Eriksson, D.~Estep, P.~Hansbo, and C.~Johnson}, {\em {Introduction to
  Adaptive Methods for Differential Equations}}, Acta Numerica, 4 (1995),
  pp.~105--158.

\bibitem{Feistauer2003}
{\sc M.~M. Feistauer, J.~J. Felcman, and I.~I. Stra{\v{s}}kraba}, {\em
  {Mathematical and computational methods for compressible flow}}, Oxford
  University Press, 2003.

\bibitem{Fletcher1983}
{\sc C.~Fletcher}, {\em {The group finite element formulation}}, Computer
  Methods in Applied Mechanics and Engineering, 37 (1983), pp.~225--244.

\bibitem{Godunov1959}
{\sc S.~Godunov}, {\em {Finite difference method for numerical computation of
  discontinuous solutions of the equations of fluid dynamics}}, Matematicheskii
  Sbornik, Steklov Mathematical Institute of Russian Academy of Sciences,
  47(89) (1959), pp.~271--306.

\bibitem{Gottlieb2001}
{\sc S.~Gottlieb, C.-W. Shu, and E.~Tadmor}, {\em {Strong Stability-Preserving
  High-Order Time Discretization Methods}}, SIAM Review, 43 (2001),
  pp.~89--112.

\bibitem{guermond_second-order_2014}
{\sc J.-L. Guermond, M.~Nazarov, B.~Popov, and Y.~Yang}, {\em {A Second-Order
  Maximum Principle Preserving Lagrange Finite Element Technique for Nonlinear
  Scalar Conservation Equations}}, SIAM Journal on Numerical Analysis, 52
  (2014), pp.~2163--2182.

\bibitem{Guermond2015}
{\sc J.-L. Guermond and B.~Popov}, {\em {Invariant domains and first-order
  continuous finite element approximation for hyperbolic systems}},  (2015),
  pp.~1--22.

\bibitem{Gurris2009}
{\sc M.~Gurris}, {\em {Implicit finite element schemes for compressible gas and
  particle-laden gas flows}}, PhD thesis, Technische Universit{\"{a}}t
  Dortmund, 2009.

\bibitem{Johnson1995}
{\sc C.~Johnson and A.~Szepessy}, {\em {Adaptive finite element methods for
  conservation laws based on a posteriori error estimates}}, Communications on
  Pure and Applied Mathematics, 48 (1995), pp.~199--234.

\bibitem{Kelly1983}
{\sc D.~W. Kelly, J.~P. {De S. R. Gago}, O.~C. Zienkiewicz, and I.~Babuska},
  {\em {A posteriori error analysis and adaptive processes in the finite
  element method: Part I-error analysis}}, International Journal for Numerical
  Methods in Engineering, 19 (1983), pp.~1593--1619.

\bibitem{kritz_fusion_2009}
{\sc A.~Kritz and D.~Keyes}, {\em {Fusion Simulation Project Workshop Report}},
  Journal of Fusion Energy, 28 (2009), pp.~1--59.

\bibitem{kuzmin_linearity-preserving_2012}
{\sc D.~Kuzmin}, {\em {Linearity-preserving flux correction and convergence
  acceleration for constrained Galerkin schemes}}, Journal of Computational and
  Applied Mathematics, 236 (2012), pp.~2317--2337.

\bibitem{Kuzmin2020}
\leavevmode\vrule height 2pt depth -1.6pt width 23pt, {\em {Monolithic convex
  limiting for continuous finite element discretizations of hyperbolic
  conservation laws}}, Computer Methods in Applied Mechanics and Engineering,
  361 (2020), p.~112804.

\bibitem{Kuzmin2017}
{\sc D.~Kuzmin, S.~Basting, and J.~N. Shadid}, {\em {Linearity-preserving
  monotone local projection stabilization schemes for continuous finite
  elements}}, Computer Methods in Applied Mechanics and Engineering, 322
  (2017), pp.~23--41.

\bibitem{Kuzmin2005}
{\sc D.~Kuzmin, R.~L{\"{o}}hner, and S.~Turek}, {\em {Flux-corrected
  transport}}, Springer, 2005.

\bibitem{kuzmin_algebraic_2005}
{\sc D.~Kuzmin and M.~M{\"{o}}ller}, {\em {Algebraic Flux Correction I. Scalar
  Conservation Laws}}, in Flux-Corrected Transport, D.~D. Kuzmin, P.~R.
  L{\"{o}}hner, and P.~D.~S. Turek, eds., Scientific Computation, Springer
  Berlin Heidelberg, jan 2005, pp.~155--206.

\bibitem{Kuzmin2012}
{\sc D.~Kuzmin, M.~M{\"{o}}ller, and M.~Gurris}, {\em {Algebraic Flux
  Correction II. Compressible flows}}, in Flux-corrected Transport: Principles,
  Algorithms, and Applications, 2012, pp.~193--238.

\bibitem{Kuzmin2003}
{\sc D.~Kuzmin, M.~M{\"{o}}ller, and S.~Turek}, {\em {Multidimensional FEM-FCT
  schemes for arbitrary time stepping}}, International Journal for Numerical
  Methods in Fluids, 42 (2003), pp.~265--295.

\bibitem{Kuzmin2018}
{\sc D.~Kuzmin, M.~{Quezada De Luna}, C.~E. Kees, D.~Kuzmin, M.~{Quezada De
  Luna}, and C.~E. Kees}, {\em {A partition of unity approach to adaptivity and
  limiting in continuous finite element methods}},  (2018).

\bibitem{kuzmin_flux_2002}
{\sc D.~Kuzmin and S.~Turek}, {\em {Flux Correction Tools for Finite
  Elements}}, Journal of Computational Physics, 175 (2002), pp.~525--558.

\bibitem{Leveque2002}
{\sc R.~J. LeVeque}, {\em {Finite Volume Methods for Hyperbolic Problems}},
  Cambridge University Press, Cambridge, 2002.

\bibitem{Lohmann2016}
{\sc C.~Lohmann and D.~Kuzmin}, {\em {Synchronized flux limiting for gas
  dynamics variables}}, Journal of Computational Physics, 326 (2016),
  pp.~973--990.

\bibitem{Lohmann2017a}
{\sc C.~Lohmann, D.~Kuzmin, J.~N. Shadid, and S.~Mabuza}, {\em {Flux-corrected
  transport algorithms for continuous Galerkin methods based on high order
  Bernstein finite elements}}, Journal of Computational Physics, 344 (2017),
  pp.~151--186.

\bibitem{Lohner1987}
{\sc R.~L{\"{o}}hner}, {\em {An adaptive finite element scheme for transient
  problems in CFD}}, Computer Methods in Applied Mechanics and Engineering, 61
  (1987), pp.~323--338.

\bibitem{Lohner2004}
{\sc R.~Lohner}, {\em {Applied Computational Fluid Dynamics Techniques: An
  Introduction Based on Finite Element Methods}}, vol.~508, 2004.

\bibitem{Mabuza2019}
{\sc S.~Mabuza, J.~N. Shadid, E.~C. Cyr, R.~P. Pawlowski, and D.~Kuzmin}, {\em
  {A linearity preserving nodal variation limiting algorithm for continuous
  Galerkin discretization of ideal MHD equations}}, Journal of Computational
  Physics, In press (2020).

\bibitem{Mabuza2018}
{\sc S.~Mabuza, J.~N. Shadid, and D.~Kuzmin}, {\em {Local bounds preserving
  stabilization for continuous Galerkin discretization of hyperbolic systems}},
  Journal of Computational Physics, 361 (2018), pp.~82--110.

\bibitem{Moller2006}
{\sc M.~M{\"{o}}ller and D.~Kuzmin}, {\em {Adaptive mesh refinement for
  high-resolution finite element schemes}}, International Journal for Numerical
  Methods in Fluids, 52 (2006), pp.~545--569.

\bibitem{Nazarov2011}
{\sc M.~Nazarov, J.-L. Guermond, and B.~Popov}, {\em {A posteriori error
  estimation for the compressible Euler equations using entropy viscosity}},
  tech. rep., 2011.

\bibitem{Nazarov2010}
{\sc M.~Nazarov and J.~Hoffman}, {\em {An adaptive finite element method for
  inviscid compressible flow}}, International Journal for Numerical Methods in
  Fluids, 64 (2010), pp.~1102--1128.

\bibitem{shakib_new_1991}
{\sc F.~Shakib, T.~J.~R. Hughes, and Z.~Johan}, {\em {A new finite element
  formulation for computational fluid dynamics: X. The compressible Euler and
  Navier-Stokes equations}}, Computer Methods in Applied Mechanics and
  Engineering, 89 (1991), pp.~141--219.

\bibitem{Suli1999}
{\sc E.~S{\"{u}}li}, {\em {A Posteriori Error Analysis and Adaptivity for
  Finite Element Approximations of Hyperbolic Problems}},  (1999),
  pp.~123--194.

\bibitem{tezduyar_stabilization_2006}
{\sc T.~E. Tezduyar and M.~Senga}, {\em {Stabilization and shock-capturing
  parameters in SUPG formulation of compressible flows}}, Computer Methods in
  Applied Mechanics and Engineering, 195 (2006), pp.~1621--1632.

\bibitem{TiankaiTu2005}
{\sc T.~{Tiankai Tu}, D.~O'Hallaron, and O.~Ghattas}, {\em {Scalable Parallel
  Octree Meshing for TeraScale Applications}}, in ACM/IEEE SC 2005 Conference
  (SC'05), IEEE, 2005, pp.~4--4.

\bibitem{Tikhonova2005}
{\sc A.~Tikhonova, G.~Tanase, O.~Tkachyshyn, N.~M. Amato, and L.~Rauchwerger},
  {\em {Parallel Algorithms in STAPL: Sorting and the Selection Problem}},
  tech. rep., 2005.

\bibitem{Toro2009}
{\sc E.~F. Toro}, {\em {Riemann Solvers and Numerical Methods for Fluid
  Dynamics}}, 3rd~ed., 2009.

\bibitem{Verfurth2013}
{\sc R.~Verfurth}, {\em {A Posteriori Error Estimation Techniques for Finite
  Element Methods}}, Oxford University Press, 2013.

\bibitem{Zienkiewicz1987}
{\sc O.~C. Zienkiewicz and J.~Z. Zhu}, {\em {A simple error estimator and
  adaptive procedure for practical engineerng analysis}}, International Journal
  for Numerical Methods in Engineering, 24 (1987), pp.~337--357.

\bibitem{Zienkiewicz1992}
{\sc O.~C. Zienkiewicz and J.~Z. Zhu}, {\em {The superconvergent patch recovery
  anda posteriori error estimates. Part 1: The recovery technique}},
  International Journal for Numerical Methods in Engineering, 33 (1992),
  pp.~1331--1364.

\end{thebibliography}

\end{document}